\documentclass[twoside,11pt]{article}

\usepackage[utf8]{inputenc}

\usepackage{amsmath}
\usepackage{amsfonts}
\usepackage{caption}
\usepackage{subcaption}

\usepackage[norelsize,ruled,vlined,commentsnumbered]{algorithm2e}

\usepackage[shortlabels]{enumitem}

\usepackage{multirow}
\usepackage{hhline}

\usepackage{jmlr2e}

\newcommand{\abf}{\mathbf{a}}

\newcommand{\Bbf}{\mathbf{B}}

\newcommand{\Cbf}{\mathbf{C}}
\newcommand{\Ccal}{\mathcal{C}}

\newcommand{\Ebb}{\mathbb{E}}
\newcommand{\fbf}{\mathbf{f}}
\newcommand{\Fcal}{\mathcal{F}}

\newcommand{\hbf}{\mathbf{h}}
\newcommand{\Hcal}{\mathcal{H}}

\newcommand{\Lbf}{\mathbf{L}}

\newcommand{\Mbf}{\mathbf{M}}

\newcommand{\Mcal}{\mathcal{M}}
\newcommand{\Nbb}{\mathbb{N}}
\newcommand{\Nbf}{\mathbf{N}}

\newcommand{\Obf}{\mathbf{O}}
\newcommand{\Pbb}{\mathbb{P}}
\newcommand{\Pbf}{\mathbf{P}}

\newcommand{\Pfrak}{\mathfrak{P}}
\newcommand{\Qbf}{\mathbf{Q}}
\newcommand{\Qcal}{\mathcal{Q}}
\newcommand{\Rbb}{\mathbb{R}}
\newcommand{\Rbf}{\mathbf{R}}

\newcommand{\Sfrak}{\mathfrak{S}}

\newcommand{\Ubf}{\mathbf{U}}
\newcommand{\ubf}{\mathbf{u}}

\newcommand{\Vbf}{\mathbf{V}}
\newcommand{\Vcal}{\mathcal{V}}

\newcommand{\Xcal}{\mathcal{X}}
\newcommand{\Ybf}{\mathbf{Y}}
\newcommand{\Ycal}{\mathcal{Y}}

\newcommand{\argmin}{\text{arg min}}
\newcommand{\cquad}{C_{\Fcal,2}}
\newcommand{\cinf}{C_{\Fcal,\infty}}
\newcommand{\dperm}{{d_{\text{perm}}}}
\newcommand{\Diag}{\text{Diag}}

\newcommand{\new}{{\text{new}}}
\newcommand{\one}{\mathbf{1}}
\newcommand{\pen}{\text{pen}}
\newcommand{\Span}{{\text{Span}}}

\jmlrheading{18}{2017}{1-46}{6/21; Revised 9/17, 4/30}{7/13}{17-345}{Luc Lehéricy}

\ShortHeadings{State-by-state Minimax HMM Estimation}{Luc Lehéricy}
\firstpageno{1}

\begin{document}

\title{State-by-state Minimax Adaptive Estimation for Nonparametric Hidden Markov Models}

\author{\name Luc Lehéricy \email luc.lehericy@math.u-psud.fr \\
       \addr Laboratoire de Mathématiques d'Orsay \\
       Univ. Paris-Sud, CNRS, Université Paris-Saclay \\
       91405 Orsay, France}

\editor{Animashree Anandkumar}

\maketitle

\begin{abstract}%   <- trailing '%' for backward compatibility of .sty file
In this paper, we introduce a new estimator for the emission densities of a nonparametric hidden Markov model. It is adaptive and minimax with respect to each state's regularity\---as opposed to globally minimax estimators, which adapt to the worst regularity among the emission densities.
Our method is based on Goldenshluger and Lepski's methodology.
It is computationally efficient and only requires a family of preliminary estimators, without any restriction on the type of estimators considered. We present two such estimators that allow to reach minimax rates up to a logarithmic term: a spectral estimator and a least squares estimator.
We show how to calibrate it in practice and assess its performance on simulations and on real data.
\end{abstract}


%\tableofcontents

\section{Introduction}

Finite state space hidden Markov models, or HMMs in short, are powerful tools for studying discrete time series and have been used in a variety of applications such as economics, signal processing and image analysis, genomics, ecology, speech recognition and ecology among others. The core idea is that the behaviour of the observations depends on a hidden variable that evolves like a Markov chain.

Formally, a hidden Markov model is a process $(X_j, Y_j)_{j \geq 1}$ in which $(X_j)_j$ is a Markov chain on $\Xcal$, the $Y_i$'s are independent conditionally on $(X_j)_j$ and the conditional distribution of $Y_i$ given $(X_j)_j$ depends only on $X_i$. The parameters of the HMM are the parameters of the Markov chain, that is its initial distribution and transition matrix, and the parameters of the observations, that is the \emph{emission distributions} $(\nu^*_k)_{k \in \Xcal}$ where $\nu^*_k$ is the distribution of $Y_j$ conditionally to $X_j = k$. Only the observations $(Y_j)_j$ are available.

In this article, we focus on estimating the emission distributions in a nonparametric setting. More specifically, assume that the emission distributions have a density with respect to some known dominating measure $\mu$, and write $f^*_k$ their densities\---which we call the \emph{emission densities}. The goal of this paper is to estimate all $f^*_k$'s with their minimax rate of convergence when the emission densities are not restricted to belong to a set of densities described by finitely many parameters.

\subsection{Nonparametric state-by-state adaptivity}

Theoretical results in the nonparametric setting have only been developed recently. \cite{dCGLLC15} and \cite{robin2014estimating} introduce spectral methods, and the latter is proved to be minimax but not adaptive\---which means one needs to know the regularity of the densities beforehand to reach the minimax rate of convergence. \cite{dCGL15} introduce a least squares estimator which is shown to be minimax adaptive up to a logarithmic term. However, all these papers have a common drawback: they study the emission densities as a whole and can not handle them separately. This comes from their error criterion, which is the supremum of the errors on all densities: what they actually prove is that $\max_{k \in \Xcal} \| \hat{f}_k - f^*_k \|_2$ converges with minimax rate when $(\hat{f}_k)_k$ are their density estimators.
In general, the regularity of each emission density could be different, leading to different rates of convergence. This means that having just one emission density that is very hard to estimate is enough to deteriorate the rate of convergence of all emission densities.

In this paper, we construct an estimator that is adaptive and estimates each emission density with its own minimax rate of convergence. We call this property state-by-state adaptivity. Our method does so by handling each emission density individually in a way that is theoretically justified\---reaching minimax and adaptive rates of convergence with respect to the regularity of the emission densities\---and computationally efficient thanks to its low computational and sample complexity.

Our approach for estimating the densities nonparametrically is model selection. The core idea is to approximate the target density using a family of parametric models that is dense within the nonparametric class of densities. For a square integrable density $f^*$, we consider its projection $f^*_M$ on a finite-dimensional space $\Pfrak_M$ (the parametric model), where $M$ is a model index. This projection introduces an error, the \emph{bias}, which is the distance $\|f^* - f^*_M\|_2$ between the target quantity and the model. The larger the model, the smaller the bias. On the other hand, larger models will make the estimation harder, resulting in a larger \emph{variance} $\| \hat{f}_M - f^*_M \|_2^2$. The key step of model selection is to select a model with a small total error\---or alternatively, a good \emph{bias-variance tradeoff}.

In many situations, it is possible to reach the minimax rate of convergence with a good bias-variance tradeoff. Previous estimators of the emission densities of a HMM perform such a tradeoff based on an error that takes the transition matrix and all emission densities into account. Such an error leads to a rate of convergence that corresponds to the slowest minimax rate amongst the different parameters. In contrast, our method performs a bias-variance tradeoff for each emission density using an error term that depends only on the density in question, which makes it possible to reach the minimax rates for each density.

\subsection{Plug-in procedure}

The method we propose is based on the method developed in the seminal papers of \cite{goldenshluger2011bandwidth, goldenshluger2014adaptive} for density estimation, extended by \cite{goldenshluger2013general} to the white noise and regression models. It takes a family of estimators as input and chooses the estimator that performs a good bias-variance tradeoff separately for each hidden state. We recommend the article of \cite{lacour2016estimator} for an insightful presentation of this method in the case of conditional density estimation.

Our method and assumptions are detailed in Section \ref{sec_statebystatemethod}. Let us give a quick overview of the method. Assume the densities belong to a Hilbert space $\Hcal$. Given a family of subsets of finite-dimensional subspaces of $\Hcal$ (the models) indexed by $M$ and estimators $\hat{f}^{(M)}_k$ of the emission densities for each hidden state $k$ and each model $M$, one computes a substitute for the bias of the estimators by
\begin{equation*}
A_k(M) = \underset{M'}{\sup} \left\{ \left\| \hat{f}^{(M')}_k - \hat{f}^{(M \wedge M')}_k \right\|_2 - \sigma(M') \right\}.
\end{equation*}
for some penalty $\sigma$. Then, for each state $k$, one selects the estimator $\hat{M}_k$ from the model $M$ minimizing the quantity $A_k(M) + 2 \sigma(M)$. 
The penalty $\sigma$ can also be interpreted as a variance bound, so that this penalization procedure can be seen as performing a bias-variance tradeoff.
The novelty of this method is that it selects a different $\hat{M}_k$, that is a different model, for each hidden state: this is where the state-by-state adaptivity comes from. Also note that contrary to \cite{goldenshluger2013general}, we do not make any assumption on how the estimators are computed, provided a variance bound holds.

The main theoretical result is an oracle inequality on the selected estimators $\hat{f}^{(\hat{M}_k)}_k$, see Theorem \ref{th_oracleLepski}. As a consequence, we are able to get a rate of convergence that is different for each state. These rates of convergence will even be adaptive minimax up to a logarithmic factor when the method is applied to our two families of estimators: spectral estimators and least squares estimators. To the best of our knowledge, this is the first state-by-state adaptive algorithm for hidden Markov models.

Note that finding the right penalty term $\sigma$ is essential in order to obtain minimax rates of convergence. This requires a fine theoretical control of the variance of the auxiliary estimators, in the form of assumption \textbf{[H$(\epsilon)$]} (see Section \ref{sec_Lepski_framework}). To the best of our knowledge, there is no suitable result in the literature. This is the second theoretical contribution of this paper: we control two families of estimators in a way that makes it possible to reach adaptive minimax rate with our state-by-state selection method, up to a logarithmic term.

On the practical side, we run this method and several variants on data simulated from a HMM with three hidden states and one irregular density, as illustrated in Section \ref{sec_simulations}. The simulations confirm that it converges with a different rate for each emission density, and that the irregular density does not alter the rate of convergence of the other ones, which is exactly what we wanted to achieve.

Better still, the added computation time is negligible compared to the computation time of the estimators: even for the spectral estimators of Section \ref{sec_spectral_th} (which can be computed much faster than the least squares estimators and the maximum likelihood estimators using EM), computing the estimators on 200 models for 50,000 observations (the lower bound of our sample sizes) of a 3-states HMM requires a few minutes, compared to a couple of seconds for the state-by-state selection step. The difference becomes even larger for more observations, since the complexity of the state-by-state selection step is independent of the sample size: for instance, computing the spectral estimators on 300 models for 2,200,000 observations requires a bit less than two hours, and a bit more than ten hours for 10,000,000 observations, compared to less than ten seconds for the selection step in both cases. We refer to Section \ref{sec_complexity} for a more detailled discussion about the algorithmic complexity of the algorithms.

\subsection{Families of estimators}

We use two methods to construct families of estimators and apply the selection algorithm. The motivation and key result of this part of the paper is to control the variances of the estimators by the right penalty $\sigma$. This part is crucial if one wants to get adaptive minimax rates, and has not been adressed in previous papers. For both methods, we develop new theoretical results that allow to obtain a penalty $\sigma$ that leads to adaptive minimax rates of convergence up to a logarithmic term. We present the algorithms and theoretical guarantees in Section \ref{sec_auxiliary_methods}.

The first method is a spectral method and is detailed in Section \ref{sec_spectral_th}. Several spectral algorithms were developed, see for instance \cite{AHK12} and \cite{HKZ12} in the parametric setting, and \cite{robin2014estimating} and \cite{dCGLLC15} in a nonparametric framework. The main advantages of spectral methods are their computational efficiency and the fact that they do not resort to optimization procedure such as the EM and more generally nonconvex optimization algorithm, thus avoiding the well-documented issue of getting stuck into local sub-optimal minima.

Our spectral algorithm is based on the one studied in \cite{dCGLLC15}. However, their estimator cannot reach the minimax rate of convergence: the variance bound $\sigma(M)$ deduced from their results is proportional to $M^3$, while reaching the minimax rate requires $\sigma(M)$ to be proportional to $M$. To solve this issue, we introduce a modified version of their algorithm and show that it has the right variance bound, so that it is able to reach the adaptive minimax rate after our state-by-state selection procedure, up to a logarithmic term. Our algorithm also has an improved complexity: it is at most quasi-linear in the number of observations and in the model dimension, instead of cubic in the model dimension for the original algorithm.

The second method is a least squares method and is detailed in Section \ref{sec_leastsquares_th}. Nonparametric least squares methods were first introduced by \cite{dCGL15} to estimate the emission densities and extended by \cite{lehericy2016order} to estimate all parameters at once. They rely on estimating the density of three consecutive observations of the HMM using a least squares criterion. Since the model is identifiable from the distribution of three consecutive observations when the emission distributions are linearly independent, it is possible to recover the parameters from this density. In practice, these methods are more accurate than the spectral methods and are more stable when the models are close to not satisfying the identifiability condition, see for instance \cite{dCGL15} for the accuracy and \cite{lehericy2016order} for the stability. However, since they rely on the minimization of a nonconvex criterion, the computation times of the corresponding algorithms are often longer than the ones from spectral methods.

A key step in proving theoretical guarantees for least squares methods is to relate the error on the density of three consecutive observations to the error on the HMM parameters in order to obtain an oracle inequality on the parameters from the oracle inequality on the density of three observations. More precisely, the difficult part is to lower bound the error on the density by the error on the parameters. Let us write $g$ and $g'$ the densities of the first three observations of a HMM with parameters $\theta$ and $\theta'$ respectively (these parameters actually correspond to the transition matrix and the emission densities of the HMM). Then one would like to get
\begin{equation*}
\| g - g' \|_2 \geq \Ccal(\theta) \, d(\theta, \theta')
\end{equation*}
where $d$ is the natural $\Lbf^2$ distance on the parameters and $\Ccal(\theta)$ is a positive constant which does not depend on $\theta'$.
Such inequalities are then used to lower bound the variance of the estimator of the density of three observations $g^*$ by the variance of the parameter estimators: let $g$ be the projection of $g^*$ and $g'$ be the estimator of $g^*$ on the current approximation space (with index $M$). Denote $\theta^*_M$ and $\hat{\theta}_M$ the corresponding parameters and assume that the error $\| g - g' \|_2$ is bounded by some constant $\sigma'(M)$, then the result will be that
\begin{equation*}
d(\hat{\theta}_M, \theta^*_M) \leq \frac{\sigma'(M)}{\Ccal(\theta^*_M)}.
\end{equation*}

Such a result is crucial to control the variance of the estimators by a penalty term $\sigma$, which is the result we need for the state-by-state selection method. In the case where only the emission densities vary, \cite{dCGL15} proved that such an inequality always holds for HMMs with 2 hidden states using brute-force computations, but it is still unknown whether it is always true for larger number of states. When the number of states is larger than 2, they show that this inequality holds under a generic assumption. \cite{lehericy2016order} extended this result to the case where all parameters may vary.
However, the constants deduced from both articles are not explicit, and their regularity (when seen as a function of $\theta$) is unknown, which makes it impossible to use in our setting: one needs the constants $\Ccal(\theta^*_M)$ to be lower bounded by the same positive constant, which requires some sort of regularity on the function $\theta \longmapsto \Ccal(\theta)$ in the neighborhood of the true parameters.

To solve this problem, we develop a finer control of the behaviour of the difference $\| g - g' \|_2$, which is summarized in Theorem \ref{th_minoration_quad}. We show that it is possible to assume $\Ccal$ to be lower semicontinuous and positive without any additional assumption. In addition, we give an explicit formula for the constant when $\theta'$ and $\theta$ are close, which gives an explicit bound for the asymptotical rate of convergence.
This result allows us to control the variance of the least squares estimators by a penalty $\sigma$ which ensures that the state-by-state method reaches the adaptive minimax rate up to a logarithmic term.

\subsection{Numerical validation and application to real data sets}

Section \ref{sec_simulations} shows how to apply the state-by-state selection method in practice and shows its performance on simulated data and a comparison with a method based on cross validation that does note estimate state by state.

Note that the theoretical results give a penalty term $\sigma$ known only up to a multiplicative constant which is unknown in practice. This problem, the \emph{penalty calibration} issue, is usual in model selection methods. It can be solved using algorithms such as the dimension jump heuristics, see for instance \cite{birge2007minimal}, who introduce this heuristics and prove that it leads to an optimal penalization in the special case of Gaussian model selection framework. This method has been shown to behave well in practice in a variety of domains, see for instance \cite{BMM12a}. We describe the method and show how to use this heuristics to calibrate the penalties in Section \ref{sec_penalty_calibration}.

We propose and compare several variants of our algorithm. Section \ref{sec_penalty_calibration} shows some variants in the calibration of the penalties and Section \ref{sec_alternatives} shows other ways to select the final estimator. We discuss the result of the simulations and the convergence of the selected estimators in Section \ref{sec_results}.

In Section \ref{sec_VC}, we compare our method with a non state-by-state adaptive method based on cross validation.
Finally, we discuss the complexities of the auxiliary estimation methods and of our selection procedures in Section \ref{sec_complexity}.

In Section \ref{sec_appli}, we apply our algorithm to two sets of GPS tracks. The first data set contains trajectories of artisanal fishers from Madagascar, recorded using a regular sampling with 30 seconds time steps. The second data set contains GPS positions of Peruvian seabird, recorded with 1 second time steps.
We convert these tracks into the average velocity during each time step and apply our method using spectral estimators as input. The observed behaviour confirms the ability of our method to adapt to the different regularities by selecting different dimensions for each emission density.

Section \ref{sec_conclusion} contains a conclusion and perspectives for this work.

Finally, Appendix A contains the details of our spectral algorithm and Appendix \ref{sec_proofs} is dedicated to the proofs.

\subsection{Notations}

We will use the following notations throughout the paper.
\begin{itemize}[itemsep=0pt]
\item $[K] = \{1, \dots, K\}$ is the set of integers between 1 and $K$.

\item $\Sfrak(K)$ is the set of permutations of $[K]$.

\item $\| \cdot \|_F$ is the Frobenius norm. We implicitely extend the definition of the Frobenius norm to tensors with more than 2 dimensions.

\item $\Span(A)$ is the linear space spanned by the family $A$.

\item $\sigma_1(A) \geq \dots \geq \sigma_{p \wedge n}(A)$ are the singular values of the matrix $A \in \Rbb^{n \times p}$.

\item $\Lbf^2(\Ycal, \mu)$ is the set of real square integrable measurable functions on $\Ycal$ with respect to the measure $\mu$.

\item For $\fbf = (f_1, \dots, f_K) \in \Lbf^2(\Ycal, \mu)^K$, $G(\fbf)$ is the Gram matrix of $\fbf$, defined by $G(\fbf)_{i,j} = \langle f_i, f_j \rangle$ for all $i,j \in [K]$.
\end{itemize}

\section{The state-by-state selection procedure}
\label{sec_statebystatemethod}

In this section, we introduce the framework and our state-by-state selection method.

In Section \ref{sec_Lepski_framework}, we introduce the notations and assumptions. In Section \ref{sec_lepski_th}, we present our selection method and prove that it satisfies an oracle inequality.

\subsection{Framework and assumptions}
\label{sec_Lepski_framework}

Let $(X_j)_{j \geq 1}$ be a Markov chain with finite state space $\Xcal$ of size $K$. Let $\Qbf^*$ be its transition matrix and $\pi^*$ be its initial distribution.
Let $(Y_j)_{j \geq 1}$ be random variables on a measured space $(\Ycal, \mu)$ with $\mu$ $\sigma$-finite such that conditionally on $(X_j)_{j \geq 1}$ the $Y_j$'s are independent with a distribution depending only on $X_j$. Let $\nu^*_k$ be the distribution of $Y_j$ conditionally to $\{X_j = k\}$. Assume that $\nu^*_k$ has density $f^*_k$ with respect to $\mu$. We call $(\nu^*_k)_{k \in \Xcal}$ the \emph{emission distributions} and $\fbf^* = (f^*_k)_{k \in \Xcal}$ the \emph{emission densities}.
Then $(X_j, Y_j)_{j \geq 1}$ is a hidden Markov model with parameters $(\pi^*, \Qbf^*, \fbf^*)$. The hidden chain $(X_j)_{j \geq 1}$ is assumed to be unobserved, so that the estimators are based only on the observations $(Y_j)_{j \geq 1}$.

Let $(\Pfrak_M)_{M \in \Nbb}$ be a nested family of finite-dimensional subspaces such that their union is dense in $\Lbf^2(\Ycal, \mu)$.
The spaces $(\Pfrak_M)_{M \in \Nbb}$ are our models; in the following we abusively call $M$ the model instead of $\Pfrak_M$.
For each index $M \in \Nbb$, we write $\fbf^{*, (M)} = (f^{*, (M)}_k)_{k \in \Xcal}$ the projection of $\fbf^*$ on $(\Pfrak_M)^K$. It is the best approximation of the true densities within the model $M$.

In order to estimate the emission densities, we do not need to use every models. Typically there is no point in taking models with more dimensions than the sample size, since they will likely be overfitting. Let $\Mcal_n \subset \Nbb$ be the set of indices which will be used for the estimation from $n$ observations. For each $M \in \Mcal_n$, we assume we are given an estimator $\hat{\fbf}^{(M)}_n = (\hat{f}^{(M)}_{n,k})_{k \in \Xcal} \in (\Pfrak_M)^K$. We will need to assume that for all models, the variance\---that is the distance between $\hat{\fbf}^{(M)}_n$ and $\fbf^{*, (M)}$\---is small with high probability. In the following, we drop the dependency in $n$ and simply write $\Mcal$ and $\hat{\fbf}^{(M)}$.

The following result is what one usually obtains in model selection. It bounds the distance between the estimators $\hat{\fbf}^{(M)}$ and the projections $\fbf^{*, (M)}$ by some penalty function $\sigma$. Thus, $\sigma/2$ can be seen as a bound of the variance term.
\begin{description}
\item[\textbf{[H$(\epsilon)$]}] With probability $1 - \epsilon$,
\begin{equation*}
\forall M \in \Mcal, \quad \inf_{\tau_{n,M} \in \Sfrak(K)} \max_{k \in \Xcal} \left\| \hat{f}^{(M)}_k - f^{*, (M)}_{\tau_{n,M}(k)} \right\|_2 \leq \frac{\sigma(M, \epsilon,n)}{2}
\end{equation*}
\end{description}
where the upper bound $\sigma : (M, \epsilon,n) \in \Mcal \times [0,1] \times \Nbb^* \longmapsto \sigma(M, \epsilon,n) \in \Rbb_+$ is nondecreasing in $M$. We show in Sections \ref{sec_spectral_th} and \ref{sec_leastsquares_th} how to obtain such a result for a spectral method and for a least squares method (using an algorithm from \cite{lehericy2016order}). In the following, we omit the parameters $\epsilon$ and $n$ in the notations and only write $\sigma(M)$.

What is important for the selection step is that the permutation $\tau_{n,M}$ does not depend on the model $M$: one needs all estimators $(\hat{f}^{(M)}_k)_{M \in \Mcal}$ to correspond to the same emission density, namely $f^*_{\tau_n(k)}$ when $\tau_{n,M} = \tau_n$ is the same for all models $M$. This can be done in the following way: let $M_0 \in \Mcal$ and let
\begin{equation*}
\hat{\tau}^{(M)} \in \underset{\tau \in \Sfrak(K)}{\argmin} \left\{
	\max_{k \in \Xcal} \left\| \hat{f}^{(M)}_{\tau(k)} - \hat{f}^{(M_0)}_k \right\|_2
\right\}
\end{equation*}
for all $M \in \Mcal$. Then, consider the estimators obtained by swapping the hidden states by these permutations. In other words, for all $k \in \Xcal$, consider
\begin{align*}
\hat{f}^{(M)}_{k, \new} &= \hat{f}^{(M)}_{\hat{\tau}^{(M)}(k)}.
\end{align*}

Now, assume that the error on the estimators is small enough. More precisely, write $B_{M,M_0} = \max_{k \in \Xcal} \left\| f^{*,(M)}_k - f^{*,(M_0)}_k \right\|_2$ the distance between the projections of $\fbf^*$ on the models $M$ and $M_0$ and assume that $2 \left[ \sigma(M)/2 + \sigma(M_0)/2 + B_{M,M_0} \right]$ (that is twice the upper bound of the distance between two estimated emission densities corresponding to the same hidden states in models $M$ and $M_0$) is smaller than $m(\fbf^*,M_0) := \min_{k' \neq k} \left\| f^{*,(M_0)}_k - f^{*,(M_0)}_{k'} \right\|_2$, which is the smallest distance between two different densities of the vector $\fbf^{*,(M_0)}$.

Then \textbf{[H$(\epsilon)$]} ensures that with probability as least $1 - \epsilon$, for all $k$, there exists a single component of $\hat{\fbf}^{(M)}$ that is closer than $\sigma(M)/2 + \sigma(M_0)/2$ of $f^{*,(M_0)}_k$, and this component will be $\hat{f}^{(M)}_{\hat{\tau}^{(M)}(k)}$ by definition. This is summarized in the following lemma.
\begin{lemma}
\label{lemma_permutation_unique}
Assume \textbf{[H$(\epsilon)$]} holds. Then with probability $1 - \epsilon$, 
there exists a permutation $\tau_n \in \Sfrak(K)$ such that for all $k \in \Xcal$ and for all $M \in \Mcal$ such that
\begin{equation*}
\displaystyle \sigma(M) + \sigma(M_0) + 2 B_{M,M_0} < m(\fbf^*,M_0),
\end{equation*}
one has
\begin{equation}
\label{eq_Halt}
\max_{k \in \Xcal} \left\| \hat{f}^{(M)}_{k, \new} - f^{*, (M)}_{\tau_n(k)} \right\|_2 \leq \frac{\sigma(M)}{2}.
\end{equation}
\end{lemma}
\begin{proof}
Proof in Section \ref{sec_preuve_lemma_permutation}
\end{proof}

Thus, this property holds asymptotically as soon as $\inf \Mcal$ tends to infinity and $\sup_{M \in \Mcal} \sigma(M)$ tends to zero.

\subsection{Estimator and oracle inequality}
\label{sec_lepski_th}

Let us now introduce our selection procedure. This method and the following theorem are based on the approach of \cite{goldenshluger2011bandwidth}, but do not require any assumption on the structure of the estimators, provided a variance bound such as Equation~\eqref{eq_Halt} holds.

For each $k \in \Xcal$ and $M \in \Mcal$, let
\begin{equation*}
A_k(M) = \underset{M' \in \Mcal}{\sup} \left\{ \left\| \hat{f}^{(M')}_k - \hat{f}^{(M \wedge M')}_k \right\|_2 - \sigma(M') \right\}.
\end{equation*}
$A_k(M)$ serves as a replacement for the bias of the estimator $\hat{f}^{(M)}_k$, as can be seen in Equation (\ref{eq_A_is_the_bias}). This comes from the fact that for large $M'$, the quantity $\| \hat{f}^{(M')}_k - \hat{f}^{(M)}_k \|_2$ is upper bounded by the variances $\| \hat{f}^{(M')}_k - f^{*,(M')}_k \|_2$ and $\| \hat{f}^{(M)}_k - f^{*,(M)}_k \|_2$ (which are bounded by $\sigma(M')/2$) plus the bias $\| f^{*,(M)}_k - f^*_k \|_2$. Thus, only the bias term remains after substracting the variance bound $\sigma(M')$.

Then, for all $k \in \Xcal$, select a model through the bias-variance tradeoff
\begin{equation*}
\hat{M}_k \in \underset{M \in \Mcal}{\argmin} \{ A_k(M) + 2 \sigma(M) \}
\end{equation*}
and finally take
\begin{equation*}
\hat{f}_k = \hat{f}^{(\hat{M}_k)}_k.
\end{equation*}

The following theorem shows an oracle inequality on this estimator.
\begin{theorem}
\label{th_oracleLepski}
Let $\epsilon \geq 0$ and assume equation (\ref{eq_Halt}) holds for all $k \in \Xcal$ with probability $1 - \epsilon$. Then with probability $1 - \epsilon$,
\begin{equation*}
\forall k \in \Xcal, \quad
\| \hat{f}_k - f^*_{\tau_n(k)} \|_2
	\leq 4 \inf_{M \in \Mcal} \left\{ \| f^{*, (M)}_{\tau_n(k)} - f^*_{\tau_n(k)} \|_2 + \sigma(M, \epsilon) \right\}.
\end{equation*}
\end{theorem}

\begin{proof}
We restrict ourselves to the event of probability at least $1-\epsilon$ where equation (\ref{eq_Halt}) holds for all $k \in \Xcal$.

The first step consists in decomposing the total error: for all $M \in \Mcal$ and $k \in \Xcal$,
\begin{align*}
\left\| \hat{f}^{(\hat{M}_k)}_k - f^*_{\tau_n(k)} \right\|_2
	\leq& \left\| \hat{f}^{(\hat{M}_k)}_k - \hat{f}^{(\hat{M}_k \wedge M)}_k \right\|_2
		+ \left\| \hat{f}^{(\hat{M}_k \wedge M)}_k - \hat{f}^{(M)}_k \right\|_2 \\
		&+ \left\| \hat{f}^{(M)}_k - f^{*, (M)}_{\tau_n(k)} \right\|_2
		+ \left\| f^{*, (M)}_{\tau_n(k)} - f^*_{\tau_n(k)} \right\|_2.
\end{align*}
From now on, we will omit the subscripts $k$ and $\tau_n(k)$. Using equation (\ref{eq_Halt}) and the definition of $A(M)$ and $\hat{M}$, one gets
\begin{align*}
\left\| \hat{f}^{(\hat{M})} - f^* \right\|_2
	\leq& (A(M) + \sigma(\hat{M}))
		+ (A(\hat{M}) + \sigma(M)) \\
		&+ \sigma(M)
		+ \left\| f^{*, (M)} - f^* \right\|_2 \\
	\leq& 2 A(M) + 4 \sigma(M) + \left\| f^{*, (M)} - f^* \right\|_2.
\end{align*}
Then, notice that $A(M)$ can be bounded by
\begin{align*}
A(M) \leq& \sup_{M'} \left\{
	\left\| \hat{f}^{(M')} - f^{*, (M')} \right\|_2
	+ \left\| \hat{f}^{(M \wedge M')} - f^{*, (M \wedge M')} \right\|_2 - \sigma(M') \right\} \\
	&+ \sup_{M'} \left\| f^{*, (M')} - f^{*, (M \wedge M')} \right\|_2.
\end{align*}
Since $\sigma$ is nondecreasing, $\sigma(M \wedge M') \leq \sigma(M')$, so that the first term is upper bounded by zero thanks to equation (\ref{eq_Halt}). The second term can be controlled since the orthogonal projection is a contraction. This leads to
\begin{equation}
\label{eq_A_is_the_bias}
A(M) \leq \left\| f^* - f^{*, (M)} \right\|_2,
\end{equation}
which is enough to conclude.
\end{proof}

\begin{remark}
The oracle inequality also holds when taking
\begin{equation*}
A_k(M) = \underset{M' \geq M}{\sup} \left\{ \left\| \hat{f}^{(M')}_k - \hat{f}^{(M)}_k \right\|_2 - \sigma(M') \right\}_+.
\end{equation*}
\end{remark}

\begin{remark}
Note that the selected $\hat{M}_k$ implicitely depends on the probability of error $\epsilon$ through the penalty $\sigma$.

In the asymptotic setting, we take $\epsilon$ as a function of $n$, so that $\hat{M}_k$ is a function of $n$ only. This will be used to get rid of $\epsilon$ when proving that the estimators reach the minimax rates of convergence.
\end{remark}

\section{Plug-in estimators and theoretical guarantees}
\label{sec_auxiliary_methods}

In this section, we introduce two methods to construct families of estimators of the emission densities. We show that they satisfy assumption \textbf{[H$(\epsilon)$]} for a given variance bound $\sigma$.

In Section \ref{sec_framework}, we introduce the assumptions we will need for both methods. Section \ref{sec_spectral_th} is dedicated to the spectral estimator and Section \ref{sec_leastsquares_th} to the least squares estimator.

\subsection{Framework and assumptions}
\label{sec_framework}

Recall that we approximate $\Lbf^2(\Ycal, \mu)$ by a nested family of finite-dimensional subspaces $(\Pfrak_M)_{M \in \Mcal}$ such that their union is dense in $\Lbf^2(\Ycal, \mu)$ and write $f^{*,(M)}_k$ the orthogonal projection of $f^*_k$ on $\Pfrak_M$ for all $k \in \Xcal$ and $M \in \Mcal$. We assume that $\Mcal \subset \Nbb$ and that the space $\Pfrak_M$ has dimension $M$. A typical way to construct such spaces is to take $\Pfrak_M$ spanned by the first $M$ vectors of an orthonormal basis.

Both methods will construct an estimator of the emission densities for each model of this family. These estimators will then be plugged in the state-by-state selection method of Section \ref{sec_lepski_th}, which will select one model for each state of the HMM.

We will need the following assumptions. 
\begin{description}
\item[\textbf{[HX]}] $(X_j)_{j \geq 1}$ is a stationary ergodic Markov chain with parameters $(\pi^*, \Qbf^*)$;

\item[\textbf{[Hid]}] $\Qbf^*$ is invertible and the family $\fbf^*$ is linearly independent.
\end{description}
The ergodicity assumption in \textbf{[HX]} is standard in order to obtain convergence results. In this case, the initial distribution is forgotten exponentially fast, so that the HMM will essentially behave like a stationary process after a short period of time. For the sake of simplicity, we assume the Markov chain to be stationary.

\textbf{[Hid]} appears in identifiability results, see for instance \cite{GCR15} and Theorem \ref{th_identifiabilite}. It is sufficient to ensure identifiability of the HMM from the law of three consecutive observations. Note that it is in general not possible to recover the law of a HMM from two observations (see for instance Appendix G of \cite{AHK12}), so that three is actually the minimum to obtain general identifiability.

\subsection{The spectral method}
\label{sec_spectral_th}

{
\begin{algorithm}[!h]
  \caption{Spectral estimation of the emission densities of a HMM (short version)}
\label{alg:Spectral}
  \SetAlgoLined
  \KwData{A sequence of observations $(Y_{1}, \dots ,Y_{n+2})$, two dimensions $m \leq M$, an orthonormal basis $(\varphi_1, \dots ,\varphi_M)$ and number of retries $r$.}
  \KwResult{Spectral estimators $(\hat f^{(M,r)}_k)_{k \in \Xcal}$.}
    \BlankLine
\begin{enumerate}[{\bf [Step 1]}]
\item Consider the following empirical estimators: for any $a, c \in [m]$ and $b \in [M]$,
\begin{itemize}
\item $\hat{\Mbf}_{m,M,m}(a,b,c):= \frac{1}{n} \sum_{s=1}^{n}\varphi_{a}(Y_{s})\varphi_{b}(Y_{s+1})\varphi_{c}(Y_{s+2})$
\item $\hat{\Pbf}_{m,m}(a,c):= \frac{1}{n} \sum_{s=1}^{n}\varphi_{a}(Y_{s})\varphi_{c}(Y_{s+2})$.
\end{itemize}

\item
Let $\hat\Ubf_m$ be the $m \times K$ matrix of orthonormal left singular vectors of $\hat\Pbf_{m,m}$ corresponding to its top $K$ singular values. $\hat\Ubf_m$ can be seen as a projection. Denote by $\Pbf'$ and $\Mbf'(\cdot,b,\cdot)$ the projected tensors, defined by $\Pbf' = \hat\Ubf_m^\top \hat\Pbf_{m,m} \hat\Ubf_m$ and likewise for $\Mbf'$.

\item 
Form the matrices $\Bbf(b) := (\Pbf')^{-1} \Mbf'$ for all $b \in [M]$.

\item Construct a matrix $\hat\Obf$ by taking the best approximate simultaneous diagonalization of all $\Bbf(b)$ among $r$ attempts: for all $b \in [M]$, $\Bbf(b) \approx \Rbf \Diag[\hat\Obf(b,\cdot)] \Rbf^{-1}$ for some matrix $\Rbf$ (see details in Algorithm~\ref{alg:Spectral_complet}, in Appendix~\ref{app_spectral}).

\item Define the emission densities estimators $\hat\fbf^{(M,r)} := (\hat f^{(M,r)}_k)_{k \in \Xcal}$ by: for all $k \in \Xcal$, $\hat f^{(M,r)}_k := \sum_{b=1}^M \hat{\Obf}(b,k) \varphi_b$.
\end{enumerate}
\end{algorithm}}

Algorithm \ref{alg:Spectral} is a variant of the spectral algorithm introduced in \cite{dCGLLC15}. Unlike the original one, it is able to reach the minimax rate of convergence thanks to two improvements. The first one consists in decomposing the joint density on different models, hence the use of two dimensions $m$ and $M$. The second one consists in trying several randomized joint diagonalizations instead of just one, and selecting the best one, hence the parameter $r$. These additional parameters do not actually add much to the complexity of the algorithm: in theory, the choice $m, r \approx \log(n)$ is fine (see Corollary \ref{cor_spectral_lepski_rate}), and in practice, any large enough constant works, see Section \ref{sec_simulations} for more details.

For all $M \in \Mcal$, let $(\varphi_1^M, \dots, \varphi_M^M)$ be an orthonormal basis of $\Pfrak_M$. Let
\begin{align*}
\eta_3(m,M)^2 := \sup_{y, y' \in \Ycal^3} \sum_{a,c = 1}^m \sum_{b=1}^M 
		( \varphi_a^m(y_1) \varphi_b^M(y_2) \varphi_c^m(y_3)
		- \varphi_a^m(y'_1) \varphi_b^M(y'_2) \varphi_c^m(y'_3)
	)^2.
\end{align*}

The following theorem follows the proof of Theorem 3.1 of \cite{dCGLLC15}, with modifications that allow to control the error of the spectral estimators in expectation and are essential to obtain the right rates of convergence in Corollary \ref{cor_spectral_lepski_rate}.
\begin{theorem}
\label{th_garantieSpectral}
Assume \textbf{[HX]} and \textbf{[Hid]} hold. Then there exists a constant $M_0$ depending on $\fbf^*$ and constants $C_\sigma$ and $n_1$ depending on $\fbf^*$ and $\Qbf^*$ such that for all $\epsilon \in (0,1)$, for all $m,M \in \Mcal$ such that $M \geq m \geq M_0$ and for all ${n \geq n_1 \eta_3^2(m,M) (-\log \, \epsilon)^2}$, with probability greater than $1 - 6 \epsilon$,
\begin{equation*}
\inf_{\tau \in \Sfrak(K)}  \max_{k \in \Xcal} \| \hat{f}^{(M, \lceil t \rceil)}_k - f^{*, (M)}_{\tau(k)} \|_2^2
	\leq C_\sigma \eta_3^2(m,M) \frac{(-\log \, \epsilon)^2}{n}
\end{equation*}
\end{theorem}
\begin{proof}
Proof in Section \ref{sec_preuve_spectral}.
\end{proof}

Note that the constants $n_1$ and $C_\sigma$ depend on $\Qbf^*$ and $\fbf^*$. This dependency will not affect the rates of convergence of the estimators (with respect to the sample size $n$), but it can change the constants of the bounds and the minimum sample size needed to reach the asymptotic regime.

Let us now apply the state-by-state selection method to these estimators. The following corollary shows that it is possible to reach the minimax rate of convergence up to a logarithmic term separately for each state under standard assumptions. Note that we need to bound the resulting estimators by some power of $n$, but this assumption is not very restrictive since $\alpha$ can be arbitrarily large.

\begin{corollary}
\label{cor_spectral_lepski_rate}
Assume \textbf{[HX]} and \textbf{[Hid]} hold. Also assume that $\eta_3^2(m, M) \leq C_\eta m^2 M$ for a constant $C_\eta > 0$ and that for all $k \in \Xcal$, there exists $s_k$ such that $\|f^{*,(M)}_k - f^*_k\|_2 = O(M^{-s_k})$. Then there exists a constant $C_\sigma$ depending on $\fbf^*$ and $\Qbf^*$ such that the following holds.

Let $\alpha > 0$ and $C \geq 2 (1+2\alpha) \sqrt{C_\eta C_\sigma}$. Let $\hat{\fbf}^\text{sbs}$ be the estimators selected from the family $(\hat{\fbf}^{(M, \lceil(1+2\alpha)\log(n)\rceil)})_{M \leq M_{\max}(n)}$ with $M_{\max}(n) = n / \log(n)^5$, $m_M = \log(n)$ and $\sigma(M) = C \sqrt{\frac{M \log(n)^4}{n}}$ for all $M$. Then there exists a sequence of random permutations $(\tau_n)_{n \geq 1}$ such that
\begin{equation*}
\forall k \in \Xcal, \qquad
	\Ebb\left[
		 \left\| (-n^\alpha) \vee (\hat{f}^\text{sbs}_{\tau_n(k)} \wedge n^\alpha) - f^*_k \right\|_2^2
	\right]
		= O\left( \left( \frac{n}{\log(n)^4} \right)^{\frac{-2s_k}{2s_k+1}} \right).
\end{equation*}
\end{corollary}

The novelty of this result is that each emission density is estimated with its own rate of convergence: the rate $\frac{-s_k}{2s_k+1}$ is different for each emission density, even though the original spectral estimators did not handle them separately. This is due to our state-by-state selection method.

Moreover, it is able to reach the minimax rate for each density in an adaptive way. For instance, in the case of a $\beta$-Hölder density on $\Ycal = [0,1]^D$ (equipped with a trigonometric basis), one can easily check the control of $\eta_3$, and the control $\|f^{*,(M)}_k - f^*_k\|_2 = O(M^{- \beta / D})$ follows from standard approximation results, see for instance \cite{devore1993constructive}. Thus, our estimators converge with the rate $(n/\log(n)^4)^{-2 \beta / (2 \beta + D)}$ to this density: this is the minimax rate up to a logarithmic factor.

\begin{remark}
\label{remark_reordering_for_asymptotics}
By aligning the estimators like in Section \ref{sec_Lepski_framework}, one can replace the sequence of permutations in Corollary \ref{cor_spectral_lepski_rate} by a single permutation, in other words there exists a random permutation $\tau$ which does not depend on $n$ such that
\begin{equation*}
\forall k \in \Xcal, \qquad
	\Ebb\left[
		 \left\| (-n^\alpha) \vee (\hat{f}^\text{sbs}_{\tau(k)} \wedge n^\alpha) - f^*_k \right\|_2^2
	\right]
		= O\left( \left( \frac{n}{\log(n)^4} \right)^{\frac{-2s_k}{2s_k+1}} \right).
\end{equation*}
This means that the sequence $(\hat{f}^\text{sbs}_k)_{n \geq 1}$ is an adaptive rate-minimax estimator of $f^*_k$\---or more precisely of one of the emission densities $(f^*_{k'})_{k' \in \Xcal}$, but since the distribution of the HMM is invariant under relabelling of the hidden states, one can assume the limit to be $f^*_k$ without loss of generality\---up to a logarithmic term.
\end{remark}

At this point, it is important to note that the choice of the constant $C \geq 2(1+2 \alpha) \sqrt{C_\eta C_\sigma}$ depends on the hidden parameters of the HMM and as such is unknown. This penalty calibration problem is very common in the model selection framework and can be solved in practice using methods such as the slope heuristics or the dimension jump method which have been proved to be theoretically valid in several cases, see for instance \cite{BMM12a} and references therein. We use the dimension jump method and explain its principle and implementation in Section \ref{sec_penalty_calibration}.

\begin{proof}
Using Theorem \ref{th_garantieSpectral}, one gets that for all $n$ and for all $M \in \Mcal$ such that $n \geq n_1 \eta_3^2(m_M,M) (1+2\alpha)^2 \log(n)^2$, with probability $1 - 6n^{-1-2\alpha}$,
\begin{align*}
\inf_{\tau \in \Sfrak(K)}  \max_{k \in \Xcal} \| \hat{f}^{(M, \lceil t \rceil)}_k - f^{*, (M)}_{\tau(k)} \|_2^2
	\leq& C_\sigma \eta_3^2(m_M,M) \frac{(1+2\alpha)^2 \log(n)^2}{n} \\
	\leq& (1+\alpha)^2 C_\sigma C_\eta M \frac{\log(n)^4}{n} \\
	\leq& \frac{\sigma(M)^2}{4}
\end{align*}
where $\sigma(M) = C \sqrt{\frac{M\log(n)^4}{n}}$ with $C$ such that $C^2 \geq 4 (1+2\alpha)^2 C_\sigma C_\eta$.

The condition on $M$ becomes
\begin{align*}
n \geq n_1 \log(n)^4 M (1+2\alpha)^2
\end{align*}
and is asymptotically true for all $M \leq M_{\max}(n)$ as soon as $M_{\max}(n) = o(n / \log(n)^4)$.

Thus, \textbf{[H($6n^{-(1+2\alpha)}$)]} is true for the family $(\hat{\fbf}^{(M, \lceil(1+2\alpha)\log(n)\rceil)})_{M \leq M_{\max}(n)}$. Note that the assumption $M_{\max}(n) = o(n / \log(n)^4)$ also implies that there exists $M_1$ such that for $n$ large enough, Lemma \ref{lemma_permutation_unique} holds for all $M \geq M_1$, so that Theorem \ref{th_oracleLepski} implies that for $n$ large enough, there exists a permutation $\tau_n$ such that with probability $1 - 6n^{-(1+2\alpha)}$, for all $k \in \Xcal$,
\begin{align*}
\| \hat{f}^\text{sbs}_{\tau_n(k)} - f^*_k \|_2
	\leq& \; 4 \inf_{M_1 \leq M \leq M_{\max}} \{  \| f^{*,(M)}_k - f^*_k \|_2 + \sigma(M) \} \\
	=& \; O \left( \inf_{M_1 \leq M \leq M_{\max}} \left\{ M^{-s_k} + \sqrt{\frac{M \log(n)^4}{n}} \right\} \right) \\
	=& \; O \left( \left( \frac{n}{\log(n)^4} \right)^{-s_k/(1+2s_k)} \right),
\end{align*}
where the tradeoff is reached for $M = (\frac{n}{\log(n)^4})^{1/(1+2s_k)}$, which is in $[M_1, M_{\max}(n)]$ for $n$ large enough.

Finally, write $A$ the event of probability smaller than $6n^{-(1+2\alpha)}$ where \textbf{[H($6n^{-(1+\alpha)}$)]} doesn't hold, then for $n$ large enough and for all $k \in \Xcal$,
\begin{align*}
\Ebb\left[
	 \left\| (-n^\alpha) \vee (\hat{f}^\text{sbs}_{\tau_n(k)} \wedge n^\alpha) - f^*_k \right\|_2^2
	\right]
	\leq& \; \Ebb\left[ \one_A
		\left\| \hat{f}^\text{sbs}_{\tau_n(k)} - f^*_k \right\|_2^2
	\right]
		+ \Ebb\left[ \one_{A^c}
			( n^{2\alpha} + \| f^*_k \|_2^2 )
	\right] \\
	=& \; O \left( \left( \frac{n}{\log(n)^4} \right)^{-2s_k/(1+2s_k)} \right)
		+ O \left( \frac{n^{2\alpha} + \| f^*_k \|_2^2}{n^{1+2\alpha}} \right) \\
	=& \; O \left( \left( \frac{n}{\log(n)^4} \right)^{-2s_k/(1+2s_k)} \right).
\end{align*}
\end{proof}

\subsection{The penalized least squares method}
\label{sec_leastsquares_th}

Let $\Fcal$ be a subset of $\Lbf^2(\Ycal, \mu)$. We will need the following assumption on $\Fcal$ in order to control the deviations of the estimators:
\begin{description}
\item[\textbf{[HF]}] $\fbf^* \in \Fcal^{K^*}$, $\Fcal$ is closed under projection on $\Pfrak_M$ for all $M \in \Mcal$ and
\begin{equation*}
\forall f \in \Fcal, \quad
\begin{cases}
\| f \|_\infty \leq \cinf \\
\| f \|_2 \leq \cquad
\end{cases}
\end{equation*}
with $\cinf$ and $\cquad$ larger than 1.
\end{description}
A simple way to construct such a set $\Fcal$ when $\mu$ is a finite measure is to take the sets $(\Pfrak_M)_M$ spanned by the first $M$ vectors of an orthonormal basis $(\varphi_i)_{i \geq 0}$ whose first vector $\varphi_0$ is proportional to $\mathbf{1}$. Then any set $\Fcal$ of densities such that $\int f d\mu = 1$, $\sum_i \langle f, \varphi_i \rangle^2 \leq \cquad$ and $\sum_i | \langle f, \varphi_i \rangle | \| \varphi_i \|_\infty \leq \cinf$ for given constants $\cquad$ and $\cinf$ and for all $f \in \Fcal$ satisfies \textbf{[HF]}.\\

When $\Qbf \in \Rbb^{K \times K}$, $\pi \in \Rbb^K$ and $\fbf \in (\Lbf^2(\Ycal, \mu))^K$, let
\begin{equation*}
g^{\pi, \Qbf, \fbf}(y_1, y_2, y_3) = \sum_{k_1, k_2, k_3 = 1}^K \pi(k_1) \Qbf(k_1, k_2) \Qbf(k_2, k_3) f_{k_1}(y_1) f_{k_2}(y_2) f_{k_3}(y_3).
\end{equation*}
When $\pi$ is a probability distribution, $\Qbf$ a transition matrix and $\fbf$ a $K$-uple of probability densities, then $g^{\pi, \Qbf, \fbf}$ is the density of the first three observations of a HMM with parameters $(\pi, \Qbf, \fbf)$. The motivation behind estimating $g^{\pi, \Qbf, \fbf}$ is that it allows to recover the true parameters under the identifiability assumption \textbf{[Hid]}, as shown in the following theorem.

Let $\Qcal$ be the set of transition matrices on $\Xcal$ and $\Delta$ the set of probability distributions on $\Xcal$. For a permutation $\tau \in \Sfrak(K)$, write $\Pbb_\tau$ its matrix (that is the matrix defined by $\Pbb_\tau(i,j) = \one_{\{j = \tau(i)\}}$). Finally, define the distance on the HMM parameters
\begin{multline*}
\dperm((\pi_1, \Qbf_1, \fbf_1), (\pi_2, \Qbf_2, \fbf_2))^2 \\
	= \inf_{\tau \in \Sfrak(K)} \left\{
		\| \pi_1 - \Pbb_\tau \pi_2 \|_2^2 + \| \Qbf_1 - \Pbb_\tau \Qbf_2 \Pbb_\tau^\top \|_F^2 + \sum_{k \in \Xcal} \| f_{1,k} - f_{2,\tau(k)} \|_2^2
	\right\}.
\end{multline*}
This distance is invariant under permutation of the hidden states. This corresponds to the fact that a HMM is only identifiable up to relabelling of its hidden states.

\begin{theorem}[Identifiability]
\label{th_identifiabilite}
Let $(\pi^*, \Qbf^*, \fbf^*) \in \Delta \times \Qcal \times (\Lbf^2(\Ycal, \mu))^K$ such that $\pi^*_x > 0$ for all $x \in \Xcal$ and \textbf{[Hid]} holds. Then for all $(\pi, \Qbf, \fbf) \in \Delta \times \Qcal \times (\Lbf^2(\Ycal, \mu))^K$,
\begin{equation*}
(g^{\pi, \Qbf, \fbf} = g^{\pi^*, \Qbf^*, \fbf^*}) \; \Rightarrow \; \dperm((\pi, \Qbf, \fbf),(\pi^*, \Qbf^*, \fbf^*)) = 0.
\end{equation*}
\end{theorem}
\begin{proof}
The spectral algorithm of \cite{dCGLLC15} applied on the finite dimensional space spanned by the components of $\fbf$ and $\fbf^*$ allows to recover all the parameters even when the emission densities are not probability densities and when the Markov chain is not stationary.
\end{proof}

\noindent
Define the empirical contrast
\begin{equation*}
\gamma_n(t) = \| t \|_2^2 - \frac{2}{n} \sum_{j=1}^n t(Z_j)
\end{equation*}
where ${Z_j := (Y_{j}, Y_{j+1}, Y_{j+2})}$ and $(Y_j)_{1 \leq j \leq n+2}$ are the observations. It is a biased estimator of the $\Lbf^2$ loss: for all $t \in (\Lbf^2(\Ycal, \mu))^3$,
\begin{equation*}
\Ebb[\gamma_n(t)] = \| t - g^* \|_2^2 - \| g^* \|_2^2
\end{equation*}
where $g^* = g^{\pi^*, \Qbf^*, \fbf^*}$. Since the bias does not depend on the function $t$, one can hope that the minimizers of $\gamma_n$ are close to minimizers of $\|t - g^*\|_2$. We will show that this is indeed the case.

The least squares estimators of all HMM parameters are defined for each model $\Pfrak_M$ by
\begin{equation*}
(\hat\pi^{(M)}, \hat\Qbf^{(M)}, \hat{\fbf}^{(M)}) \in \underset{\pi \in \Delta, \, \Qbf \in \Qcal, \, \fbf \in (\Pfrak_M \cap \Fcal)^K}{\argmin} \; \gamma_n(g^{\pi, \Qbf, \fbf}).
\end{equation*}
The procedure is summarized in Algorithm \ref{alg:LS}. Note that with the notations of the algorithm,
\begin{equation*}
\gamma_n(g^{\pi, \Qbf, \Obf^\top \Phi}) = \| \Mbf_{(\pi, \Qbf, \Obf)} - \hat\Mbf_M \|_F^2 - \| \hat\Mbf_M \|_F^2.
\end{equation*}

{
\begin{algorithm}[!t]
  \SetAlgoLined
  \KwData{A sequence of observations $(Y_{1}, \dots ,Y_{n+2})$, a dimension $M$ and an orthonormal basis $\Phi = (\varphi_1, \dots ,\varphi_M)$.}
  \KwResult{Least squares estimators $\hat\pi^{(M)}$, $\hat\Qbf^{(M)}$ and $ (\hat{f}^{(M)}_k)_{k \in \Xcal}$.}
    \BlankLine
\begin{enumerate}[{\bf [Step 1]}]
\item Compute the tensor $\hat{\Mbf}_M$ defined by $\hat{\Mbf}_{M}(a,b,c):= \frac{1}{n} \sum_{s=1}^{n}\varphi_{a}(Y_{s})\varphi_{b}(Y_{s+1})\varphi_{c}(Y_{s+2})$ for all $a, b, c \in [M]$.

\item
Find a minimizer $(\hat\pi^{(M)}, \hat\Qbf^{(M)}, \hat{\Obf})$ of $(\pi, \Qbf, \Obf) \longmapsto \| \Mbf_{(\pi, \Qbf, \Obf)} - \hat\Mbf_M \|_F^2$ where
\begin{itemize}
\item $\pi \in \Rbb^K$ is a probability distribution on $\Xcal$, i.e. $\sum_{k \in \Xcal} \pi_k = 1$;
\item $\Qbf \in \Rbb^{K \times K}$ is a transition matrix on $\Xcal$, i.e. $\sum_{k' \in \Xcal} Q(k,k') = 1$ for all $k \in \Xcal$;
\item $\Obf$ is a $M \times K$ matrix such that for all $k \in \Xcal$, $\sum_{b=1}^M \Obf(b,k) \varphi_b \in \Fcal$;
\item $\Mbf_{(\pi, \Qbf, \Obf)} \in \Rbb^{M \times M \times M}$ is defined by $\Mbf_{(\pi, \Qbf, \Obf)}(\cdot, b, \cdot) = \Obf \Diag[\pi] \Qbf \Diag[\Obf(b,\cdot)] \Qbf \Obf^\top$ for all $b \in [M]$.
\end{itemize}

\item
Consider the emission densities estimators $\hat\fbf^{(M)} := (\hat f^{(M)}_k)_{k \in \Xcal}$ defined by for all $k \in \Xcal$, $\hat f^{(M)}_k := \sum_{b=1}^M \hat{\Obf}(b,k) \varphi_b$.
\end{enumerate}
  \caption{Least squares estimation of the emission densities of a HMM}
\label{alg:LS}
\end{algorithm}}

Then, the proof of the oracle inequality of \cite{lehericy2016order} allows to get the following result.
\begin{theorem}
\label{th_oracle}
Assume \textbf{[HF]}, \textbf{[HX]} and \textbf{[Hid]} hold.

Then there exists constants $C$ and $n_0$ depending on $\cquad$, $\cinf$ and $\Qbf^*$ such that for all $n \geq n_0$, for all $t > 0$, with probability greater than ${1 - e^{-t}}$, one has for all $M \in \Mcal$ such that $M \leq n$:
\begin{align*}
\| \hat{g}^{\hat\pi^{(M)}, \hat\Qbf^{(M)}, \hat{\fbf}^{(M)}}
	- g^{\pi^*, \Qbf^*, \fbf^{*, (M)}} \|_2^2 
	&\leq C \Big( \frac{t}{n} + M \frac{\log(n)}{n}
	\Big).
\end{align*}
\end{theorem}

In order to deduce a control of the error on the parameters\---and in particular on the emission densities\---from the previous result, we will need to assume that the quadratic form derived from the second-order expansion of $(\pi, \Qbf, \fbf) \in \Delta \times \Qcal \times \Fcal^K \longmapsto \| g^{\pi, \Qbf, \fbf} - g^* \|_2^2$ around $(\pi^*, \Qbf^*, \fbf^*)$ is nondegenerate.

It is still unknown whether this nondegeneracy property is true for all parameters $(\pi^*, \Qbf^*, \fbf^*)$ such that \textbf{[Hid]} and \textbf{[HX]} hold. \cite{dCGL15} prove it for $K = 2$ hidden states when only the emission densities are allowed to vary by using brute-force computations. To do so, they introduce an (explicit) polynomial in the coefficients of $\pi^*$, $\Qbf^*$ and of the Gram matrix of $\fbf^*$ and prove that its value is nonzero if and only if the quadratic form is nondegenerate for the corresponding parameters. The difficult part of the proof is to show that this polynomial is always nonzero.

For the expression of this polynomial\---which we will write $H$\---in our setting, we refer to Section \ref{sec_expression_H}. Note that \cite{lehericy2016order} proves that this polynomial $H$ is non identically zero: it is shown that there exists parameters $(\pi, \Qbf, \fbf)$ satisfying \textbf{[HX]} and \textbf{[Hid]} such that $H(\pi, \Qbf, \fbf) \neq 0$, which means that the following assumption is generically satisfied:
\begin{description}
\item[\textbf{[Hdet]}] $H(\pi^*, \Qbf^*, \fbf^*) \neq 0$.
\end{description}

The following result allows to lower bound the $\Lbf^2$ error on the density of three consecutive observations by the error on the parameters of the HMM using this condition.
It is an improvement of Theorem 6 of \cite{dCGL15} and Theorem 9 of \cite{lehericy2016order}. The main difference is that the constant $c^*(\pi^*, \Qbf^*, \fbf^*, \Fcal)$ does not depend on the $\fbf$ around which the parameters are taken%, as long as $\fbf$ lies in a compact neighborhood of $\fbf^*$
. This is crucial to obtain Corollary \ref{cor_H0_LS}, from which we will deduce \textbf{[H0]}. Note that we do not need $\fbf$ to be in a compact neighborhood of $\fbf^*$. Another improvement is that the constant in the minoration only depends on the true parameters and on the set $\Fcal$.

\begin{theorem}
\label{th_minoration_quad}
\begin{enumerate}
\item
Assume that \textbf{[HF]} holds and that for all $f \in \Fcal$, $\int f d\mu = 1$.

Then there exist a lower semicontinuous function $(\pi^*, \Qbf^*, \fbf^*) \longmapsto c^*(\pi^*, \Qbf^*, \fbf^*, \Fcal)$ that is positive when \textbf{[Hid]} and \textbf{[Hdet]} hold and a neighborhood $\Vcal$ of $\fbf^*$ in $\Fcal^K$ depending only on $\pi^*$, $\Qbf^*$, $\fbf^*$ and $\Fcal$ such that for all $\fbf \in \Vcal$ and for all $\pi \in \Delta$, $\Qbf \in \Qcal$ and $\hbf \in \Fcal^K$,
\begin{equation*}
	\|g^{\pi, \Qbf, \hbf} - g^{\pi^*, \Qbf^*, \fbf}\|_2^2
	\geq c^*(\pi^*, \Qbf^*, \fbf^*, \Fcal) \dperm((\pi, \Qbf, \hbf),(\pi^*, \Qbf^*, \fbf))^2.
\end{equation*}

\item
There exists a continuous function $\epsilon : (\pi^*, \Qbf^*, \fbf^*) \mapsto \epsilon(\pi^*, \Qbf^*, \fbf^*)$ that is positive when \textbf{[Hid]} and \textbf{[Hdet]} hold and such that for all $\pi \in \Delta$, $\Qbf \in \Qcal$ and $\hbf \in (\Lbf^2(\Ycal, \mu))^K$ a $K$-uple of probability densities such that $\dperm((\pi, \Qbf, \hbf),(\pi^*, \Qbf^*, \fbf^*)) \leq \epsilon(\pi^*, \Qbf^*, \fbf^*)$, one has
\begin{equation*}
	\|g^{\pi, \Qbf, \hbf} - g^{\pi^*, \Qbf^*, \fbf^*}\|_2^2
	\geq c_0(\pi^*, \Qbf^*, \fbf^*) \dperm((\pi, \Qbf, \hbf),(\pi^*, \Qbf^*, \fbf^*))^2.
\end{equation*}
where
\begin{multline*}
c_0(\pi^*, \Qbf^*, \fbf^*)
	= \frac{\left(\inf_{k \in \Xcal} \pi^*(k) \right) \sigma_K(\Qbf^*)^4 \sigma_K(G(\fbf^*))^2}{4} \\
	\wedge \frac{H(\pi^*, \Qbf^*, \fbf^*)}{2 (1 \wedge K \| G(\fbf^*) \|_\infty) (3 K^3 (1 \vee \| G(\fbf^*) \|_\infty^4))^{K^2-K/2}}.
\end{multline*}
\end{enumerate}
\end{theorem}

\begin{proof}
Proof in Section \ref{sec_proof_minoration_quad}.
\end{proof}

\begin{corollary}
\label{cor_H0_LS}
Assume \textbf{[HX]}, \textbf{[HF]}, \textbf{[Hid]} and \textbf{[Hdet]} hold. Also assume that for all $f \in \Fcal$, $\int f d\mu = 1$.

Then there exists a constant $n_0$ depending on $\cquad$, $\cinf$ and $\Qbf^*$ and constants $M_0$ and $C'$ depending on $\Fcal$, $\Qbf^*$ and $\fbf^*$ such that for all $n \geq n_0$ and $t > 0$, with probability greater than $1 - e^{-t}$, one has for all $M \in \Mcal$ such that $M_0 \leq M \leq n$:
\begin{align*}
\inf_{\tau \in \Sfrak(K)} \max_{k \in \Xcal} \| \hat{f}^{(M)}_k - f^{*, (M)}_{\tau(k)} \|_2^2
	\leq C' \left( M \frac{\log(n)}{n} + \frac{t}{n} \right).
\end{align*}
\end{corollary}

\begin{remark}
Using the second point of Theorem \ref{th_minoration_quad}, one can alternatively take $n_0$ and $M_0$ depending on $\Fcal$, $\Qbf^*$ and $\fbf^*$, and $C'$ depending on $\cquad$, $\cinf$, $\Qbf^*$ and $\fbf^*$ only. For instance, one can take $C' =  C / c_0(\pi^*, \Qbf^*, \fbf^*)$ with the notations of Theorems \ref{th_oracle} and \ref{th_minoration_quad}.

In particular, this means that the asymptotic variance bound of the least squares estimators (and therefore the rate of convergence of the estimators selected by our state-by-state selection method) does not depend on the set $\Fcal$, but only on the HMM parameters and on the bounds $\cquad$ and $\cinf$ on the square and supremum norms of the emission densities. Note that this universality result is essentially an asymptotic one since it requires $n_0$ to depend on $\Fcal$ in a non-explicit way.
\end{remark}

\begin{proof}
Let $\Vcal$ be the neighborhood given by Theorem \ref{th_minoration_quad}, then there exists $M_0$ such that for all $M \geq M_0$, $\fbf^{*, (M)} \in \Vcal$. Then Theorem \ref{th_oracle} and Theorem \ref{th_minoration_quad} applied to $\pi = \hat{\pi}^{(M)}$, $\Qbf = \hat{\Qbf}^{(M)}$, $\hbf = \hat{\fbf}^{(M)}$ and $\fbf = \fbf^{*, (M)}$ for all $M$ allow to conclude.
\end{proof}

We may now state the following result which shows that the state-by-state selection method applied to these estimators reaches the minimax rate of convergence (up to a logarithmic factor) in an adaptive manner under generic assumptions. Its proof is the same as the one of Corollary \ref{cor_spectral_lepski_rate}.

\begin{corollary}
\label{cor_LS_lepski_rate}
Assume \textbf{[HX]}, \textbf{[HF]}, \textbf{[Hid]} and \textbf{[Hdet]} hold. Also assume that for all $f \in \Fcal$, $\int f d\mu = 1$ and that for all $k$, there exists $s_k$ such that $\|f^{*,(M)}_k - f^*_k\|_2 = O(M^{-s_k})$. Then there exists a constant $C_\sigma$ depending on $\cquad$, $\cinf$, $\Qbf^*$ and $\fbf^*$ such that the following holds.

Let $C \geq C_\sigma$ and let $\hat{\fbf}^\text{sbs}$ be the estimators selected from the family $(\hat{\fbf}^{(M)})_{M \leq n}$ with $\sigma(M) = C \sqrt{\frac{M \log(n)}{n}}$ for all $M$, aligned like in Remark \ref{remark_reordering_for_asymptotics}. Then there exists a random permutation $\tau$ which does not depend on $n$ such that
\begin{equation*}
\forall k \in \Xcal, \qquad
	\Ebb\left[
		 \left\| \hat{f}^\text{sbs}_{\tau(k)} - f^*_k \right\|_2
	\right]
		= O\left( \left( \frac{n}{\log(n)} \right)^{\frac{-s_k}{2s_k+1}} \right).
\end{equation*}
\end{corollary}

\section{Numerical experiments}
\label{sec_simulations}

This section is dedicated to the discussion of the practical implementation of our method. We run the spectral estimators on simulated data for different number of observations and study the rate of convergence of the selected estimators for several variants of our method. Finally, we discuss the algorithmic complexity of the different estimators and selection methods.

In Section \ref{sec_sim_setting}, we introduce the parameters with which we generate the observations. In Section \ref{sec_penalty_calibration}, we discuss how to calibrate the constant of the penalty in practice. In Section \ref{sec_alternatives}, we introduce two other ways to select the final estimators, the POS and MAX variants. Section \ref{sec_results} contains the results of the simulations for each variant and calibration method.
In Section \ref{sec_VC}, we present a cross validation procedure and compare its results with the one obtained using our method.
Finally, we discuss the algorithmic complexity of the different algorithms and estimators in Section \ref{sec_complexity}.

\subsection{Setting and parameters}
\label{sec_sim_setting}

We take $\Ycal = [0,1]$ equipped with the Lebesgue measure. We choose the approximation spaces spanned by a trigonometric basis: ${\Pfrak_M := \Span(\varphi_1, \dots, \varphi_M)}$ with
\begin{equation*}
\begin{cases}
\varphi_1(x) &= 1 \\
\varphi_{2m}(x) &= \sqrt{2} \cos(2\pi m x) \\
\varphi_{2m+1}(x) &= \sqrt{2} \sin(2\pi m x)
\end{cases}
\end{equation*}
for all $x \in [0,1]$ and $m \in \Nbb^*$. We will consider a hidden Markov model with $K=3$ hidden states and the following parameters:

\begin{itemize}[noitemsep]
\item Transition matrix
\begin{align*}
\Qbf^* &=
\left(\begin{array}{c c c}
0.7  & 0.1  & 0.2 \\
0.08 & 0.8  & 0.12\\
0.15 & 0.15 & 0.7
\end{array}\right) ;
\end{align*}

\item Emission densities (see Figure \ref{fig_emission_densities})
\begin{itemize}
\item Uniform distribution on $[0;1]$;
\item Symmetrized Beta distribution, that is a mixture with the same weight of $\frac{2}{3} X$ and $1 - \frac{1}{3} X'$ with $X,X'$ i.i.d. following a Beta distribution with parameters $(3,1.6)$;
\item Beta distribution with parameters $(3,7)$.
\end{itemize}
\end{itemize}

\begin{figure}[!h]
\centering
\includegraphics[scale=0.7]{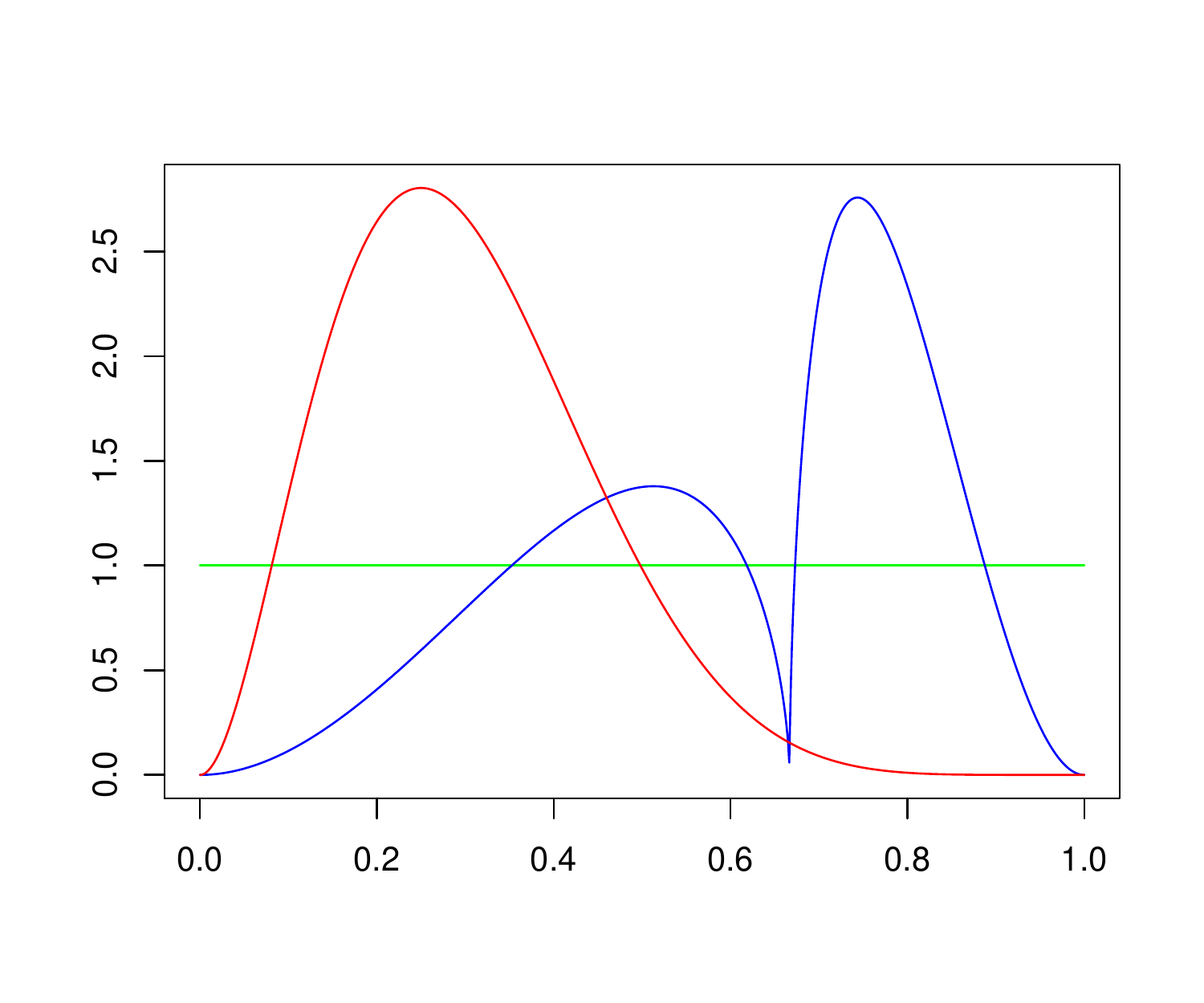}
\vspace{-1em}
\caption{Emission densities. In all following figures, the uniform distribution corresponds to the green lines, the Beta distribution to the red lines and the symmetrized Beta distribution to the blue lines.}
\label{fig_emission_densities}
\end{figure}

We generate $n$ observations and run the spectral algorithm in order to obtain estimators for the models $\Pfrak_M$ with $\verb?M_min? \leq M \leq \verb?M_max?$, $m = 20$ and $r = \lceil 2\log(n) + 2\log(M) \rceil$, where $\verb?M_min? = 3$ and $\verb?M_max? = 300$. Finally, we use the state-by-state selection method to choose the final estimator for each emission density. The main reason for using spectral estimators instead of maximum likelihood estimation or least squares estimation is its computational speed: it is much faster for large $n$ than the least squares algorithm or the EM algorithm, which makes studying asymptotic behaviours possible.

We made 300 simulations, 20 per value of $n$, with $n$ taking values in $\{$%
$5 \times 10^4, 7 \times 10^4, 1 \times 10^5, 1.5 \times 10^5, 2.2 \times 10^5, 3.5 \times 10^5, 5 \times 10^5, 7 \times 10^5, 1 \times 10^6, 1.5 \times 10^6, 2.2 \times 10^6, 3.5 \times 10^6, 5 \times 10^6, 7 \times 10^6, 1 \times 10^7$%
$\}$.

\subsection{Penalty calibration}
\label{sec_penalty_calibration}

It is important to note that when considering spectral and least squares methods, the penalty $\sigma$ in the state-by-state selection procedure depends on the hidden parameters of the HMM and as such is unknown in practice. This penalty calibration problem is well known and several procedures exist that allow to solve it, for instance the slope heuristics and the dimension jump method (see for instance \cite{BMM12a} and references therein). In the following, we will use the dimension jump method to calibrate the penalty in the state-by-state selection procedure.

Consider a penalty shape $\pen_\text{shape}$ and define $\hat{M}_k(\rho)$ the model selected for the hidden state $k$ by the state-by-state selection estimator using the penalty $\rho \, \pen_\text{shape}$:
\begin{equation*}
\hat{M}_k(\rho) \in \underset{M \in \Mcal}{\argmin} \{ A_k(M) + 2 \rho \, \pen_\text{shape}(M) \}.
\end{equation*}
where
\begin{equation*}
A_k(M) = \underset{M' \in \Mcal}{\sup} \left\{ \left\| \hat{f}^{(M')}_k - \hat{f}^{(M \wedge M')}_k \right\|_2 - \rho \, \pen_\text{shape}(M') \right\}.
\end{equation*}

The dimension jump method relies on the heuristics that there exists a constant $C$ such that $C \, \pen_\text{shape}$ is a \emph{minimal penalty}. This means that for all $\rho < C$, the selected models $\hat{M}_k(\rho)$ will be very large, while for $\rho > C$, the models will remain small. This translates into a sharp jump located around a value $\rho_{\text{jump},k} = C$ in the plot of $\rho \longmapsto \hat{M}_k(\rho)$. The final step consists in taking twice this value to calibrate the constant of the penalty, thus selecting the model $\hat{M}(2 \rho_{\text{jump},k})$. In practice, we take $\rho_{\text{jump},k}$ as the position of the largest jump of the function $\rho \longmapsto \hat{M}_k(\rho)$.

Figure \ref{fig_dim_jump} shows the resulting dimension jumps for $n = 220,000$ observations. Each curve corresponds to one of the $\hat{M}_k(\rho)$ and has a clear dimension jump, which confirms the relevance of the heuristics. Several methods may be used to calibrate the constant of the penalty:
\begin{description}
\item[\textbf{eachjump}.] Calibrate the constant independently for each state. This method has the advantage of being easy to calibrate since there is usually a single sharp jump in each state's complexity. However, our theoretical results do not suggest that the penalty constant is different for each state;

\item[\textbf{jumpmax}.] Calibrate the constant for all states together using only the latest jump. This consists in taking the maximum of the $\rho_{\text{jump},k}$ to select the final models. Since the penalty is known up to a multiplicative constant and taking a constant larger than needed does not affect the rates of convergence\---contrary to smaller constants\---this is the ``safe'' option;

\item[\textbf{jumpmean}.] Calibrate the constant for all states together using the mean of the positions of the different jumps.
\end{description}
We try and compare these calibration methods in Section \ref{sec_results}.

\begin{figure}[!t]
\centering
\vspace{-2em}
\includegraphics[scale=0.84]{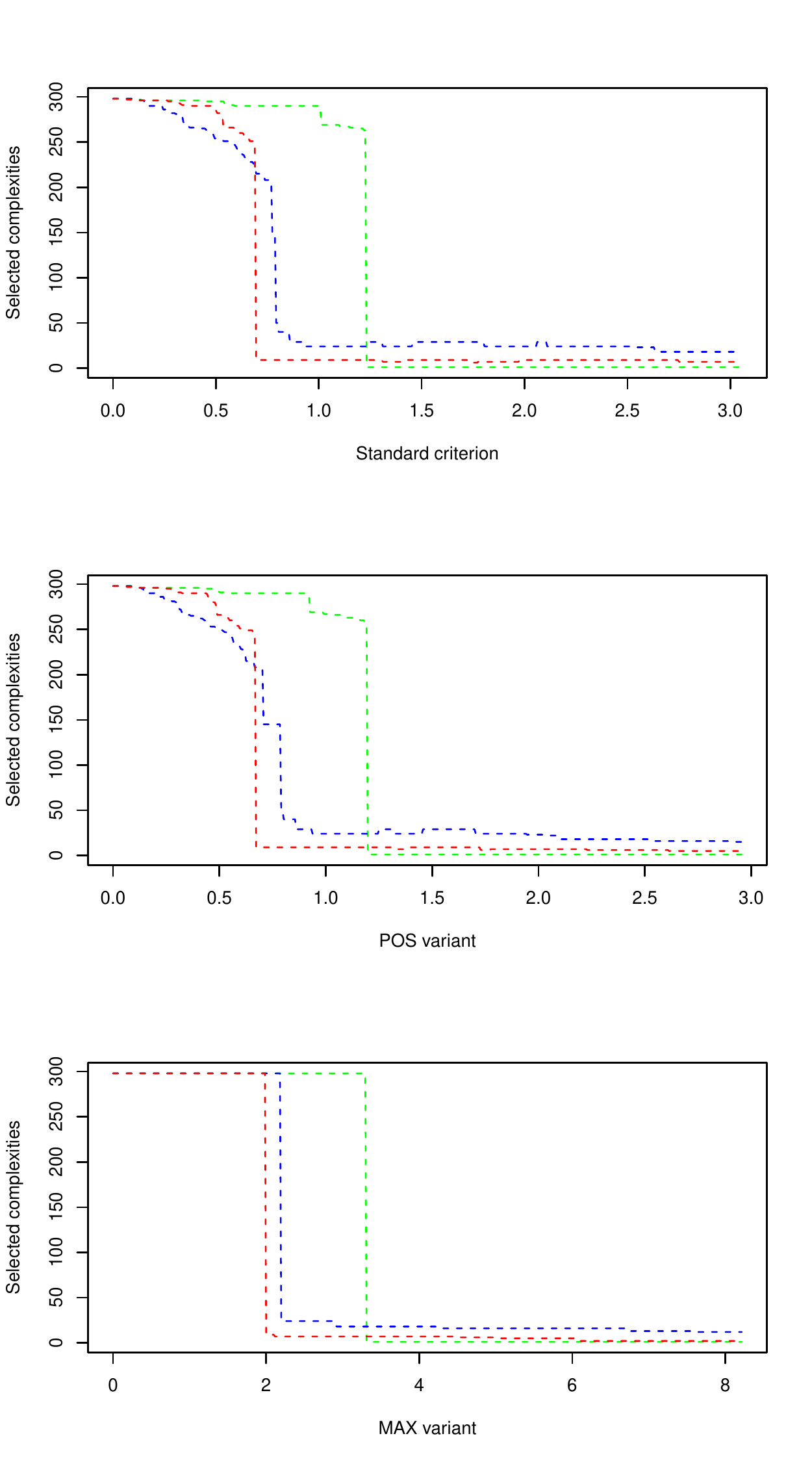} % Taille R : 5, 9 ; simulation 62.

\vspace{-17.5cm}
\hspace{6cm}\includegraphics[scale=0.15]{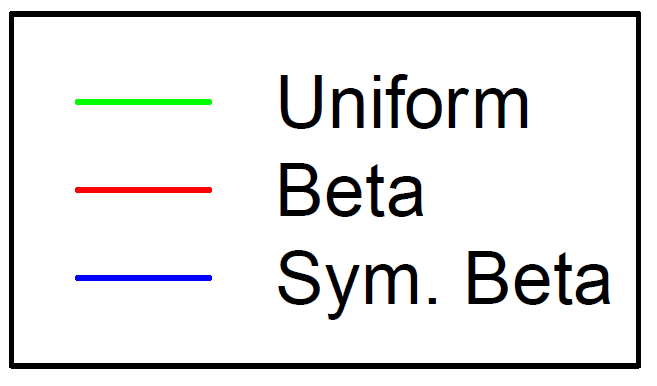}
\vspace{16cm}
\caption{Selected complexities with respect to the penalty constant $\rho$ for the same simulation of $n = 500, 000$ observations.  The colored dashed lines correspond to the single-state complexities $M_k(\rho)$.}
\label{fig_dim_jump}
\end{figure}

\subsection{Alternative selection procedures}
\label{sec_alternatives}

\subsubsection{Variant POS.}

As mentionned in Section \ref{sec_lepski_th}, it is also possible to select the estimators using the criterion
\begin{equation*}
A_k(M) = \underset{M' \geq M}{\sup} \left\{ \left\| \hat{f}^{(M')}_k - \hat{f}^{(M)}_k \right\|_2 - \sigma(M') \right\}_+
\end{equation*}
followed by
\begin{equation*}
\hat{M}_k \in \underset{M \in \Mcal}{\argmin} \{ A_k(M) + 2 \sigma(M) \}.
\end{equation*}
This positivity condition was in the original Goldenshluger-Lepski method. The theoretical guarantees remain the same as the previous method and both behave almost identically in practice, as shown in Section \ref{sec_results}.

\subsubsection{Variant MAX.}

In the context of kernel density estimation, \cite{lacour2016estimator} show that the Goldenshluger-Lepski method still works when the biais estimate $A_k(M)$ of the model $M$ is replaced by the distance between the estimator of the model $M$ and the estimator with the smallest bandwidth (the analog of the largest model in our setting). They also prove an oracle inequality for this method after adding a corrective term to the penalty.

The following variant is based on the same idea. It consists in selecting the model
\begin{equation*}
\hat{M}_k \in \underset{M \in \Mcal}{\argmin} \{ \| \hat{f}^{(M_{\max})}_k - \hat{f}^{(M)}_k \|_2 + \sigma(M) \}
\end{equation*}
for each $k \in \Xcal$ and takes
\begin{equation*}
\hat{f}_k = \hat{f}^{(\hat{M}_k)}_k,
\end{equation*}
where $\sigma$ is the same penalty as the one in the usual state-by-state selection method.

An advantage of this algorithm is its lower complexity, since it requires $O(M_{\max})$ computations of $\Lbf^2$ norms instead of $O(M_{\max}^2)$. We do not study this method theoretically in our setting. However, the simulations (and in particular Figure \ref{fig_results_summaries}) show that it behaves similarly to the standard state-by-state selection method in the asymptotic regime and even has a smaller error for small number of observations. In addition, the dimension jumps are much sharper for this method than for the usual state-by-state selection method (see Figure \ref{fig_dim_jump}), which makes the calibration heuristics easier to use.

\subsection{Results}
\label{sec_results}

Figure \ref{fig_results} shows the evolution of the error $\|\hat{f}_k - f^*_k \|_2$ for each state $k$ with respect to the number of observations $n$, for all penalty calibration methods and all variants of the model selection procedure. Figure \ref{fig_results_summaries} compares the evolution of the median error for  the different calibration methods and for the different selection variants, and Figure \ref{fig_decr_example} compares two estimators with the oracle estimators.

\begin{figure}[!h]
\vspace{-1em}
\centering
\begin{subfigure}{0.32\textwidth}
\centering
\includegraphics[scale=0.53]{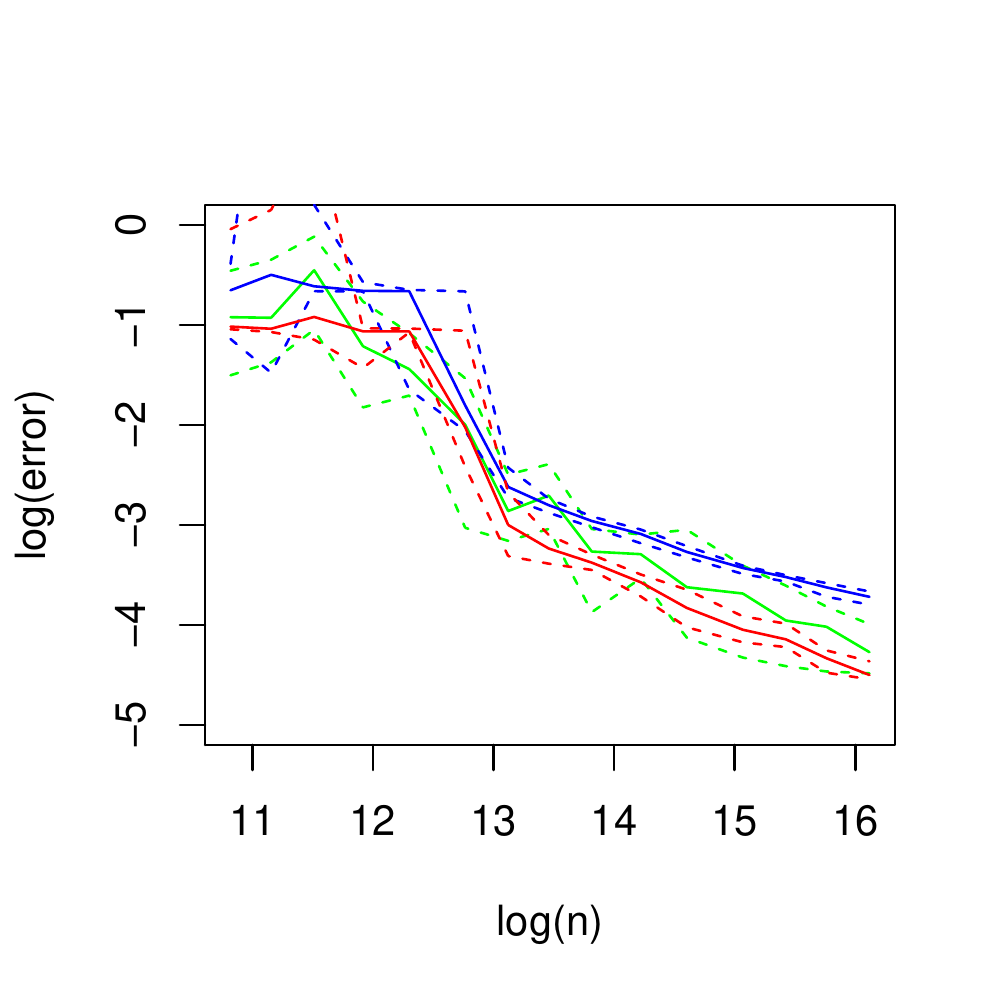} % Taille R : 4, 4
\vspace{-1em}
\caption{eachjump}
\end{subfigure}
\begin{subfigure}{0.32\textwidth}
\centering
\includegraphics[scale=0.53]{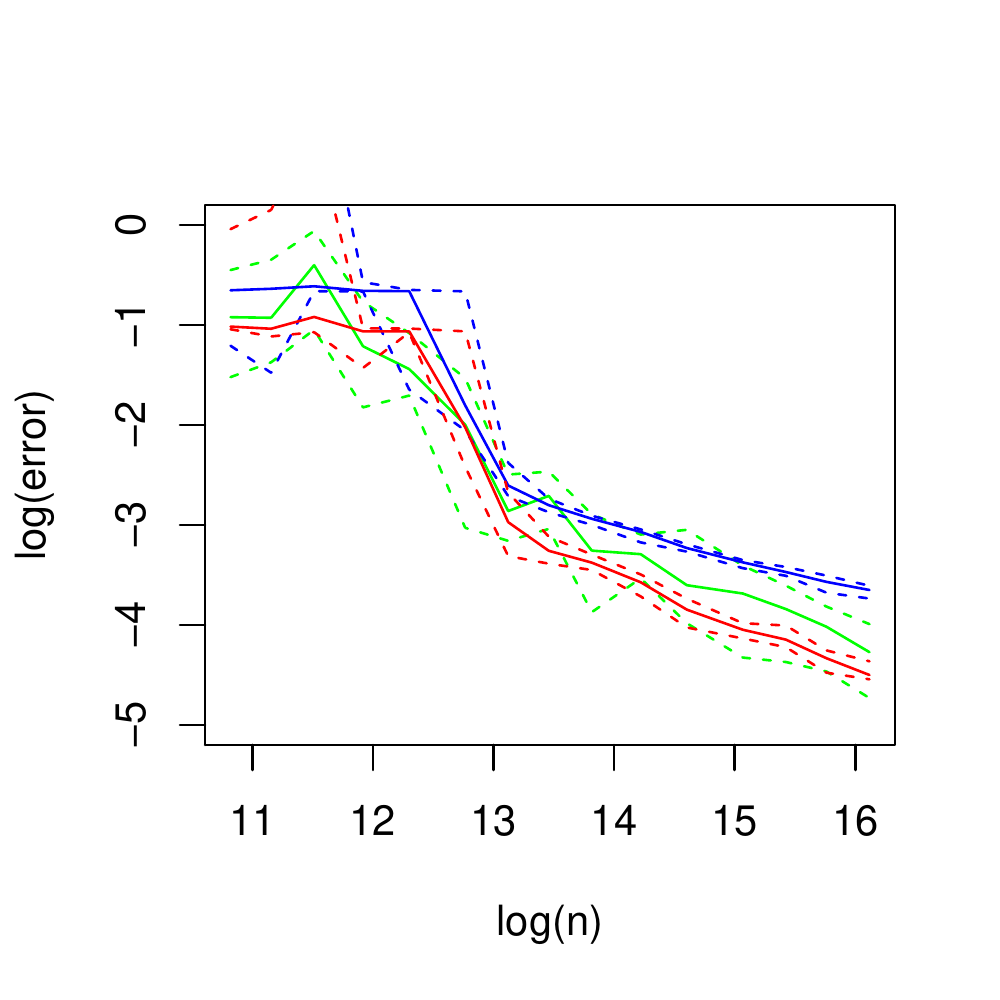}
\vspace{-1em}
\caption{eachjump POS}
\end{subfigure}
\begin{subfigure}{0.32\textwidth}
\centering
\includegraphics[scale=0.53]{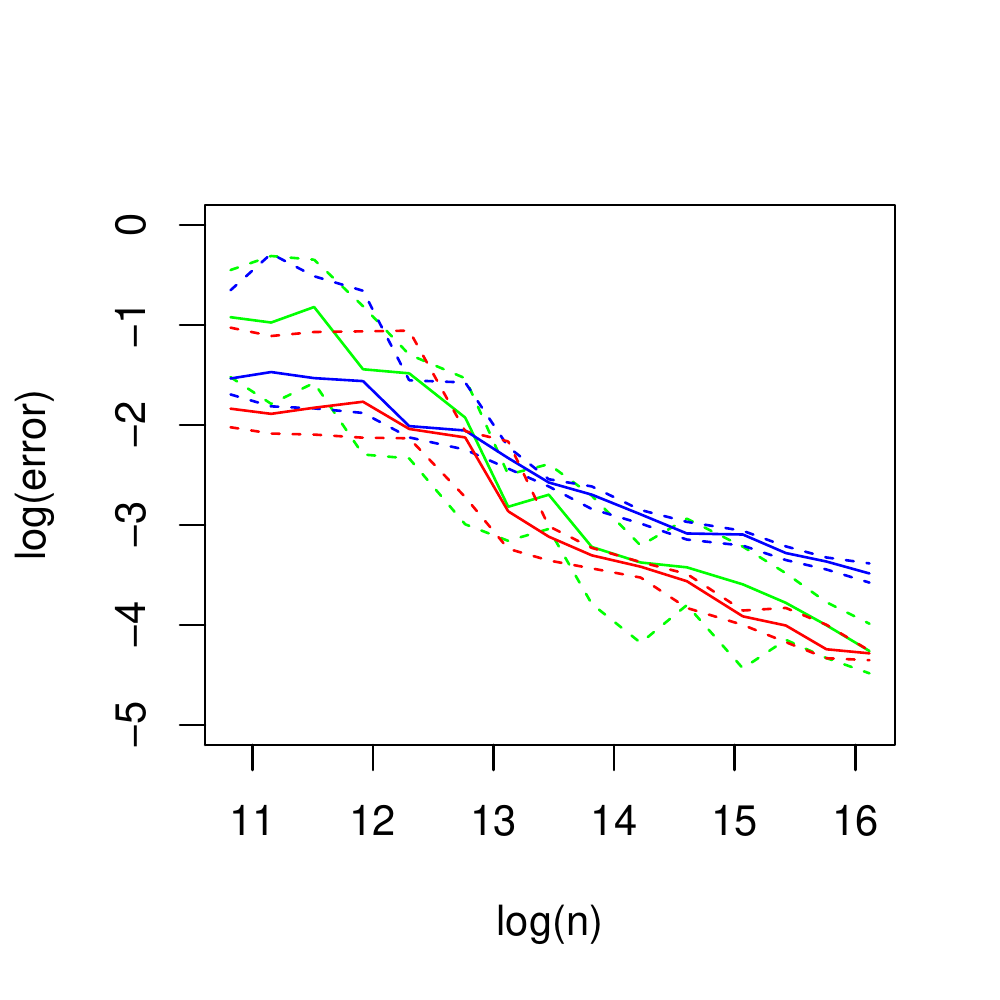}
\vspace{-1em}
\caption{eachjump MAX}
\end{subfigure}

\begin{subfigure}{0.32\textwidth}
\centering
\includegraphics[scale=0.53]{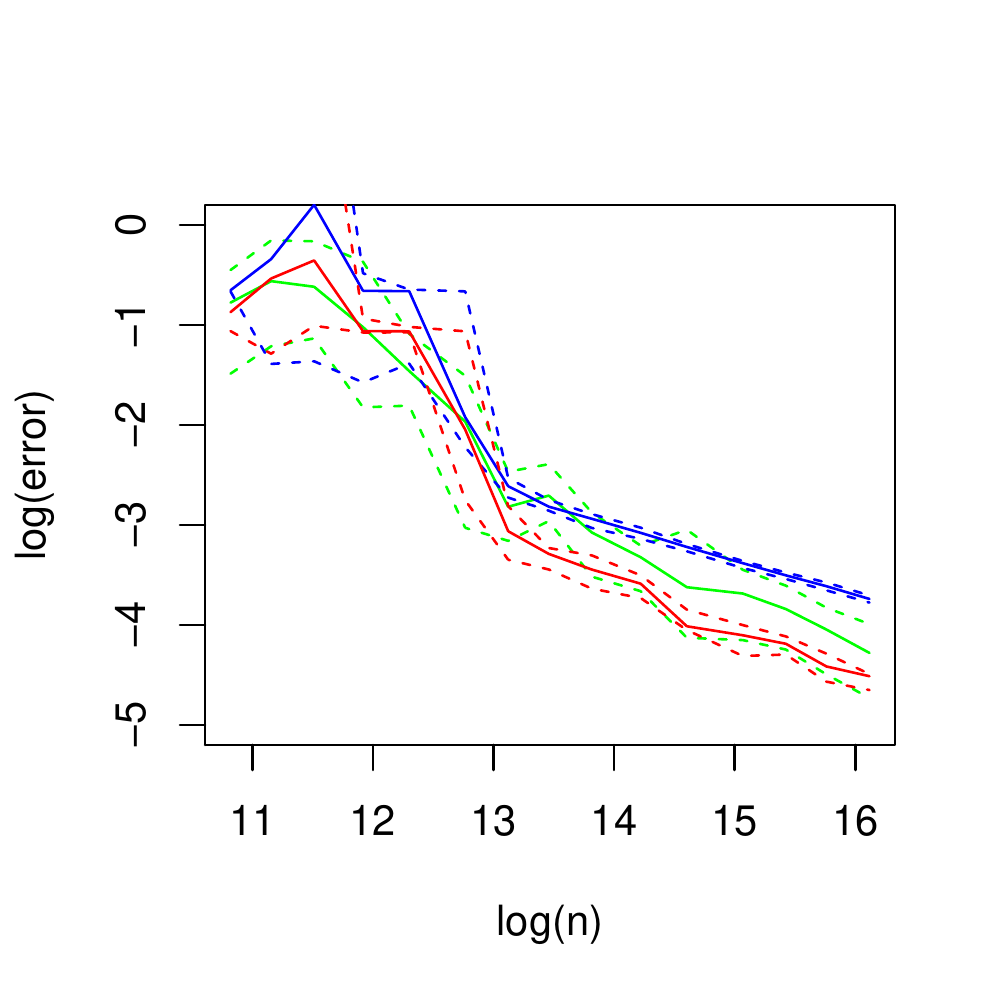}
\vspace{-1em}
\caption{jumpmean}
\end{subfigure}
\begin{subfigure}{0.32\textwidth}
\centering
\includegraphics[scale=0.53]{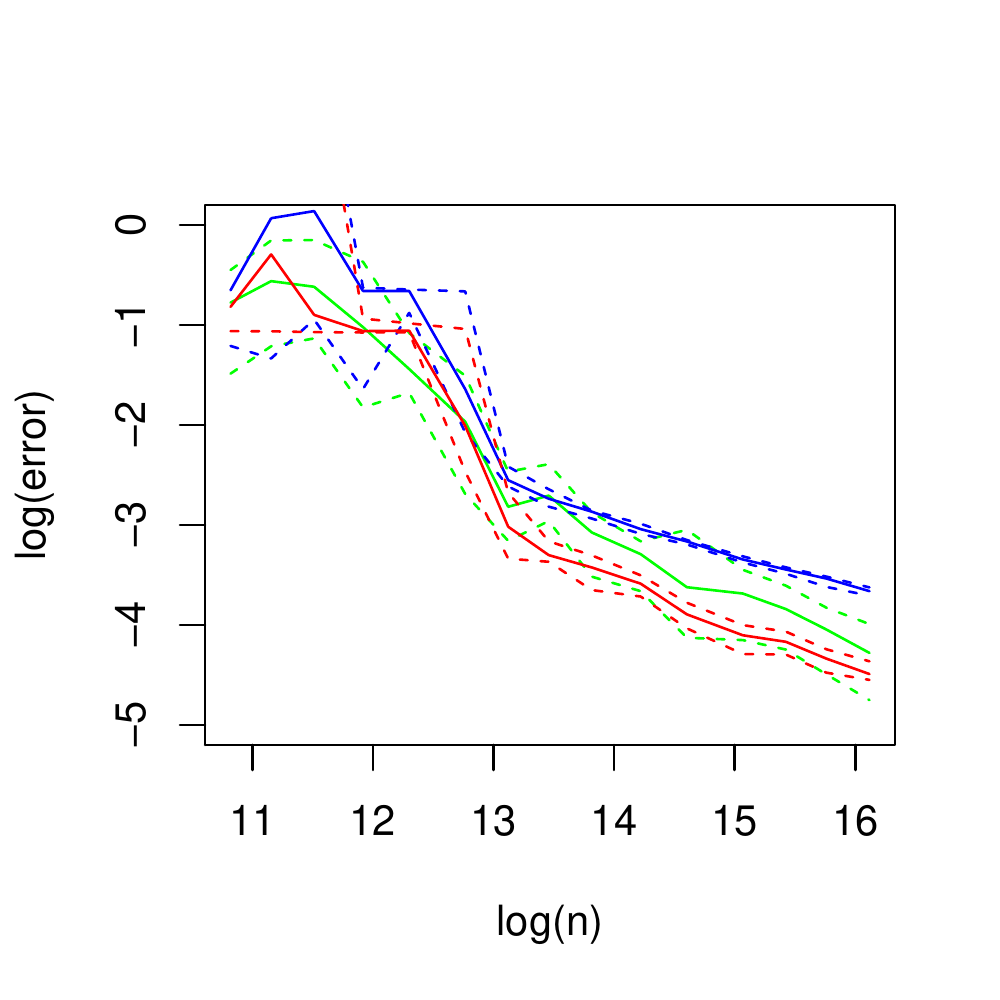}
\vspace{-1em}
\caption{jumpmean POS}
\end{subfigure}
\begin{subfigure}{0.32\textwidth}
\centering
\includegraphics[scale=0.53]{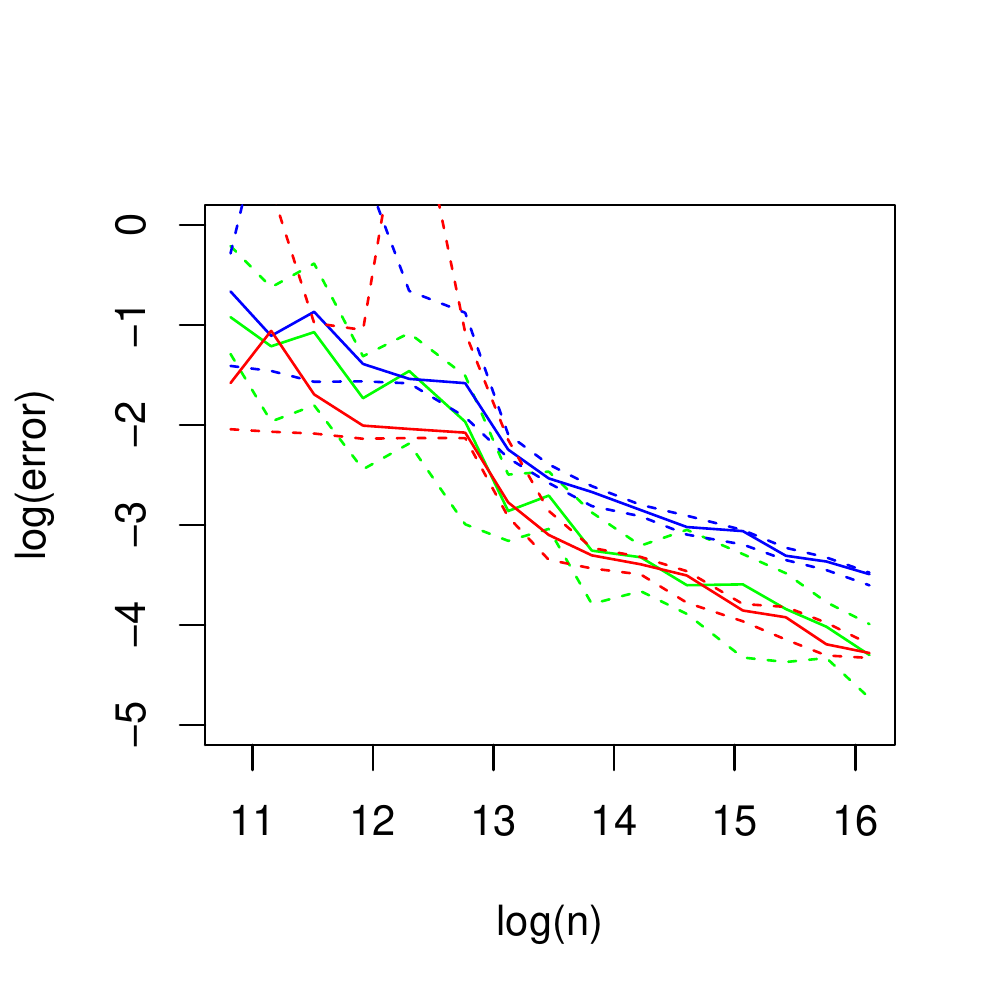}
\vspace{-1em}
\caption{jumpmean MAX}
\end{subfigure}

\begin{subfigure}{0.32\textwidth}
\centering
\includegraphics[scale=0.53]{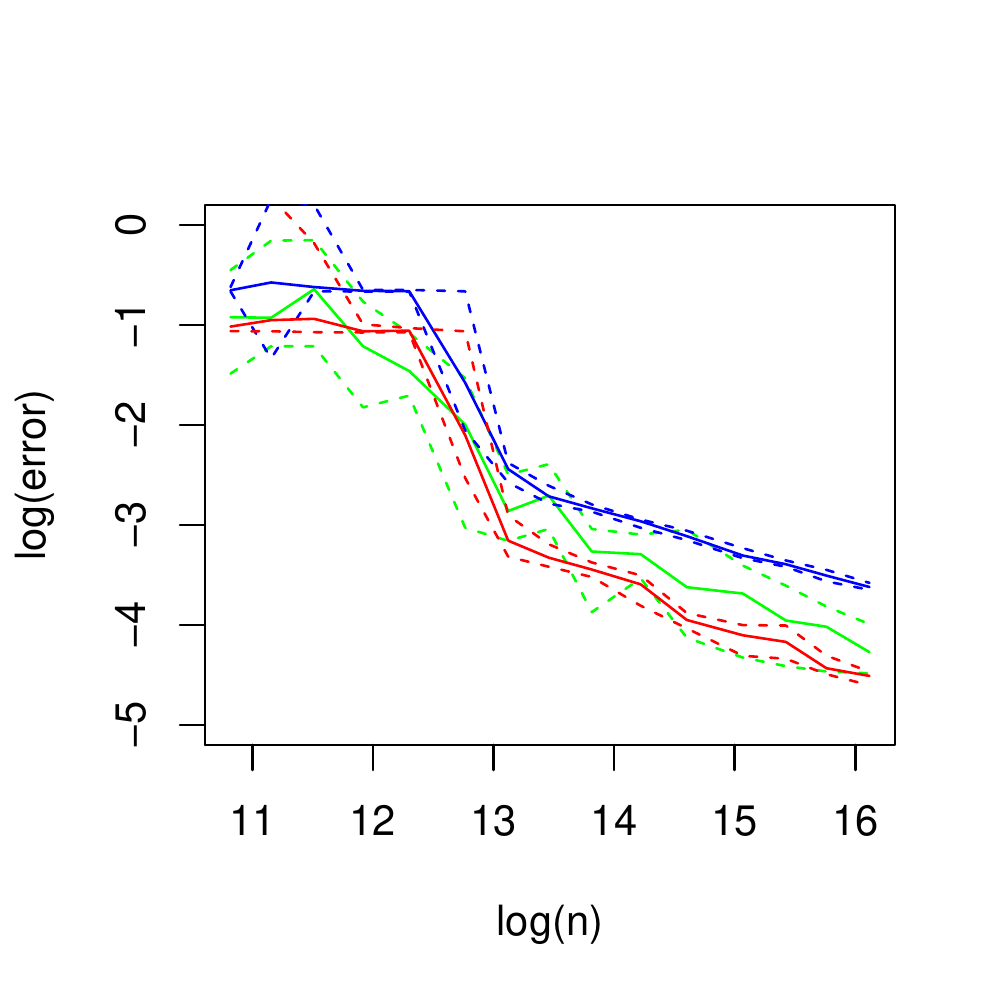}
\vspace{-1em}
\caption{jumpmax}
\end{subfigure}
\begin{subfigure}{0.32\textwidth}
\centering
\includegraphics[scale=0.53]{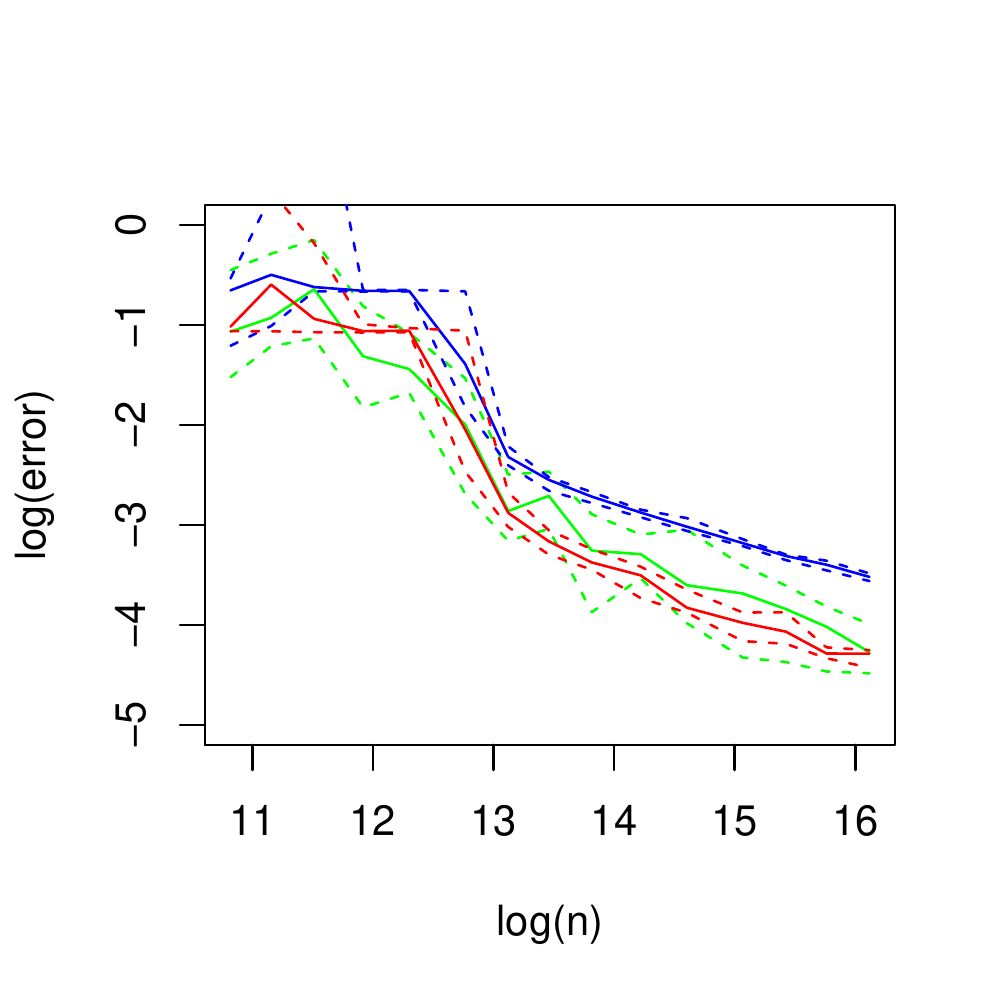}
\vspace{-1em}
\caption{jumpmax POS}
\end{subfigure}
\begin{subfigure}{0.32\textwidth}
\centering
\includegraphics[scale=0.53]{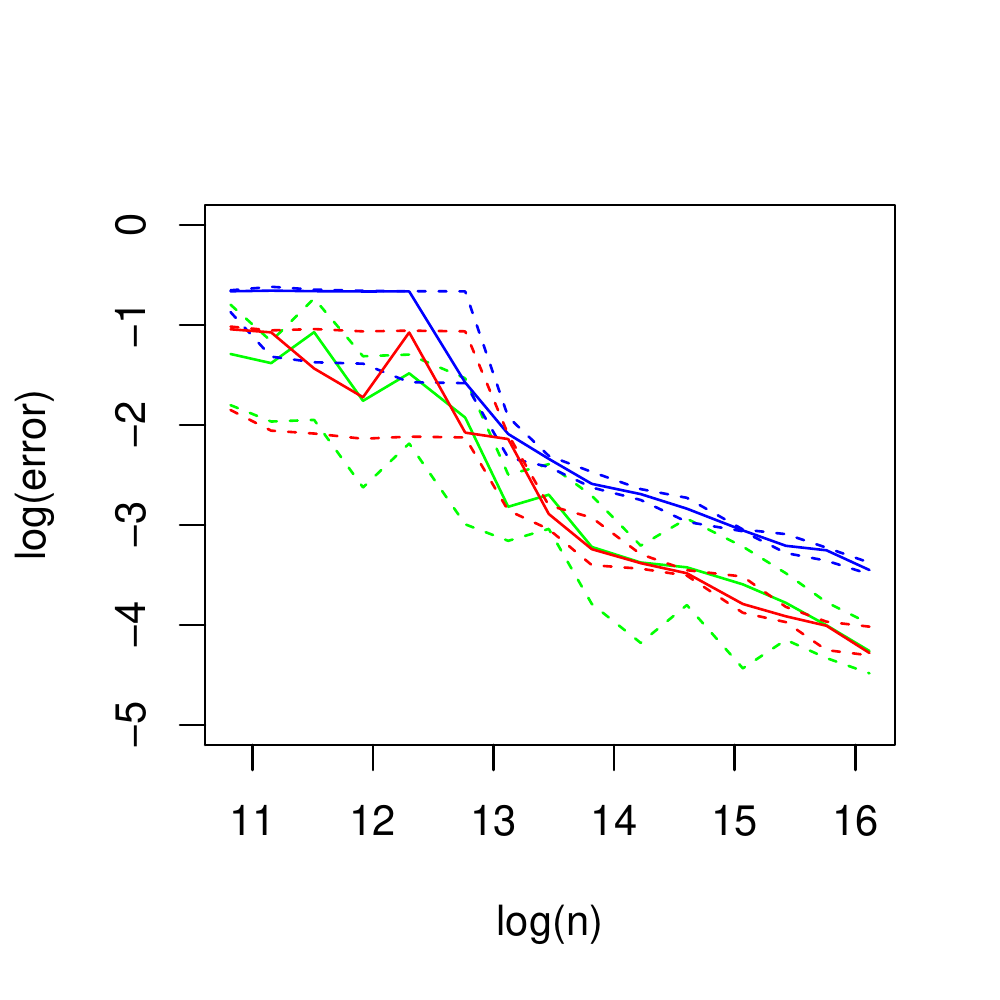}
\vspace{-1em}
\caption{jumpmax MAX}
\end{subfigure}

\caption{Logarithm of the $\Lbf^2$ error on each emission densities depending on the logarithm of the number of observations for each of the selection and calibration methods. Each color corresponds to one emission density. The full lines are the medians of the 20 observations and the dashed ones are the 25 and 75 percentiles.} % Sur 20 expériences, on a tracé les 5, 11 et 16.
\label{fig_results}
\end{figure}

\begin{figure}[!h]
\centering
\begin{subfigure}{0.32\textwidth}
\centering
\includegraphics[scale=0.55]{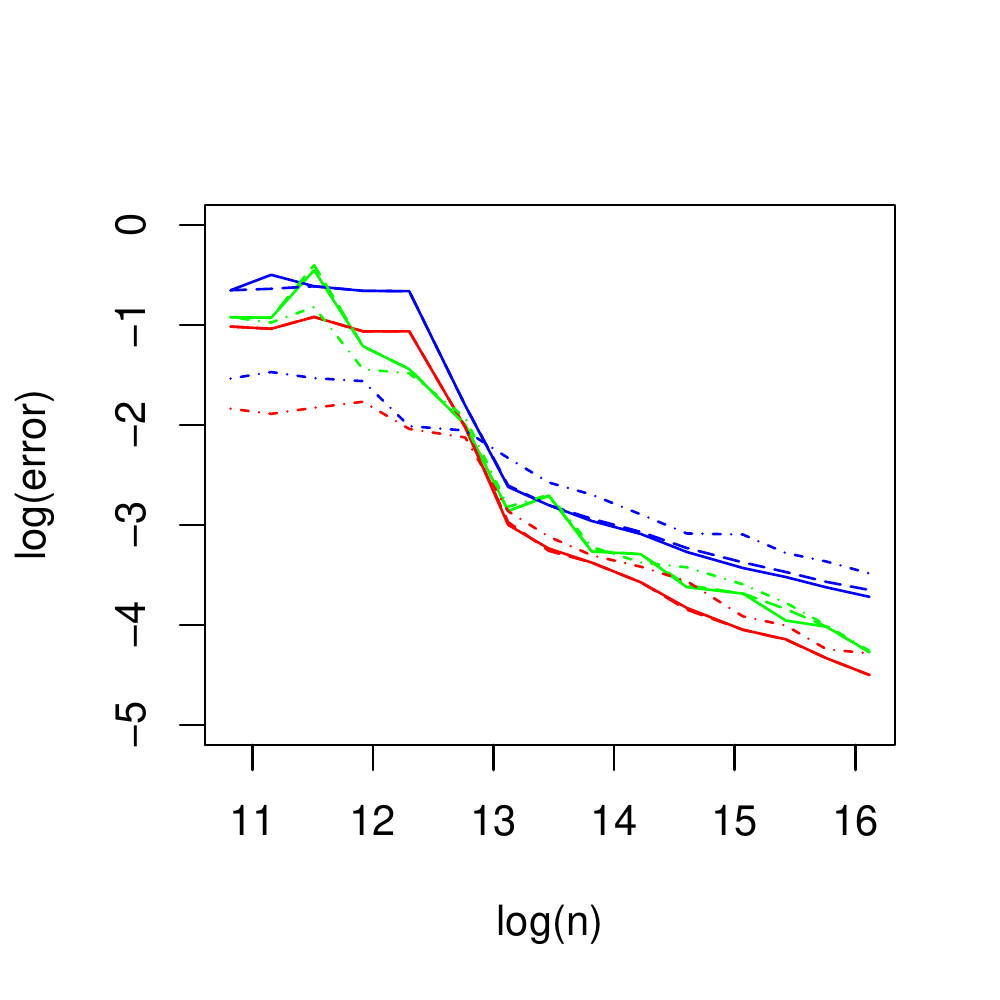}
\vspace{-1em}
\caption{eachjump summary}
\end{subfigure}
\begin{subfigure}{0.32\textwidth}
\centering
\includegraphics[scale=0.55]{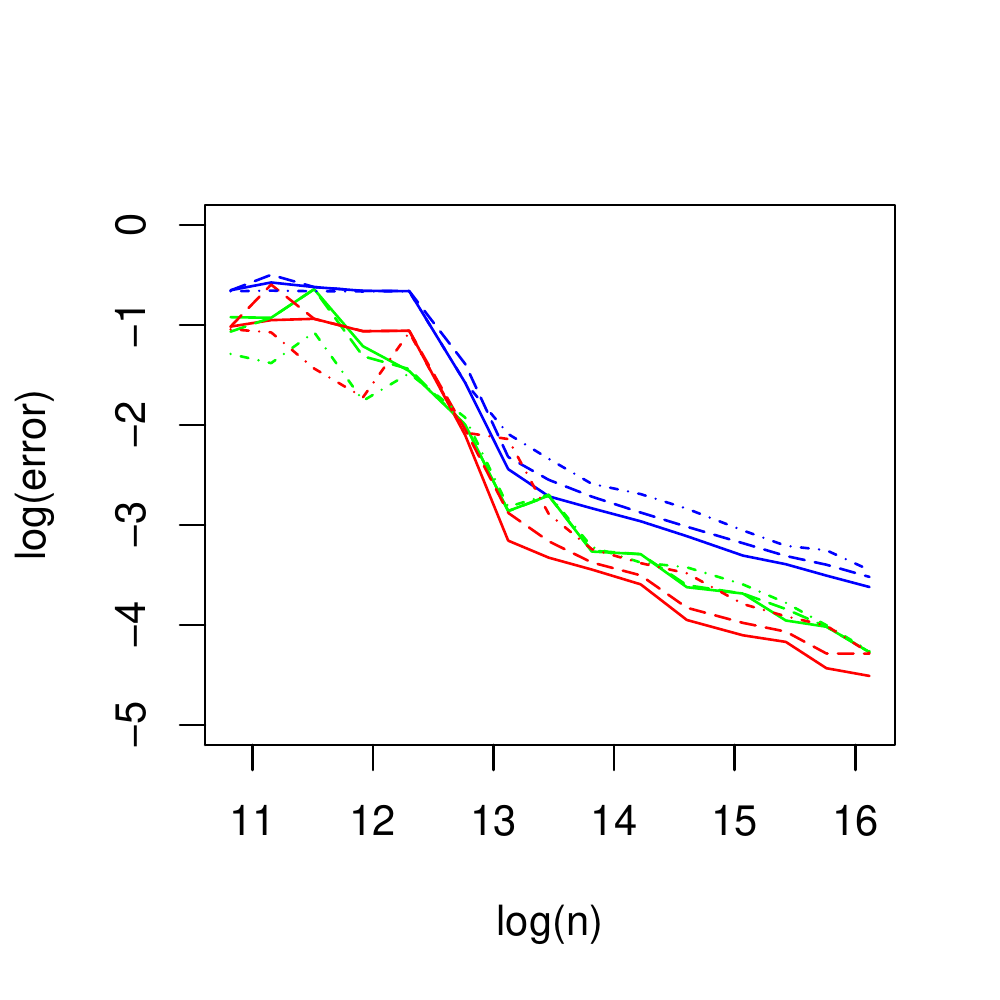}
\vspace{-1em}
\caption{jumpmax summary}
\end{subfigure}
\begin{subfigure}{0.32\textwidth}
\centering
\includegraphics[scale=0.55]{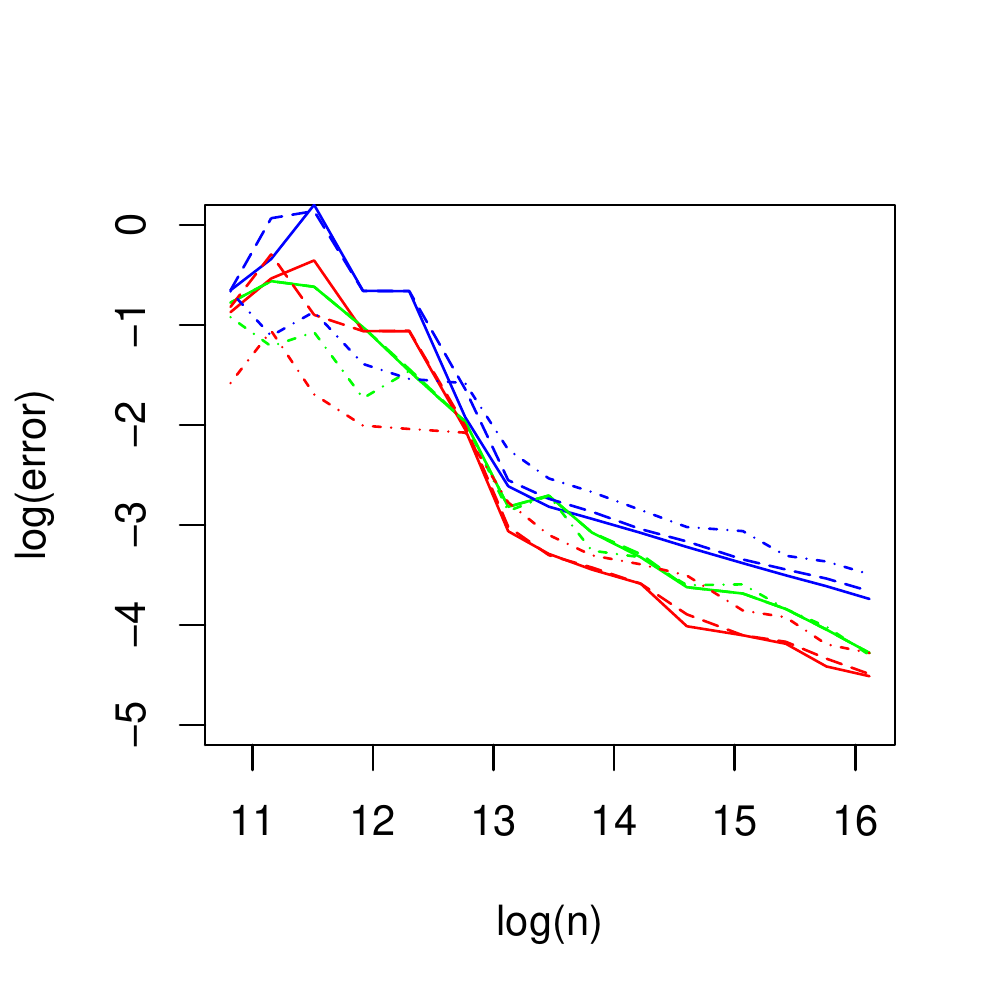}
\vspace{-1em}
\caption{jumpmean summary}
\end{subfigure}

\begin{subfigure}{0.32\textwidth}
\centering
\includegraphics[scale=0.55]{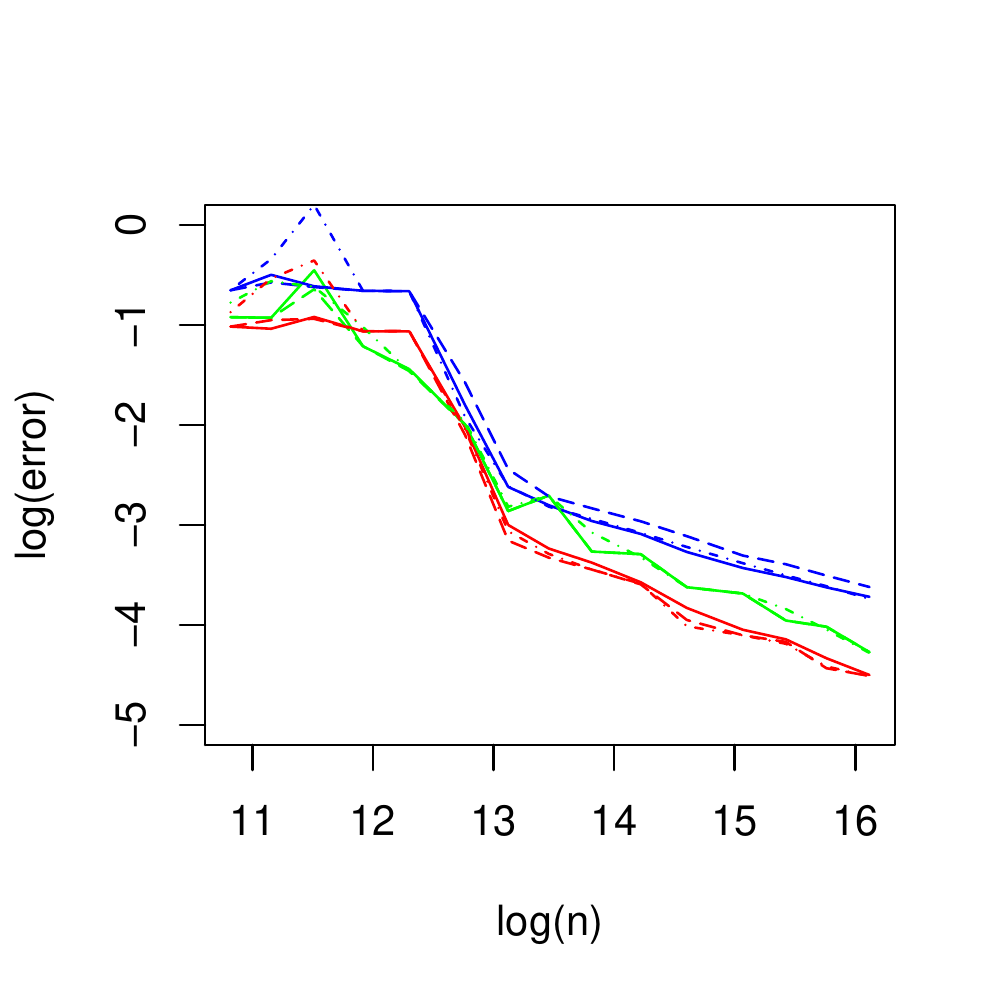}
\vspace{-1em}
\caption{none summary}
\end{subfigure}
\begin{subfigure}{0.32\textwidth}
\centering
\includegraphics[scale=0.55]{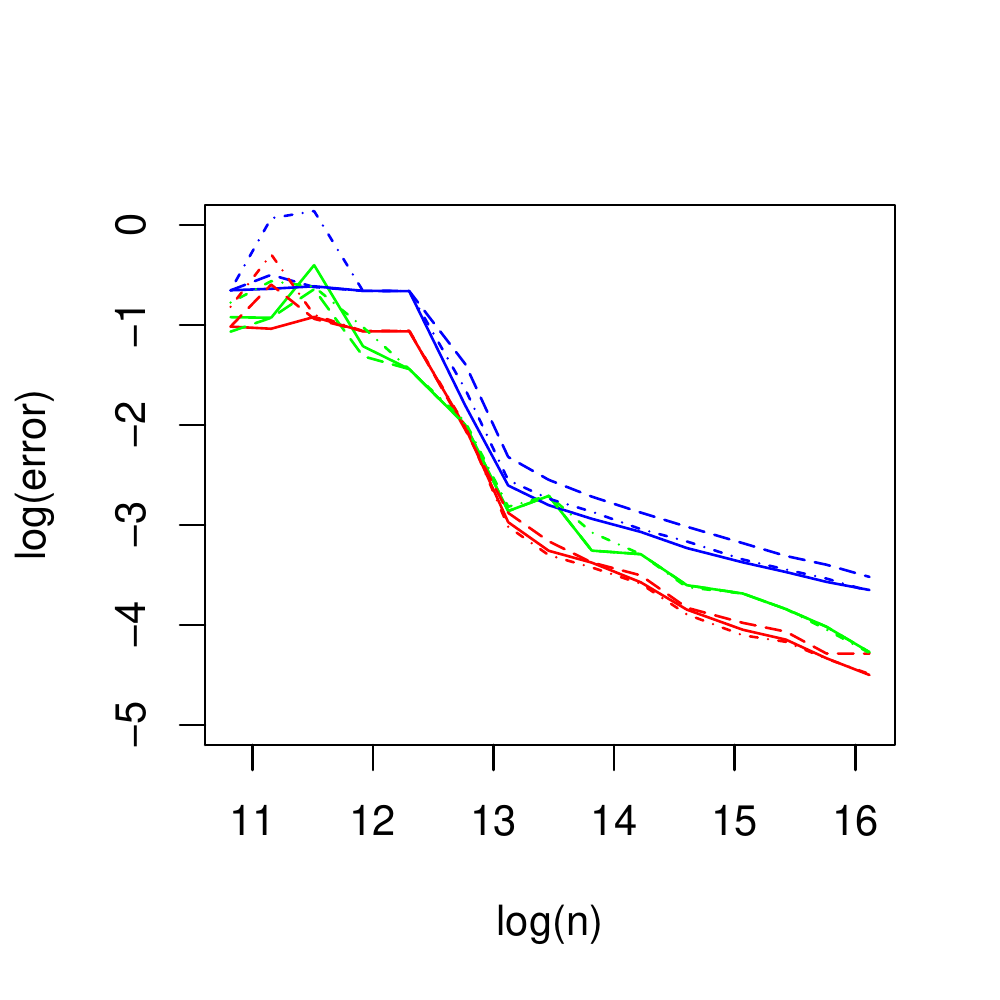}
\vspace{-1em}
\caption{POS summary}
\end{subfigure}
\begin{subfigure}{0.32\textwidth}
\centering
\includegraphics[scale=0.55]{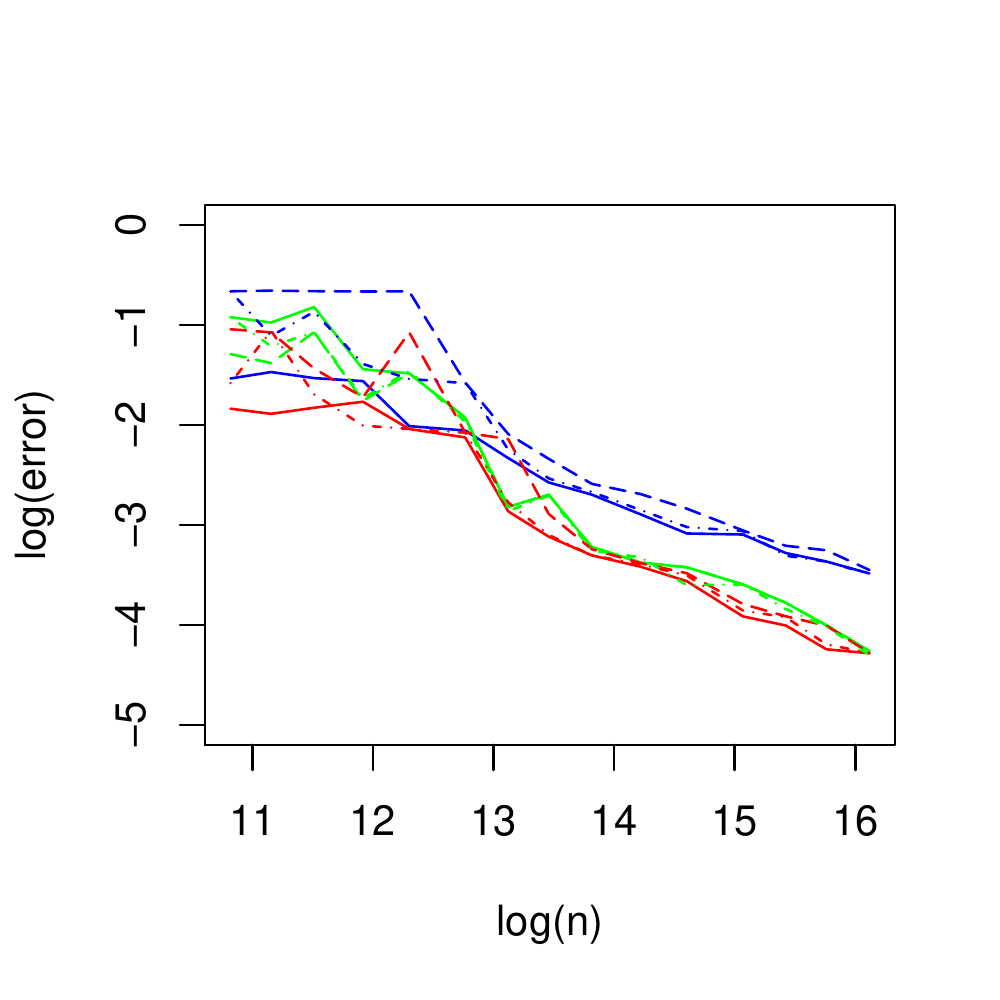}
\vspace{-1em}
\caption{MAX summary}
\end{subfigure}

\caption{Superposition of the median lines of Figure \ref{fig_results} by selection method and by calibration variant. Each color corresponds to one emission density. In Subfigures (a)-(c), the full lines correspond to the basic selection method, the dashed ones to the POS method and the dotted ones to the MAX method. In Subfigures (d)-(f), the full lines correspond to the eachjump method, the dashed ones to the jumpmax method and the dotted ones to the jumpmean method.} % On a respectivement : plein pour rien, longdash pour POS et dotdash pour MAX, puis eachjump / jumpmax / jumpmean.
\label{fig_results_summaries}
\end{figure}

\begin{figure}
\vspace{-2em}
\centering
\includegraphics[scale=0.55]{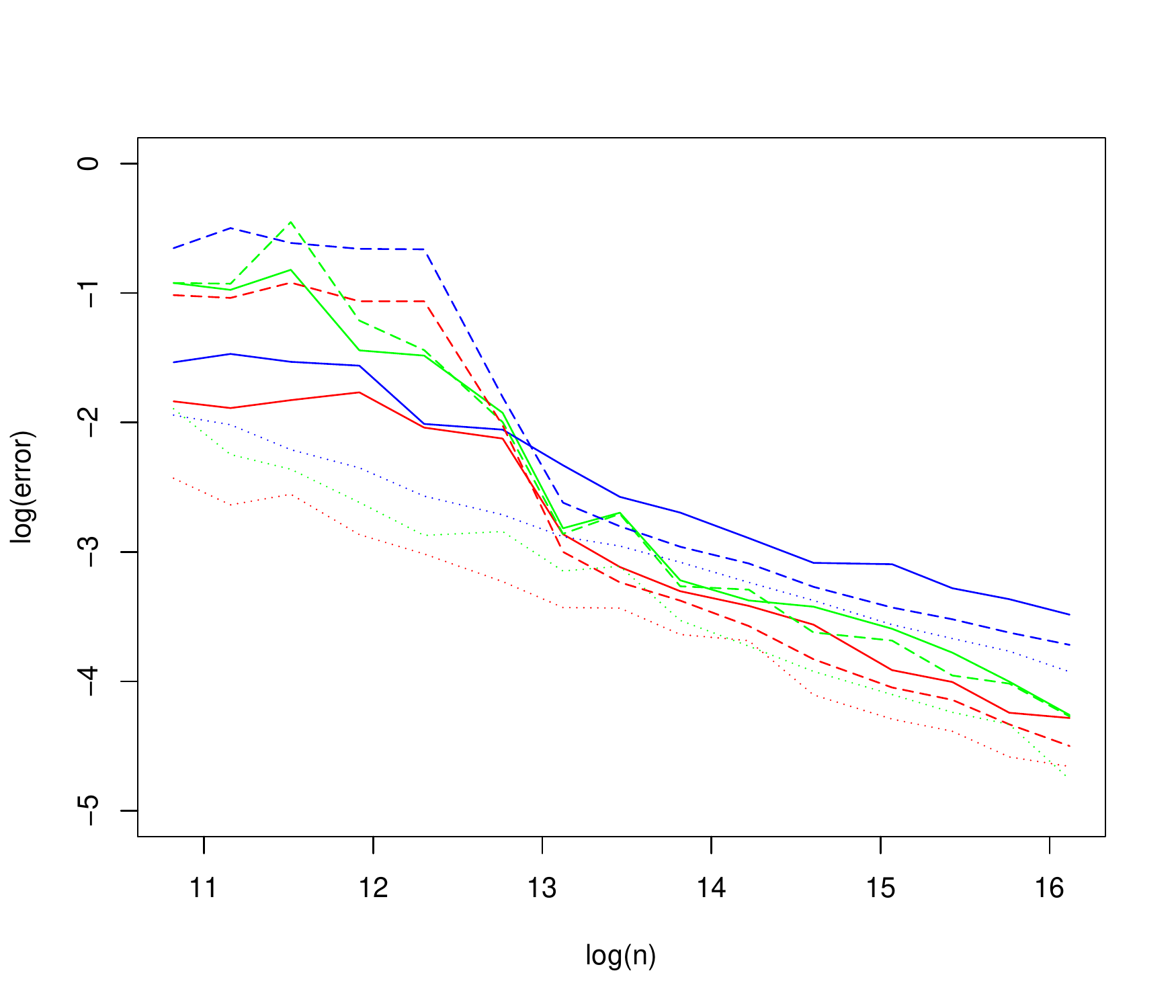}

\vspace{-6.5cm}
\hspace{6cm}\includegraphics[scale=0.15]{legende.png}
\vspace{4.5cm}
\caption{Comparison of the errors for the eachjump MAX method (full lines), for the eachjump method (dashed lines) and for the oracle estimators (dotted lines). For each $k$, the oracle estimator is defined as $\hat{f}^{(M^\text{oracle}_k)}_k$ where $M^\text{oracle}_k$ minimizes $M \longmapsto \| \hat{f}^{(M)}_k - f^*_k \|_2$. The oracle corresponds to the best estimator one could possibly select among the preliminary estimator if the true densities were known.}
\label{fig_decr_example}
\end{figure}

When the number of observations $n$ is large enough, the logarithm of the error decreases linearly with respect to $\log(n)$. This corresponds to the asymptotic convergence regime: the error is expected to decrease as a power of the number of observations $n$ when $n$ tends to infinity. The corresponding slopes are listed in Table \ref{table_convergence_rates_summary}.

\begin{table}[!h]
\centering
\renewcommand{\arraystretch}{1.3}
\begin{tabular}{l|c|c|c}
\hline
\multicolumn{1}{|c|}{\multirow{2}{*}{Estimator}} & \multicolumn{3}{c|}{Convergence rate exponents} \\ \cline{2-4}
\multicolumn{1}{|c|}{} & Uniform & Sym. Beta & \multicolumn{1}{c|}{Beta} \\ \hhline{====}
Eachjump     & $-0.500 \pm 0.046$ & $-0.347 \pm 0.007$ & $-0.470 \pm 0.015$ \\ %\hline
Eachjump POS & $-0.503 \pm 0.047$ & $-0.327 \pm 0.008$ & $-0.469 \pm 0.015$ \\ %\hline
Eachjump MAX & $-0.480 \pm 0.052$ & $-0.335 \pm 0.009$ & $-0.449 \pm 0.015$ \\ %\hhline{====}
\hline
Jumpmean     & $-0.532 \pm 0.048$ & $-0.349 \pm 0.006$ & $-0.471 \pm 0.017$ \\ %\hline
Jumpmean POS & $-0.540 \pm 0.048$ & $-0.350 \pm 0.006$ & $-0.456 \pm 0.016$ \\ %\hline
Jumpmean MAX & $-0.493 \pm 0.049$ & $-0.374 \pm 0.009$ & $-0.437 \pm 0.015$ \\ %\hhline{====}
\hline
Jumpmax      & $-0.500 \pm 0.046$ & $-0.349 \pm 0.006$ & $-0.464 \pm 0.016$ \\ %\hline
Jumpmax POS  & $-0.492 \pm 0.046$ & $-0.358 \pm 0.006$ & $-0.442 \pm 0.015$ \\ %\hline
Jumpmax MAX  & $-0.480 \pm 0.052$ & $-0.404 \pm 0.009$ & $-0.466 \pm 0.015$ \\ \hhline{====}
Cross Validation & $-0.434 \pm 0.007$ & $-0.263 \pm 0.011$ & $-0.377 \pm 0.008$ \\ \hhline{====}
Oracle       & $-0.517 \pm 0.048$ & $-0.360 \pm 0.006$ & $-0.459 \pm 0.017$ \\ \hline
Hidden states known & $-0.526 \pm 0.031$ & $-0.293 \pm 0.005$ & $-0.428 \pm 0.007$ \\ \hline
Minimax (Hölder) & $-0.5$	 & $-3/11 \approx -0.273$  & $-3/7 \approx -0.429$ \\ \hline
\end{tabular}
\caption{Exponents of the rates of convergence for the different algorithms. The rates are obtained from a linear regression with the relation $\log(\| \hat{f}_k - f^*_k \|_2) \sim \log(n)$ for the estimators $\hat{f}_k$ computed with $n \geq 700,000$ observations ($n \geq 1,000,000$ for the cross validation estimators from Section \ref{sec_VC}). The smaller the exponent, the faster the estimators converge. The line ``hidden states known" is obtained by density estimation when the hidden states are observed.}
\label{table_convergence_rates_summary}
\end{table}

For each state, the confidence intervals of the rates of all estimators\---including the oracle estimators\---have a common intersection (except for the symmetrized Beta distribution in the jumpmax MAX variant, whose estimators seem to converge faster than the others). This tends to confirm that the calibration and selection variants are asymptotically equivalent. This phenomenon is also visible in Figures \ref{fig_results} and \ref{fig_results_summaries}: in the asymptotic regime, the errors decrease in a similar way for all methods.

Furthermore, the rates of convergence are clearly distinct. The uniform distribution is estimated with a rate of convergence of approximately $n^{-1/2}$, which is also the best possible rate (it corresponds to a parametric estimation rate). In comparison, the rate of convergence for the symmetrized Beta distribution is much slower (around $n^{-0.36}$). This shows that the algorithm effectively adapts to the regularity of each state and that one irregular emission density does not deteriorate the rates of convergence of the other densities.

Note that the above rates are in accordance with the minimax rates as far as the Hölder regularity is concerned.
The minimax Hölder rate for the symmetrized Beta (which is $0.6$-Hölder) is $n^{-3/11}$, or approximately $n^{-0.27}$, which means our estimator converges faster than the minimax rate would suggest.
The minimax Hölder rate for the Beta distribution (which is $3$-Hölder) is $n^{-3/7}$, or approximately $n^{-0.43}$, which is around the observed value.

\subsection{Comparison with cross validation}
\label{sec_VC}

In this section, we use a cross validation procedure based on our spectral estimators to check whether our method actually improves estimation accuracy.œ

When estimating a density by taking an estimator within some class (the model), two sources of error appear: the bias, that is the (deterministic) distance between the true density and the model, and the variance, that is the (random) error of the estimation within the model. Small models will have a large bias but a small variance, while large models will have a small bias and a large variance.
The core issue of model selection is to select a model that minimizes the total error, that is large enough to accurately describe the true densities and small enough to prevent overfitting: in other words, perform a bias-variance tradeoff.

Cross validation seeks to achieve such a tradeoff by computing an estimate of the total error. This is done by splitting the sample into two sets, the training sample being used for the calibration of the estimator and the validation sample for measuring the error. Taking the mean of these errors for different splits between training and validation samples provides an estimator of the total error. This method has become popular for its simplicity of use. We refer to the survey of \cite{arlot2010survey} for an overview on this method and its guarantees.

\subsubsection{Risk}

We use the least squares criterion of Algorithm~\ref{alg:LS} to quantify the error of the estimators.
Since the guarantees on spectral estimators rely on the $\Lbf^2$ norm, a least squares criterion is more natural than the likelihood. In addition, the spectral estimators might take negative values depending on the orthonormal basis, which is not a problem as far as $\Lbf^2$ error is concerned but can be an issue for the likelihood.

Let us first recall this criterion. Given an orthonormal basis $(\varphi_i)_{i \in \Nbb}$ of $\Lbf^2(\Ycal, \mu)$, define the coordinate tensor of the empirical distribution of the triplet $(Y_1, Y_2, Y_3)$ on this basis by
\begin{equation*}
\hat{\Mbf}(a,b,c):= \frac{1}{n} \sum_{s=1}^{n}\varphi_{a}(Y_{s})\varphi_{b}(Y_{s+1})\varphi_{c}(Y_{s+2}) \quad \text{ for all }a, b, c \in \Nbb.
\end{equation*}

Given a transition matrix $\Qbf$ of size $K$, a stationary distribution $\pi$ of $\Qbf$ and a vector of densities $\fbf = (f_1, \dots, f_K)$, define the coordinate matrix $\Obf$ of $\fbf$ by $\Obf(b,k) = \langle \varphi_b, f_k \rangle$. Let $\Mbf_{(\pi, \Qbf, \fbf)}$ be the coordinate tensor of the distribution of $(Y_1, Y_2, Y_3)$ under the parameters $(\pi, \Qbf, \fbf)$, that is
\begin{equation*}
\Mbf_{(\pi, \Qbf, \fbf)}(\cdot, b, \cdot) = \Obf \Diag[\pi] \Qbf \Diag[\Obf(b,\cdot)] \Qbf \Obf^\top \quad \text{ for all } b \in \Nbb.
\end{equation*}

The empirical least squares criterion is $\| \Mbf_{(\pi, \Qbf, \fbf)} - \hat{\Mbf} \|_F^2$. It corresponds to the $\Lbf^2$ error between the empirical distribution of three consecutive observations and the theoretical distribution under the parameters $(\pi, \Qbf, \fbf)$.

\subsubsection{Implementation}

We use 10-fold cross validation, that is we split the sequence into 10 segments of same size $I_1, \dots, I_{10}$. In order to avoid interferences between samples, we prune the ends of each segment, so that the observations in each segment can be considered independent. In practice, we take a gap of 30 observations between two segments.

We ran 150 simulations, 10 per value of $n$, with the same parameters as in Section \ref{sec_sim_setting}. Each simulation is as follows.

For each segment $I_j$, we run the spectral algorithm on all models $\Pfrak_M$ for $\verb?M_min? \leq M \leq \verb?M_max?$ using only the observations from the other segments. The transition matrix is estimated using an additional step of the spectral method which is adapted from Steps 8 and 9 of Algorithm 1 of \cite{dCGLLC15}. Then, we compute the least squares criterion for the estimated parameters using the segment $I_j$ as observed sample. Finally, for each $M$, we average this error on all segments $I_j$, which gives the least squares cross validation error $E_\text{VC}(M)$.

This cross validation criterion is used to select one model $\hat{M}_\text{VC} \in \argmin_M E_\text{VC}(M)$, from which we construct the final estimators of the emission densities $\hat{f}_k = \hat{f}_k^{(\hat{M}_\text{VC})}$ for all $k$. Note that the selected model is the same for all emission densities.

\subsubsection{Results}

Figure~\ref{fig_VC_selected_models} compares the selected model dimensions for each $n$ using our state-by-state selection method and using the cross validation method. When the number of observations $n$ becomes larger than $10^6$, the cross validation tends to always pick the largest model, which means that it does not prevent overfitting as well as our method.

The $\Lbf^2$ errors on the emission densities are shown in Figure~\ref{fig_VC_error}. It appears that the cross validation has a lower error for small $n$ ($n \leq 350,000$) than our method.
However, for larger values of $n$, the errors becomes larger than the ones of our method (see Figure~\ref{fig_decr_example}) by up to one order of magnitude, and only start decreasing once the selected model is set to the maximum dimension.

Finally, the estimated rates of convergence are shown in Table~\ref{table_convergence_rates_summary}. Our state-by-state method outperforms the cross validation method for all emission densities. The cross validation estimators only reach the minimax rate of convergence for the less regular density: the symmetrized Beta, and even then they converge slower than the state-by-state estimator. All other emission densities are estimated slower than their minimax rate.

\begin{figure}
\vspace{-2em}
\centering
\includegraphics[scale=0.6]{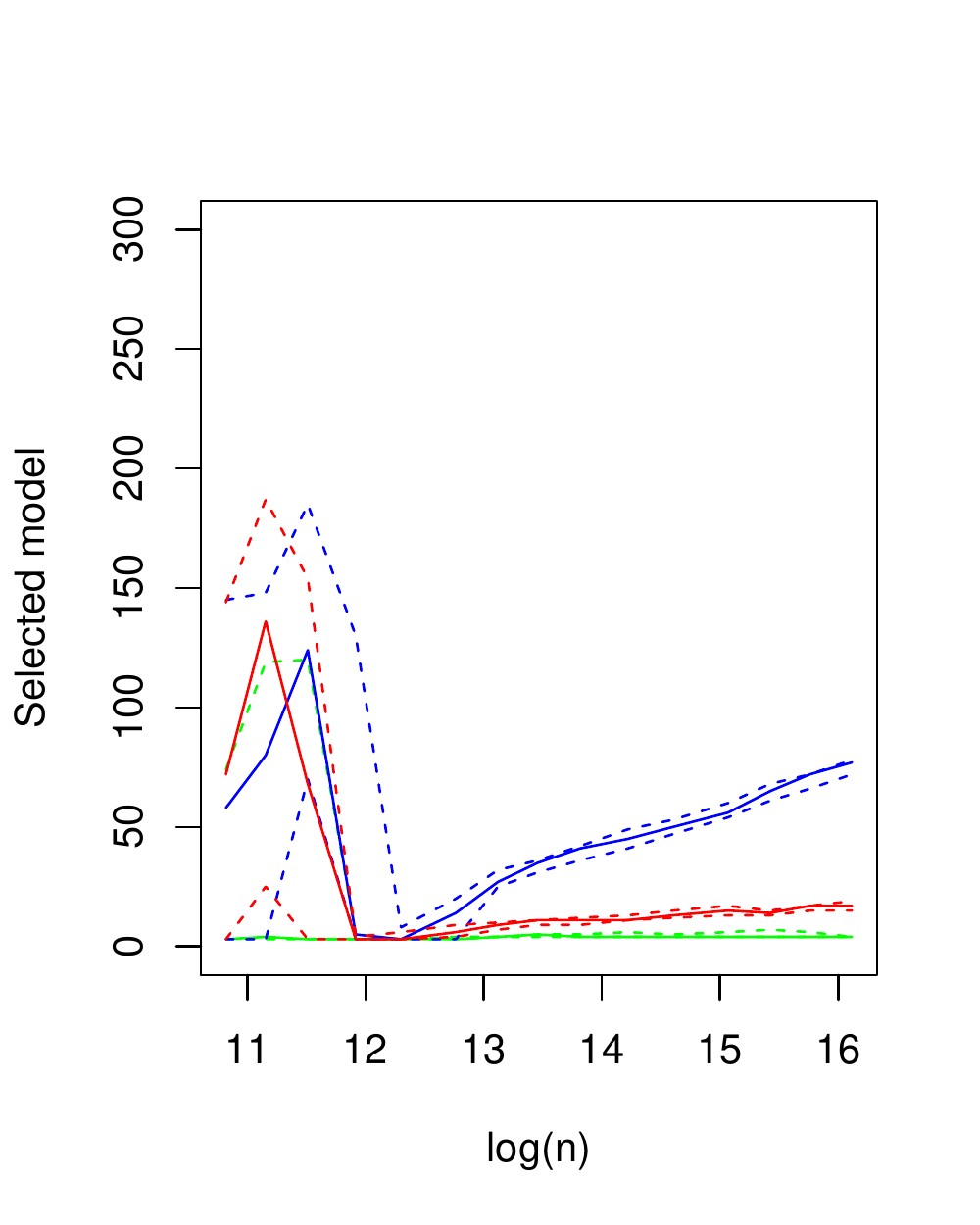}
\includegraphics[scale=0.6]{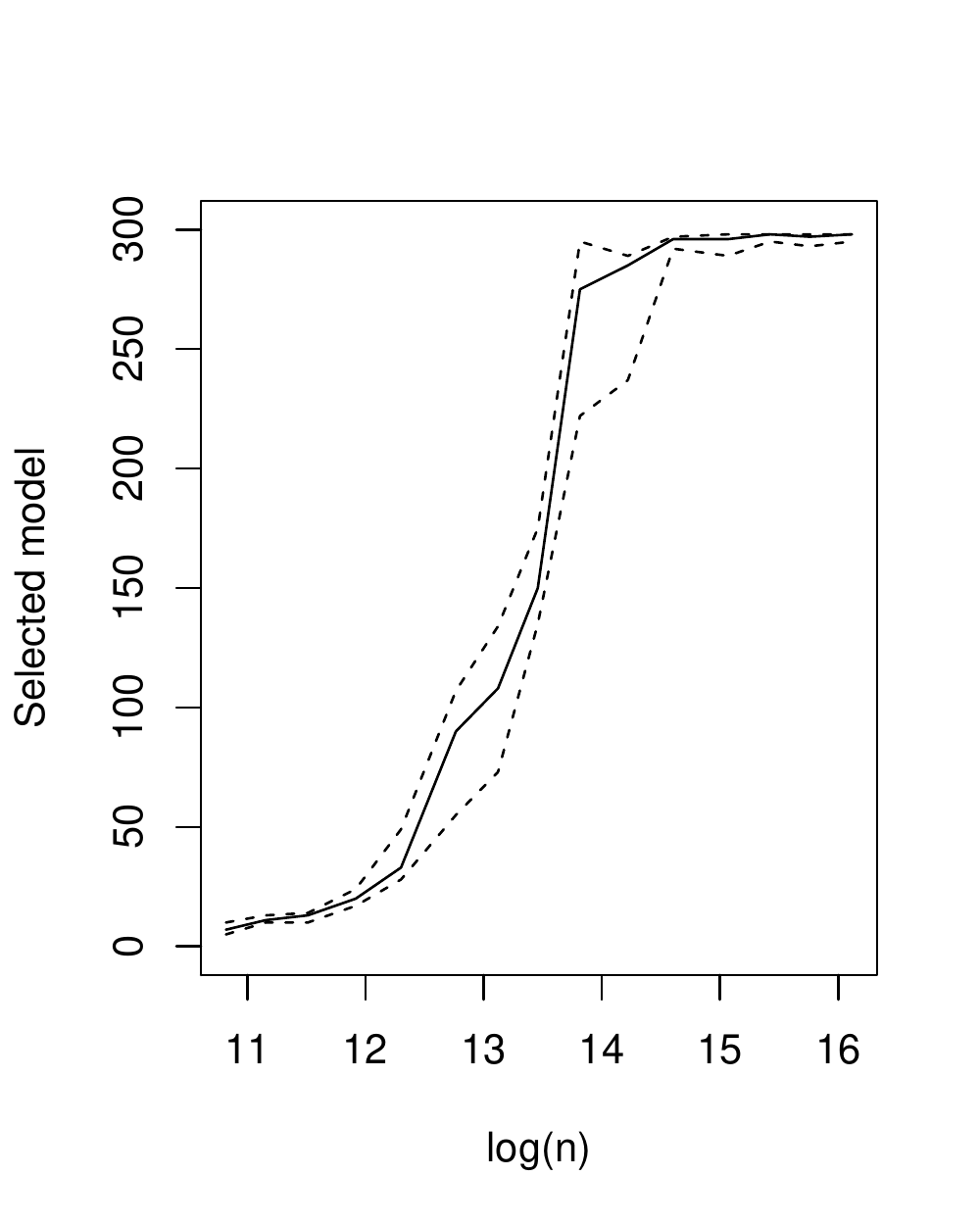}
\vspace{-1em}
\caption{Selected model dimensions for each $n$ using our state-by-state selection method (left) and 10-fold cross validation (right). The full lines are the median model dimensions and the dashed lines are the 25 and 75 percentiles.}
\label{fig_VC_selected_models}
\end{figure}

\begin{figure}
\vspace{-2em}
\centering
\includegraphics[scale=0.6]{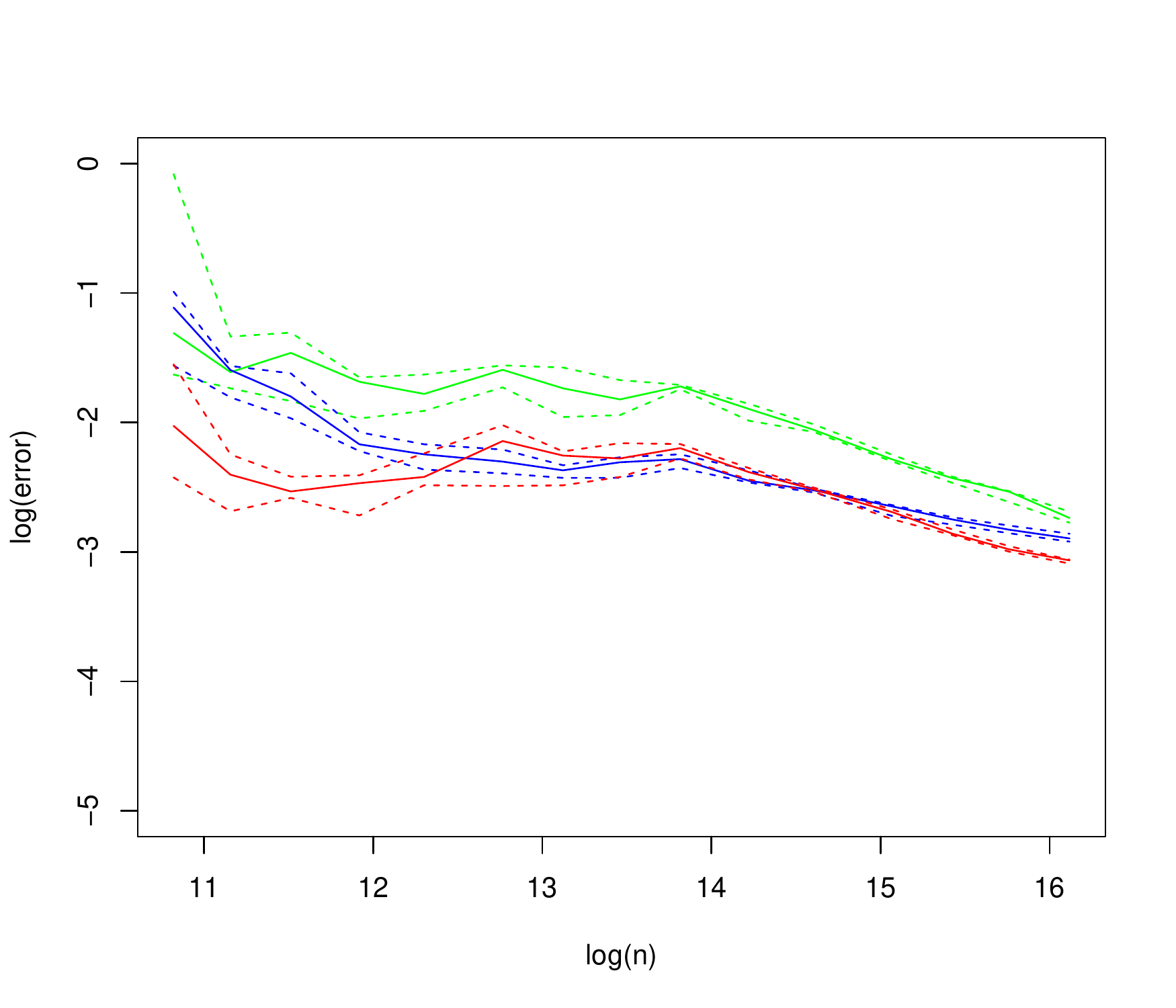}

\vspace{-7.5cm}
\hspace{6cm}\includegraphics[scale=0.15]{legende.png}
\vspace{6cm}
\caption{Error of the cross validation estimators for each $n$ using 10-fold cross validation. The full lines are the median errors for each density and the dashed lines are the 25 and 75 percentiles.}
\label{fig_VC_error}
\end{figure}

\subsection{Algorithmic complexity}
\label{sec_complexity}

In the following, we treat $K$ as a constant as far as the algorithmic complexity is concerned. The different complexities are summarized in Table \ref{table_complexity_summary}.

\subsubsection{Spectral algorithm (see Section \ref{sec_spectral_th})}

We consider the algorithmic complexity of estimating the emission densities for all models $M$ such that $M_{\min} \leq M \leq M_{\max}$ with $n$ observations and auxiliary parameters $r$ and $m$ depending on $n$ and $M$ (upper bounded by $m_{\max}$ and $r_{\max}$).

Step 1 can be computed for all models with $O(n M_{\max} m_{\max}^2)$ operations. It is the only step whose complexity depends on $n$. Steps 2 and 3 require $O(m^3 M)$ operations for each model and Steps 4 to 7 require $O(M r)$ operations for each model, for a total of $O(n M_{\max} m_{\max}^2 + M_{\max}^2 m_{\max}^3 + M_{\max}^2 r_{\max})$ operations.

In practice, one takes $m \propto \log(n)$, $r \propto \log(n) + \log(M)$ and $M_{\max} \leq n$, so that the total complexity of the spectral algorithm is $O(n \log(n)^2 M_{\max})$.

In comparison, the complexity of the spectral algorithm of \cite{dCGLLC15} is $O(n M_{\max}^3)$ because of Step 1. This becomes much larger than our complexity when $M_{\max}$ grows as a power of $n$ (which is necessary in order to reach minimax rates).

\subsubsection{Least squares algorithm (see Section \ref{sec_leastsquares_th})}

We consider the algorithmic complexity of estimating the emission densities for all models $M$ such that $M_{\min} \leq M \leq M_{\max}$ with $n$ observations.

Step 1 is similar to the one of the spectral algorithm, but with $O(n M_{\max}^3)$ operations. The complexity of Step 2 is more difficult to evaluate. Since the criterion is nonconvex, finding the minimizer requires to run an approximate minimization algorithm whose complexity $C_n$ will depend on the desired precision\---which will in turn depend on the number of observations $n$\---and on the initial points. As discussed in \cite{lehericy2016order}, this is usually the longest step when computing least squares estimators.
Thus, the total complexity of the least squares algorithm is $O(n M_{\max}^3 + C_n)$.

Note that despite the worse sample complexity, the least squares algorithm is tractable and can greatly improve the estimation for small sample size. As shown in Section \ref{sec_results}, the spectral algorithm is unstable for small samples, which makes the state-by-state selection procedure return abnormal results. This can be explained by the matrix inversions of the spectral method, which sometimes lead to nearly singular matrices when the noise is too large. On the other hand, the least squares method does not involve any matrix inversion, and often gives better results than the spectral estimators, as shown in \cite{dCGL15}, thus making it a relevant choice for small to medium data sets.

\subsubsection{Selection method and POS variant (see Sections \ref{sec_lepski_th} and \ref{sec_alternatives})}

We consider the algorithmic complexity of selecting estimators from a family $(\hat{\fbf}^{(M)})_{M_{\min} \leq M \leq M_{\max}}$ of estimators. The selection algorithms can be decomposed in two parts.
\begin{itemize}
\item Compute the distances $\| \hat{f}_k^{(M)} - \hat{f}_k^{(M')} \|_2$ for all $M$, $M'$ and $k$. This has complexity $O(M_{\max}^3)$: it requires to compute the $\Lbf^2$ distance of at most $M_{\max}^2$ couples of functions in a Hilbert space of dimension $M_{\max}$.

\item Compute $\hat{\rho}_k$ defined as the abscissa of the largest jump of the function $\rho \longmapsto \hat{M}_k(\rho)$ for all $k$, where $\hat{M}_k$ is defined as in Section \ref{sec_penalty_calibration}. Note that computing $\hat{M}_k(\rho)$ requires $O(M_{\max}^2)$ operations. An approximate value of $\hat{\rho}_k$ can be computed in $O(\log(\hat{\rho}_k) M_{\max}^2)$ operations, which is usually $O(M_{\max}^2)$.
\end{itemize}
Once the $\hat{\rho}_k$ are known, it is possible to calibrate the penalty in constant time for the three calibrations methods (eachjump, jumpmax and jumpmean) and to select the final models in $O(M_{\max}^2)$ operations.

Thus, the total complexity of the selection algorithm and of its POS variant is $O(M_{\max}^3)$.

\subsubsection{Selection method, MAX variant (see Section \ref{sec_alternatives})}

In the MAX variant, the first step of the standard selection procedure is replaced by computing the distances $\| \hat{f}^{(M_{\max})}_k - \hat{f}^{(M)}_k \|_2$ for all $M$. This has complexity $O(M_{\max}^2)$. The complexity of the other steps remains unchanged.

Thus, the total complexity of the MAX variant of the selection algorithm is $O(M_{\max}^2)$.

\begin{table}
\centering
\renewcommand{\arraystretch}{1.3}
\begin{tabular}{l|l|c|}
\cline{2-3}
 & \multicolumn{1}{|c|}{Algorithm} & Complexity \\ \hhline{-==}
\multirow{2}{*}{\parbox{2cm}{Preliminary\\estimators}} & Spectral method & $O(n \log(n)^2 M_{\max})$ \\ \cline{2-3}
 & Spectral method (\cite{dCGLLC15}) & $O(n M_{\max}^3)$ \\ \cline{2-3}
 & Least squares method & $O(n M_{\max}^3 + C_n)$ \\ \hhline{===}
Selection step & Standard and POS variant & $O(M_{\max}^3)$ \\ \cline{2-3}
 & MAX variant & $O(M_{\max}^2)$ \\ \cline{2-3}
\end{tabular}
\caption{Complexities of the different algorithms. $n$ is the number of observations, $M_{\max}$ is the largest model dimension considered.}
\label{table_complexity_summary}
\end{table}

\section{Application to real data}
\label{sec_appli}

In this section, we present the results of our method on two sets of trajectories. Trajectories are a typical example of dependent data that shows several behaviours depending on the activity of the entity being tracked, which makes hidden Markov models a popular modelling choice. For instance, the movement of a fisher is not the same depending on whether he's travelling to the next fishing zone or actually fishing.

The first data set follows artisanal fishers in Madagascar. The second one contains seabird movements. Studying the movements of fishers and seabirds has many applications, for instance understanding the fishing habits of the tracked entity, controlling the fishing pressure on local ecosystems and monitoring the dynamics of coastal ecosystems, see for instance \cite{boyd2014HmmUsesOfBirdTrajectories, vermard2010HmmUsesOfVesselTrajectories} and references therein.

\subsection{Artisanal fishery}

We use GPS tracks of artisanal fishers with a regular sampling period of 30 seconds. These tracks were produced by Faustinato Behivoke (Institut Halieutiques et des Sciences Marines, Université de Toliara, Madagascar) and Marc Léopold (IRD), who recorded artisanal fishers from Ankilibe, in Madagascar. Their fishing method is a seine netting.

From this data, we compute the velocity of the fisher during each time step. In order to estimate densities on $[0,1]$, we divide this velocity by an upper bound of the maximum observed velocity.
We consider the observation space $\Ycal = [0,1]$ endowed with the dominating measure $\delta_0 + \text{Leb}$, where $\delta_0$ is the dirac measure in zero and Leb is the Lebesgue measure on $[0,1]$. As a proof of concept, we use the orthonormal basis consisting of the trigonometric basis on $[0,1]$ and the indicator function of $\{0\}$, that is the family $(\varphi_m)_{m \in \Nbb}$ defined on $[0,1]$ by
\begin{align*}
&\text{if } x = 0,
	\begin{cases}
	\varphi_0(x) = 1 \\
	\varphi_m(x) = 0 \quad \text{for all } m \in \Nbb^*
	\end{cases} \\
&\text{if } x \neq 0,
	\begin{cases}
	\varphi_0(x) = 0 \\
	\varphi_1(x) = 1 \\
	\varphi_{2m}(x) = \sqrt{2} \cos(2\pi mx) \quad \text{for all } m \in \Nbb^* \\
	\varphi_{2m+1}(x) = \sqrt{2} \sin(2\pi mx) \quad \text{for all } m \in \Nbb^*
	\end{cases}
\end{align*}

The number of hidden states is chosen using the spectral thresholding method of \cite{lehericy2016order}. This methods consists is based on the fact that the rank of the spectral tensor $\Ebb \hat{\Nbf}_{m,m}$ (with the notations of Algorithm~\ref{alg:Spectral_complet} in Appendix~\ref{app_spectral}) is the number of hidden states. This is visible in the spectrum of $\hat{\Nbf}_{m,m}$ by an elbow, as shown in Figure~\ref{fig_elbow_selectionK}. Based on these spectra, we use two hidden states.

\begin{figure}[!t]
\centering
\vspace{-1em}
\includegraphics[scale=0.57]{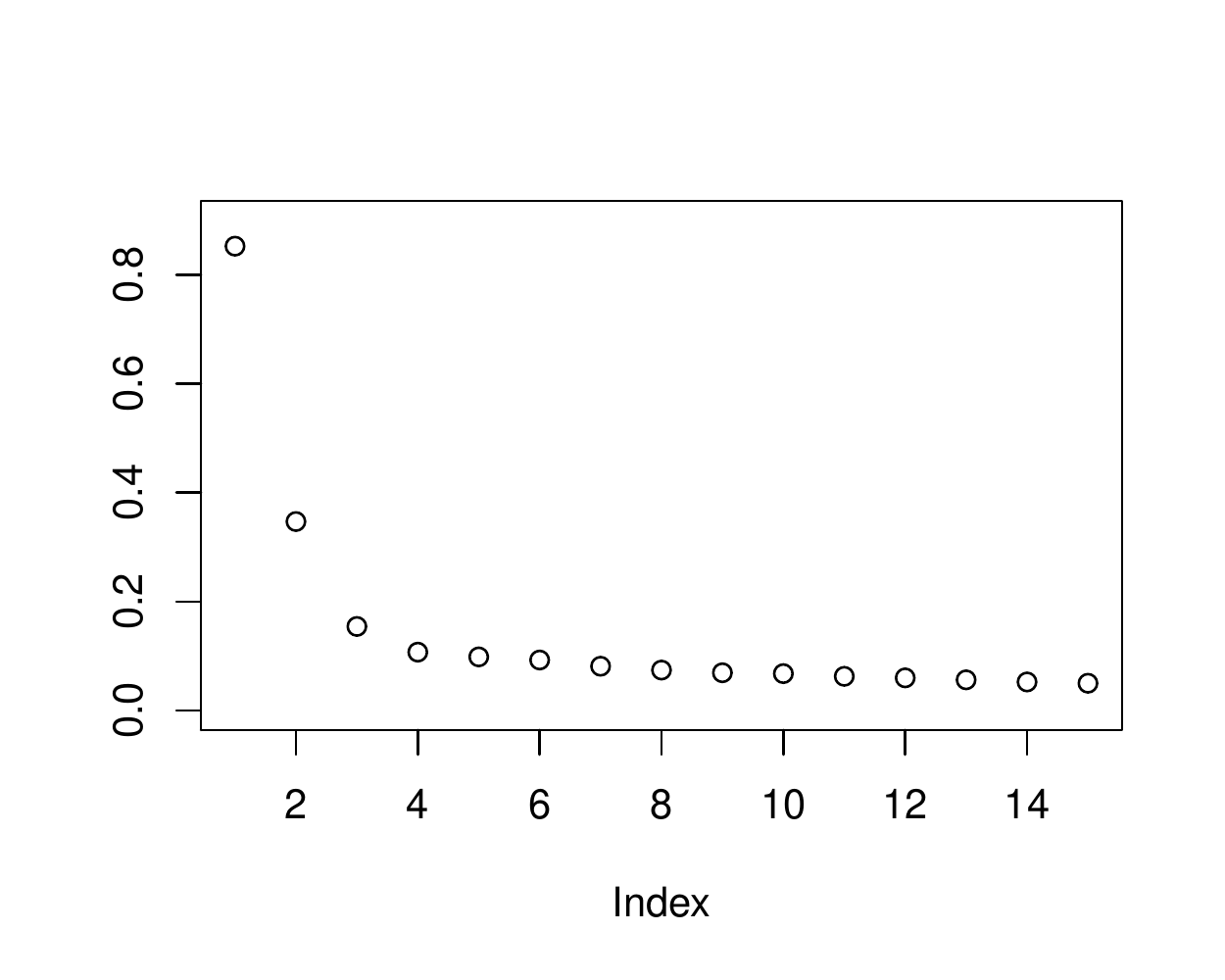}
\includegraphics[scale=0.57]{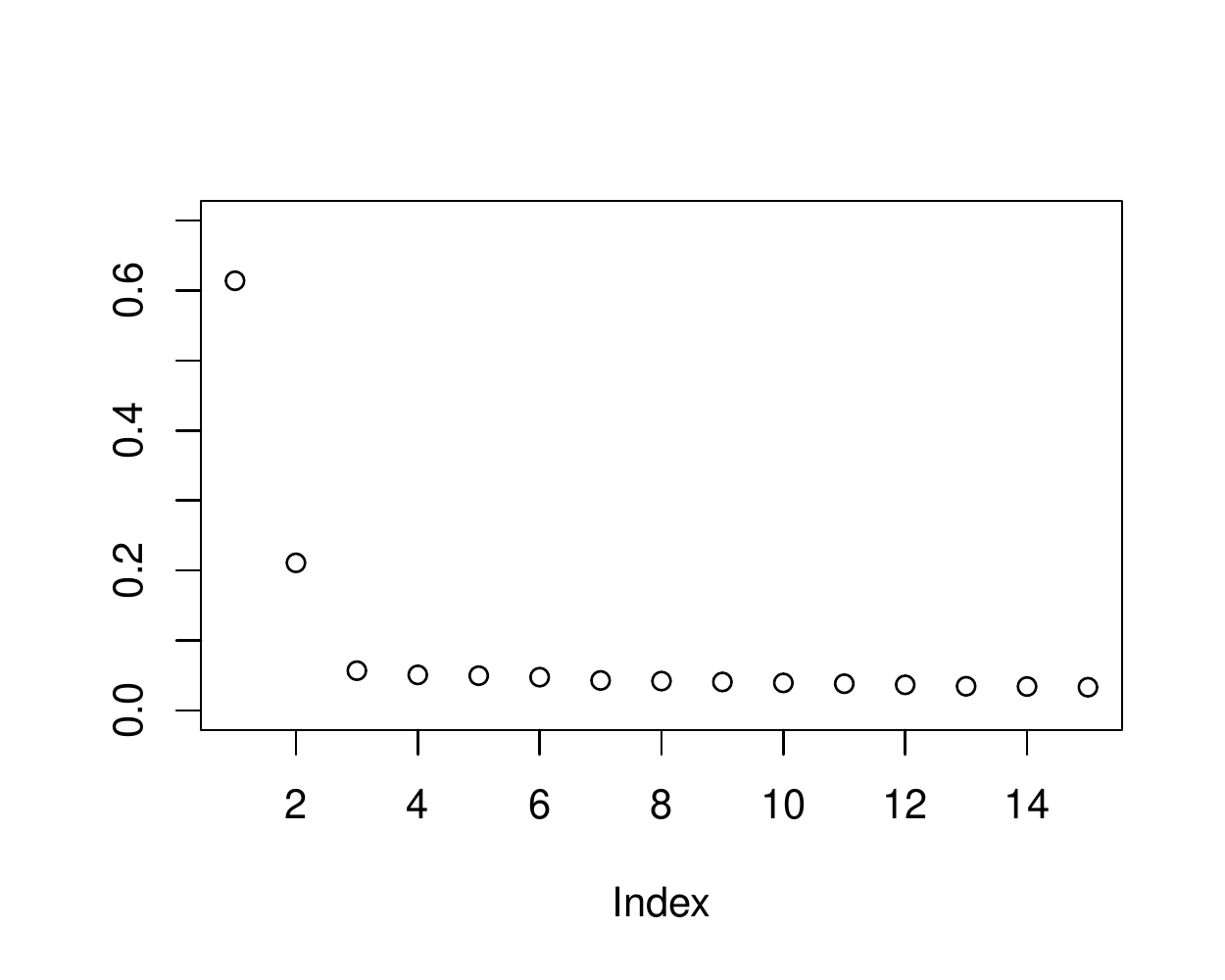}
\vspace{-1em}
\caption{First 15 eigenvalues of the spectrum of the empirical tensor $\hat{\Nbf}_{50,50}$ (see Algorithm~\ref{alg:Spectral_complet} in Appendix~\ref{app_spectral}). Left: fisher 1, right: fisher 2.}
\label{fig_elbow_selectionK}
\end{figure}

The results using $\verb?M_max? = 1000$ are shown in Figures \ref{fig_appli_1} and \ref{fig_appli_2}. We took the normalizing velocity large enough that all observed normalized velocities belong to $[0,0.8]$, hence the plot betweeen 0 and 0.8 for the densities.
\begin{figure}[!t]
\centering
\vspace{-1em}
\includegraphics[scale=0.67]{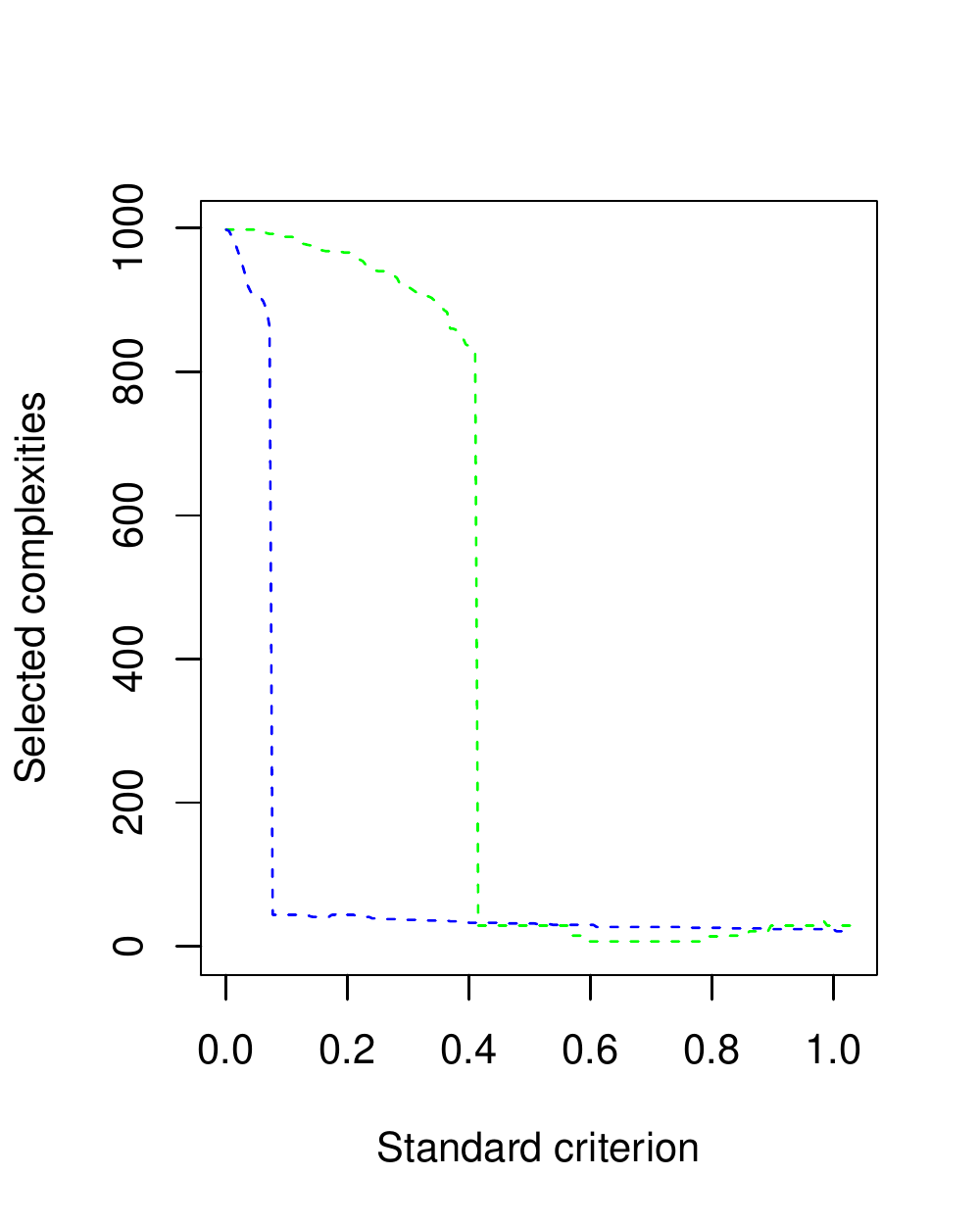} % Taille R : 4, 5
\includegraphics[scale=0.67]{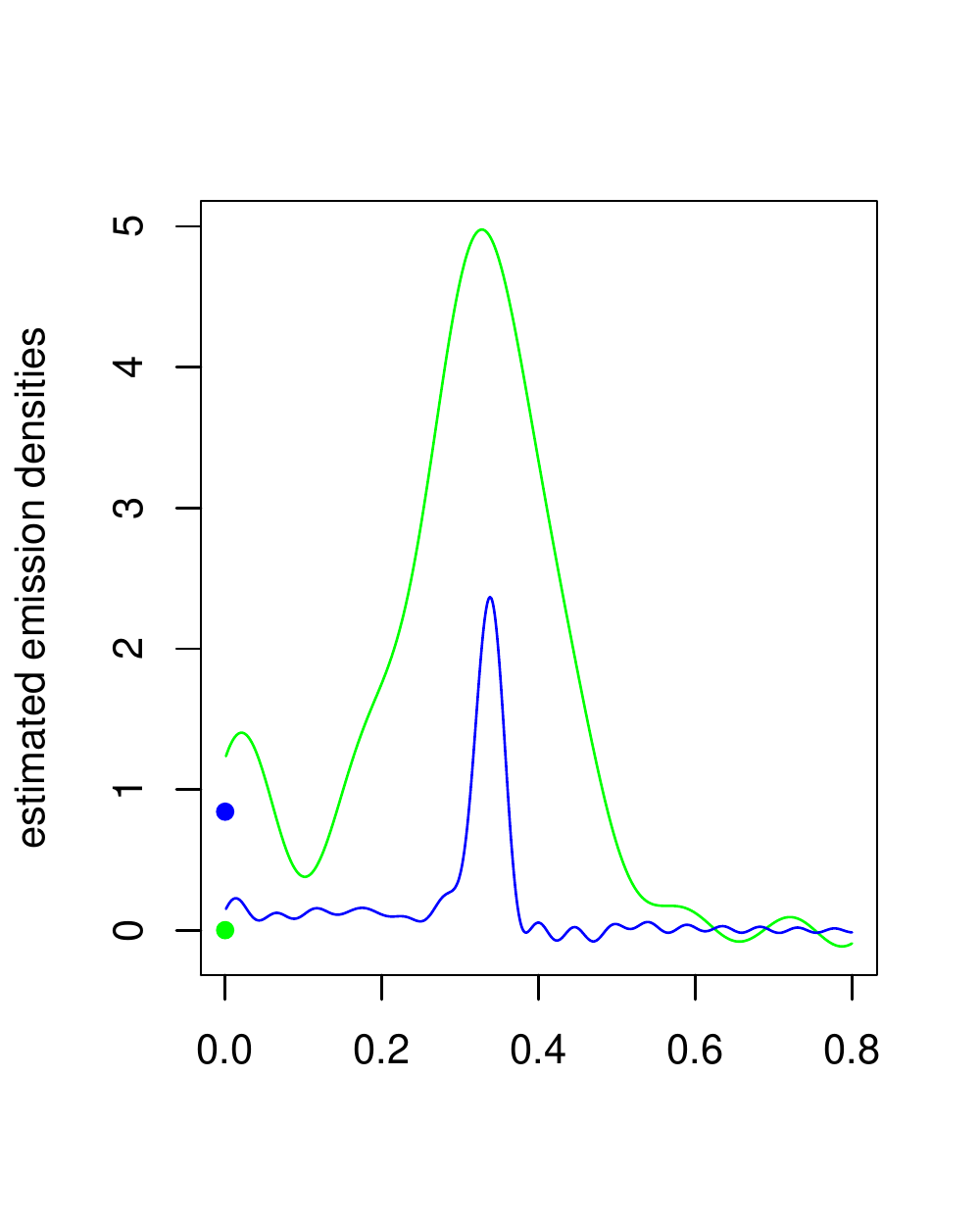}
\vspace{-1em}
\caption{Selected complexities and estimated densities on artisanal fishery data (fisher 1, $n = 17,300$). Green = state 1, blue = state 2. The dirac component is shown as a dot at $y = 0$. The selected dimensions are $(14,41)$.}
\label{fig_appli_1}
\end{figure}

\begin{figure}[!t]
\centering
\vspace{-1em}
\includegraphics[scale=0.67]{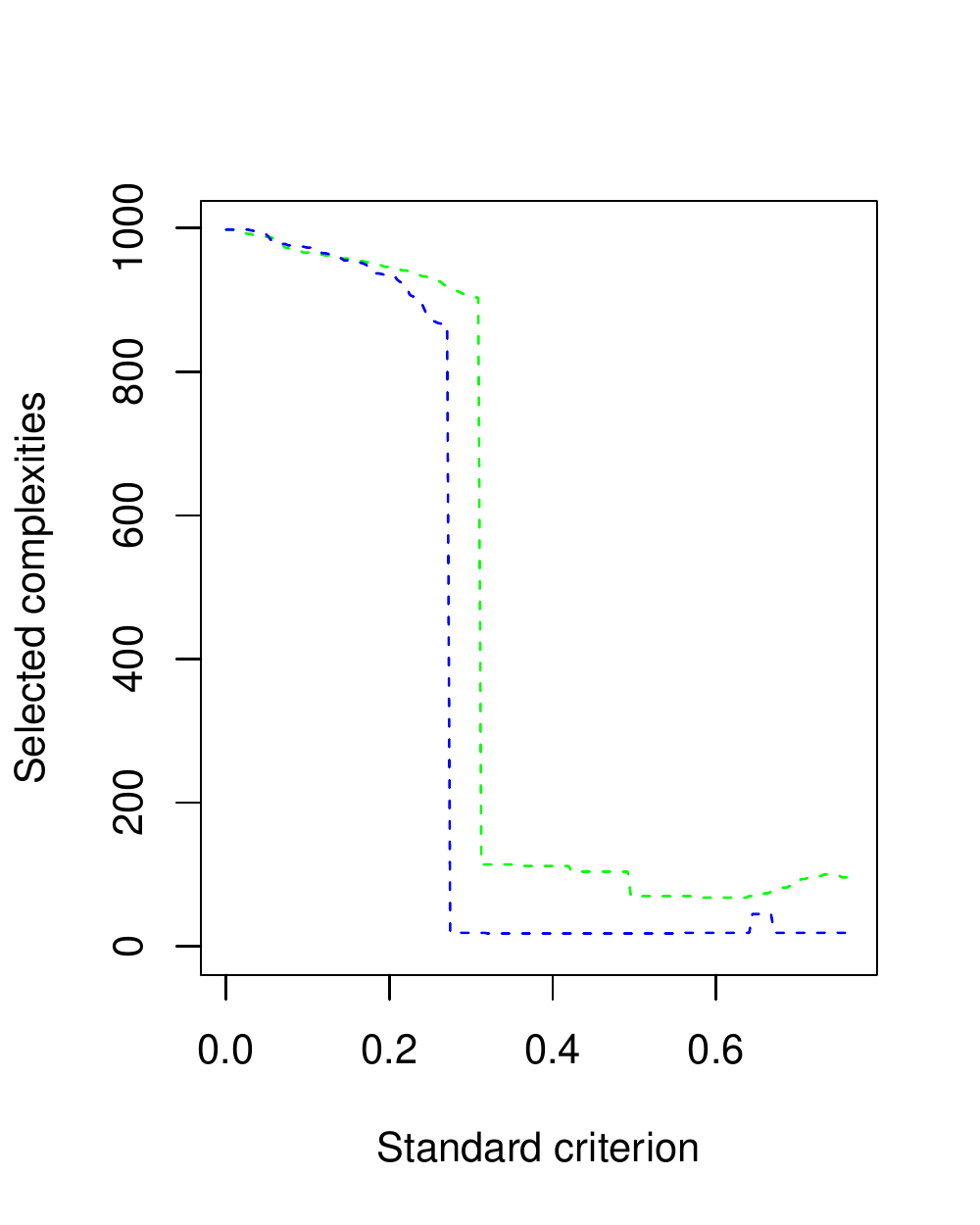} % Taille R : 4, 5
\includegraphics[scale=0.67]{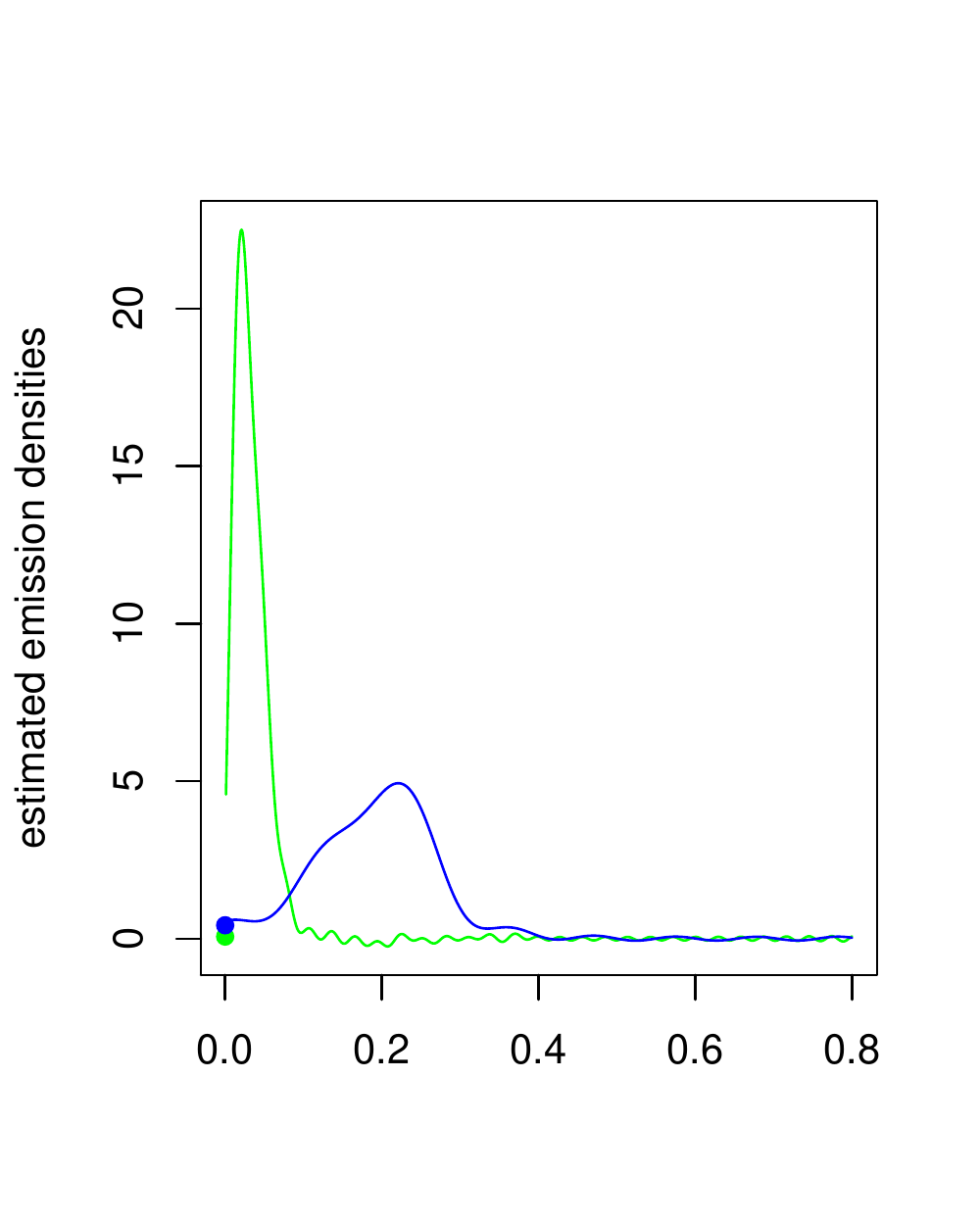}
\vspace{-1em}
\caption{Selected complexities and estimated densities on artisanal fishery data (fisher 2, $n = 11,600$). Green = state 1, blue = state 2. The dirac component is shown as a dot at $y = 0$. The selected dimensions are $(68, 18)$.}
\label{fig_appli_2}
\end{figure}

In both cases, the selected model complexities differ greatly depending on the state. This comes from the fact that in both cases, one of the density is spiked, thus requiring more vectors of the orthonormal basis to be approximated. This illustrates that our method is able to estimate the smoother densities with fewer vectors of the basis, thus preventing overfitting.

As a side note, we needed considerably less observations than in the simulations: around 10,000, compared to 500,000 in the simulations. This can be explained by the fact that each state is very stable, with an estimated probability of leaving the states below 0.02\---compared to 0.3 in the simulations. This is encouraging, as hidden states in real data are expected to be rather stable, especially when the sampling frequency is high, as long as the conditional independance of the observations can be assumed to hold.

\subsection{Seabird foraging}

In this Section, we consider the seabird data from \cite{bertrand2015ParetoRWBirdTrajectories} and we focus on the tracks named cormorant d in this paper.

We apply the same transformation as in the previous section to obtain normalized velocities in $[0,0.8]$ (after removal of anomalous velocities exceeding 150 m/s) and run the spectral algorithm with the trigonometric basis on $[0,1]$ plus the indicator of $\{0\}$. The spectral thresholding gives a number of hidden states equal to two; we set it to three to account for more complex behaviours of the seabirds. The results are shown in Figure~\ref{fig_appli_cormorant_d}.

\begin{figure}[!t]
\centering
\vspace{-1em}
\includegraphics[scale=0.67]{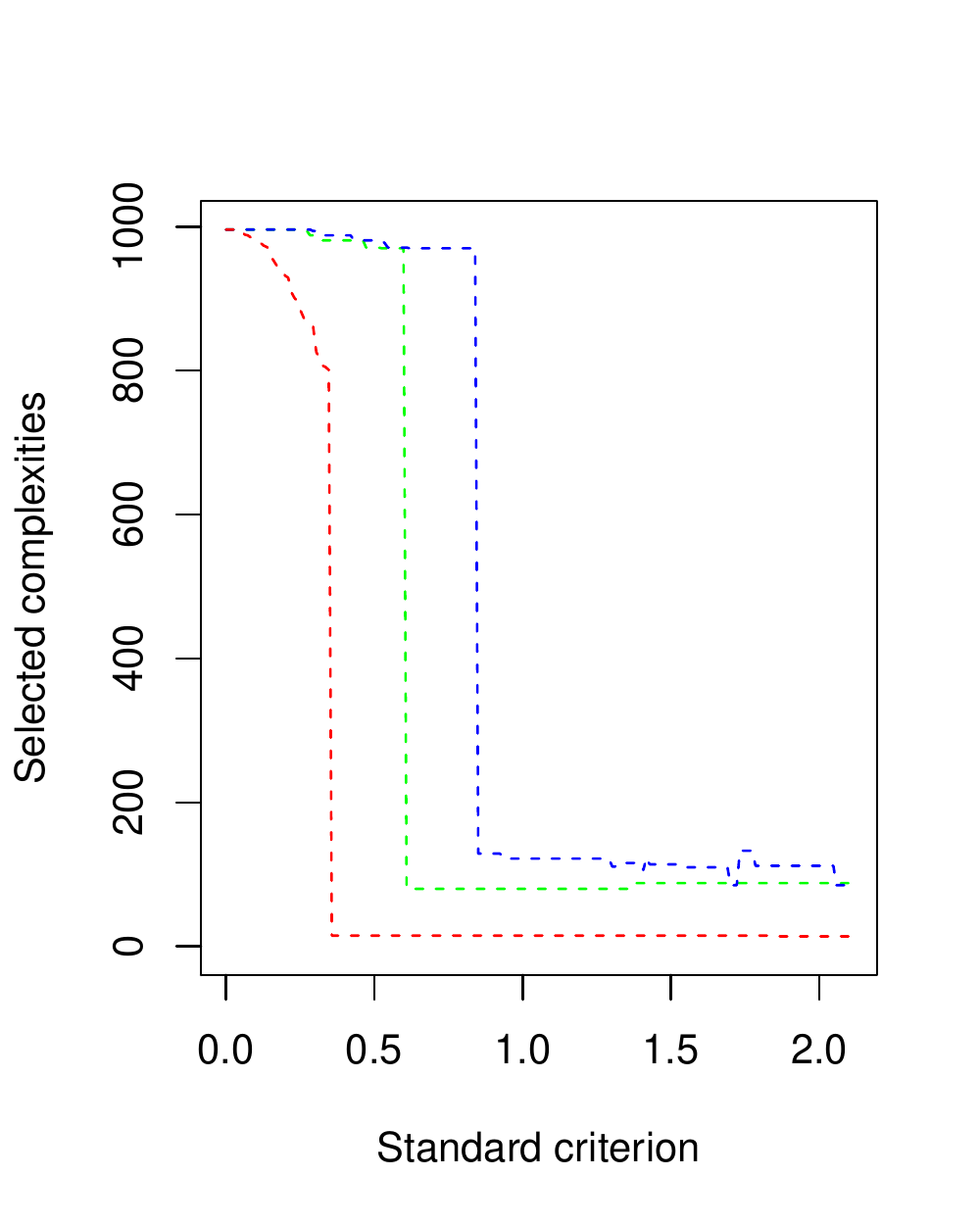} % Taille R : 4, 5
\includegraphics[scale=0.67]{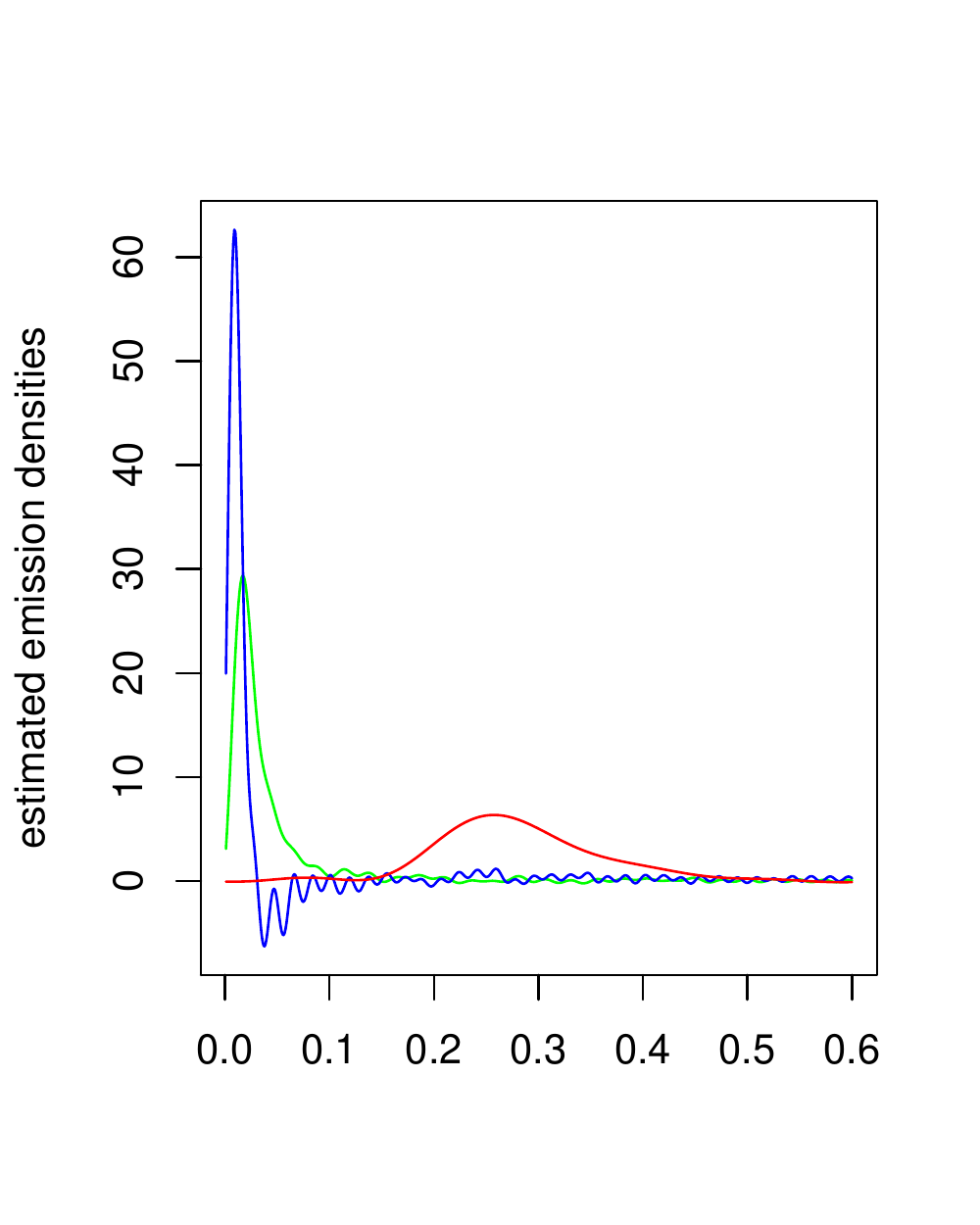}
\vspace{-1em}
\caption{Selected complexities and estimated densities for Cormorant d's trajectory ($n = 2,891$). Green = state 1, blue = state 2, red = state 3. The selected dimensions are $(80, 110, 15)$.}
\label{fig_appli_cormorant_d}
\end{figure}

Note that the use of the trigonometric basis allows the estimated densities to take negative values. This is not a problem as far as minimax rates of convergence (in $\Lbf^2$ norm) are concerned, however this can become an issue if one wants to use these densities in a forward-backward algorithm in order to get an estimator of the hidden states. One way to circumvent this problem is to use simplex projection to compute an approximation of the projection of these estimated density on the simplex of all probability densities. Note that since this is an $\Lbf^2$ projection on a convex set which contains the true densities, the projected densities have an even smaller error, thus keeping the minimax rate of convergence of the original estimators. The resulting densities are shown in Figure~\ref{fig_appli_cormorant_d_projSimplexe}
\begin{figure}[!t]
\centering
\vspace{-1em}
\includegraphics[scale=0.67]{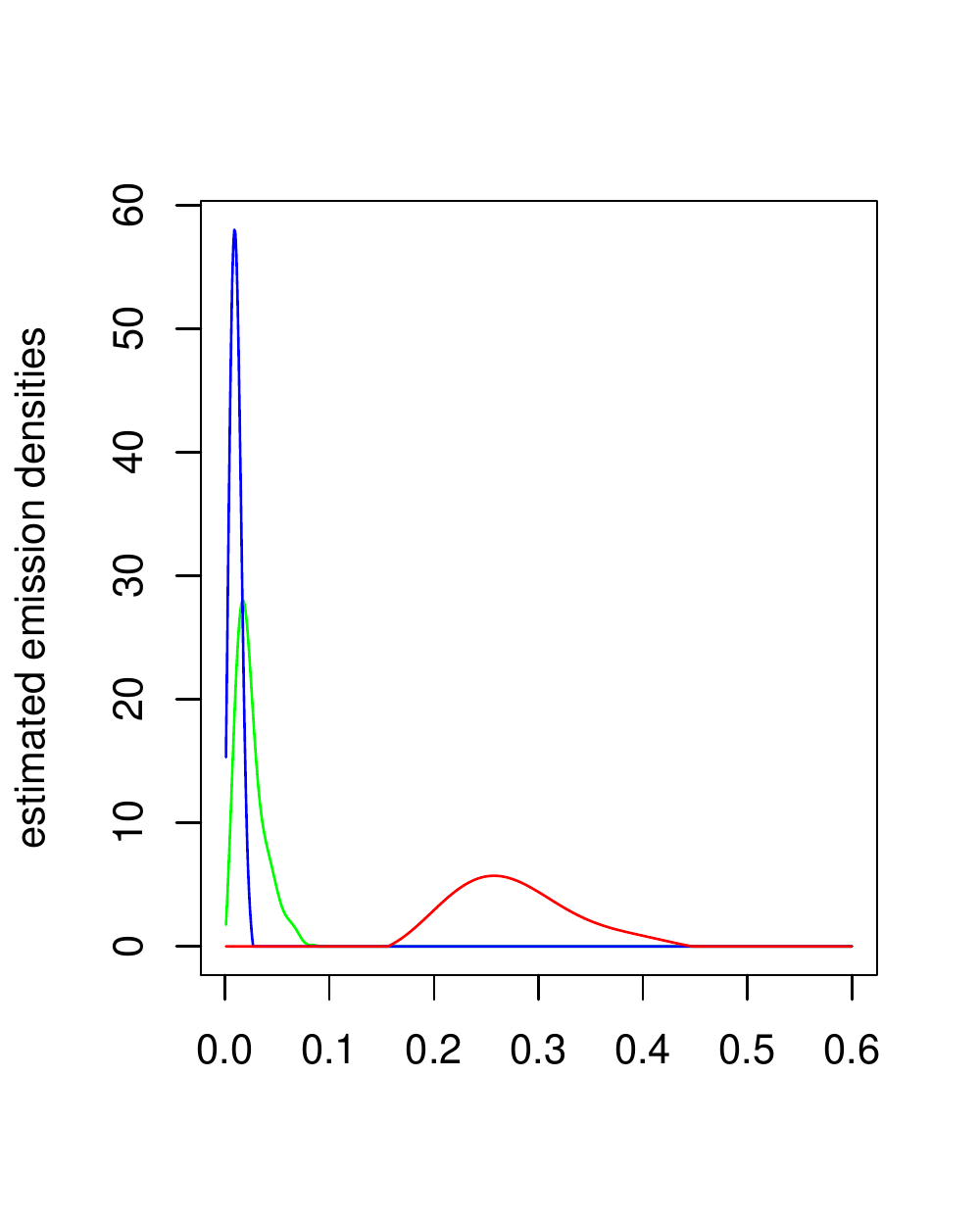}
\includegraphics[scale=0.67]{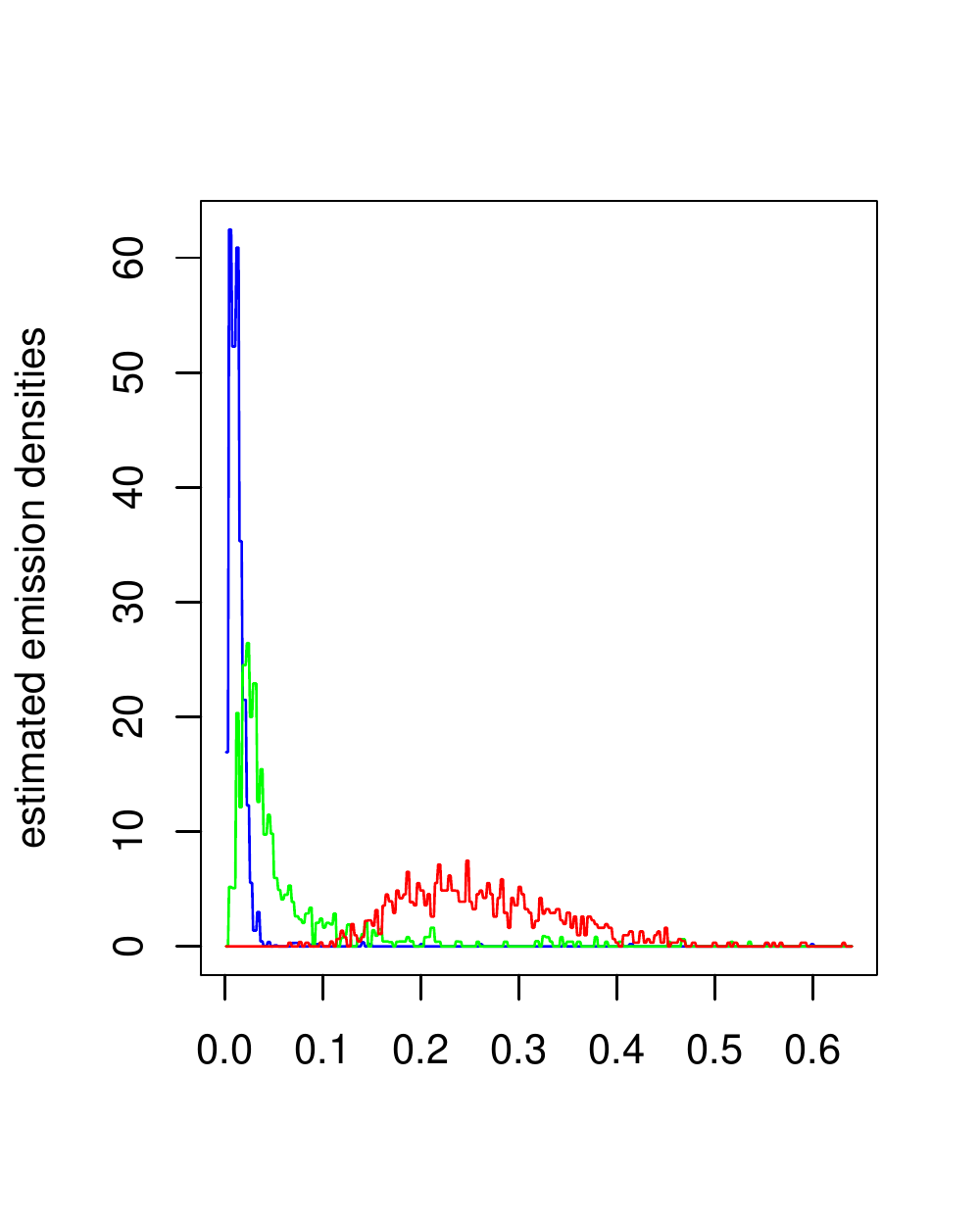}
\vspace{-1em}
\caption{Projection of the estimated densities of Figure~\ref{fig_appli_cormorant_d} for Cormorant d on the set of probability densities (left) and comparison with an estimation with histogram densities on a regular partition of size 300 using the EM algorithm (right).}
\label{fig_appli_cormorant_d_projSimplexe}
\end{figure}

The number of observations in this setting is even smaller than for the fishery's data set: our algorithm was able to recover three emission densities from less than 3,000 observations, despite the states being less stable than in the fishery data set: the diagonal terms of the estimated transition matrix using the EM algorithm are $(0.83, 0.93, 0.98)$. In addition, the result of our method is consistent with other estimation methods, as shown in Figure~\ref{fig_appli_cormorant_d_projSimplexe}: estimating the parameters with the EM algorithm using piecewise constant densities leads to a very similar result.

\section{Conclusion and perspectives}
\label{sec_conclusion}

We propose a state-by-state selection method to infer the emission densities of a HMM. Using a family of estimators, our method selects one estimator for each hidden state in a way that is adaptive with respect to this state's regularity. This method does not depend on the type of preliminary estimator, as long as a suitable variance bound is available. As such, it may be seen as a plug-in that takes a family of estimators and the corresponding variance bound and outputs the selected estimator. Note that its complexity does not depend on the number of observations used to compute the estimators, which makes it applicable to arbitrarily large data sets.

To apply this method, we present two families of estimators: a least squares estimator and a spectral estimator. For both, we prove a bound on their variance and show that this bound allows to recover the minimax rate of convergence separately on each hidden state, up to a logarithmic factor. The variance bounds are similar to a BIC penalty, with an additional logarithmic factor for the spectral estimators.

We carry out a numerical study of the method and some variants on simulated data. We use the spectral estimators, which are both fast and don't suffer from initialization issues, unlike the least squares and maximum likelihood estimators. The simulations show that our selection method is very fast compared to the computation of the estimators and that indeed, the final estimators reach the minimax rate of convergence on each state.

Then, we compare our method with a cross validation estimator based on a least square risk. This estimator only reaches the minimax rate corresponding to the worst regularity among the emission densities and fails to select models with small dimensions. It is still noteworthy that the cross validation returns relevant results for small sample sizes, whereas our method requires the sample size to be large enough to work properly. An interesting problem would be to investigate whether cross validation or other methods can be combined with our state-by-state selection method to give an algorithm that is both fast, stable for small sample sizes and optimal in the asymptotic setting.

Finally, we apply our algorithm to real trajectory data sets. On this data, our method proves that it is able to match the regularity of the underlying emission densities. In addition, it is able to produce sensible results with far fewer observations than in our simulation study.
\\

Our state-by-state selection method can be easily applied to multiview mixture models (also named mixture models with repeated measurement, see for instance \cite{bonhomme2016multiview} and \cite{gassiat2016multiview}). Let us first describe the model. A multiview mixture model consists of two random variables, a hidden state $U$ and an observation vector $\Ybf := (Y_i)_{i \in [m]}$ such that conditionally to $U$, the components $Y_i$ of $\Ybf$ are independent with a distribution depending only on $U$ and $i$. Let us assume that $U$ takes its values in a finite set $\Xcal$ of size $K$ and that the $Y_i$ have some density $f^*_{u,i}$ conditionally to $U = u$ with respect to a dominating measure.
A question of interest is to estimate the densities $f^*_{u,i}$ from a sequence of observed $(\Ybf_n)_{n \geq 1}$.

Our state-by-state selection method can be applied directly to such a model as long as estimators with a proper variance bound are available (see assumption \textbf{[H$(\epsilon)$]} in Section \ref{sec_Lepski_framework}). Indeed, we never use the dependency structure of the model. Regarding the development of preliminary estimators, multiview mixture models appear closely related to hidden Markov models: \cite{AHK12} and \cite{robin2014estimating} develop spectral methods that work for both multiview mixtures and HMMs at the same time using the same theoretical arguments. Thus, it seems clear that variance bounds such as the ones we developed can also be written for multiview mixture models.

% Acknowledgements should go at the end, before appendices and references

\acks{I am grateful to Elisabeth Gassiat and Claire Lacour for their precious advice. I thank Augustin Touron for providing me with a R implementation of the spectral algorithm. I would also like to thank Marie-Pierre Etienne and of course Faustinato Behivoke (Institut Halieutiques et des Sciences Marines, Université de Toliara, Madagascar), Marc Léopold (IRD) and Sophie Bertrand (IRD) for letting me work on their data sets.
%We would like to acknowledge support for this project from the National Science Foundation (NSF grant IIS-9988642) and the Multidisciplinary Research Program of the Department of Defense (MURI N00014-00-1-0637).
}

\appendix

\section{Spectral algorithm, full version}
\label{app_spectral}

\begin{algorithm}[!h]
  \caption{Spectral estimation of the emission densities of a HMM (full version)}
\label{alg:Spectral_complet}
  \SetAlgoLined
  \KwData{A sequence of observations $(Y_{1}, \dots ,Y_{n+2})$, two dimensions $m \leq M$, an orthonormal basis $(\varphi_1, \dots ,\varphi_M)$ and number of retries $r$.}
  \KwResult{Spectral estimators $(\hat f^{(M,r)}_k)_{k \in \Xcal}$, $\hat\Qbf$ and $\hat\pi$.}
    \BlankLine
\begin{enumerate}[{\bf [Step 1]}]
\item Consider the following empirical estimators: for any $a, c \in [m]$ and $b \in [M]$,
\begin{itemize}
\item $\hat{\Lbf}_{m}(a):= \frac{1}{n} \sum_{s=1}^{n}\varphi_{a}(Y_{s})$
\item $\hat{\Mbf}_{m,M,m}(a,b,c):= \frac{1}{n} \sum_{s=1}^{n}\varphi_{a}(Y_{s})\varphi_{b}(Y_{s+1})\varphi_{c}(Y_{s+2})$
\item $\hat{\Nbf}_{m,M}(a,b):= \frac{1}{n} \sum_{s=1}^{n}\varphi_{a}(Y_{s})\varphi_{b}( Y_{s+1})$
\item $\hat{\Pbf}_{m,m}(a,c):= \frac{1}{n} \sum_{s=1}^{n}\varphi_{a}(Y_{s})\varphi_{c}(Y_{s+2})$.
\end{itemize}

\item
Let $\hat\Ubf_m$ be the $m \times K$ matrix of orthonormal left singular vectors and $\hat\Vbf_M$ be the $M \times K$ matrix of orthonormal right singular vectors of $\hat{\Nbf}_{m,M}$ corresponding to its top $K$ singular values.

\item 
Form the matrices for all $b \in [M]$, $\hat\Bbf(b) := (\hat\Ubf_m^{\top} \hat{\Pbf}_{m,m} \hat\Ubf_m)^{-1}
\hat\Ubf_m^{\top} \hat{\Mbf}_{m,M,m} (\ldotp,b,\ldotp) \hat\Ubf_m$.

\item
Set $(\Theta_i)_{1 \leq i \leq r}$ $r$ i.i.d. $(K \times K)$ unitary matrix uniformly drawn. Form the matrices for all $k \in \Xcal$ and $i \in [r]$, ${\hat\Cbf}_i(k):=\sum_{b=1}^{M} (\hat\Vbf_M\Theta_i)(b,k) \hat\Bbf(b)$.

\item
Compute $\hat\Rbf_i$ a $(K\times K)$ unit Euclidean norm columns matrix that diagonalizes the matrix $\hat\Cbf_i(1)$: $\hat\Rbf_i^{-1} \hat\Cbf_i(1) \hat\Rbf_i = \Diag{(\hat\Lambda_i(1,1), \dots ,\hat\Lambda_i(1,K))}$.

\item
Set for all $k,k'\in \Xcal$, $\hat\Lambda_i(k,k') := (\hat\Rbf_i^{-1} \hat\Cbf_i(k) \hat\Rbf_i) (k',k')$. Choose $i_0$ maximizing $\min_k \min_{k_1 \neq k_2} |\hat\Lambda_i(k,k_1) - \hat\Lambda_i(k,k_2)|$ and set $\hat{\Obf} := \hat\Vbf_M \Theta_{i_0} \hat\Lambda_{i_0}$.

\item
Consider the emission densities estimators $\hat\fbf^{(M,r)} := (\hat f^{(M,r)}_k)_{k \in \Xcal}$ defined by for all $k \in \Xcal$, $\hat f^{(M,r)}_k := \sum_{b=1}^M \hat{\Obf}(b,k) \varphi_b$.

\item
Let $\hat{\Obf}_m$ be the $m \times K$ matrix containing the first $m$ rows of $\hat{\Obf}$. Set $\hat\pi = \Pi_\Delta \left( (\hat\Ubf_m^\top \hat\Obf_m)^{-1} \hat\Ubf_m^\top \hat\Lbf_m \right)$ where $\Pi_\Delta$ is the $\Lbf^2$ projection onto the probability simplex.

\item
Let $\hat{\Qbf}$ be the transition matrix defined by
$\hat{\Qbf} = \Pi_\text{TM}\left(
(\hat\Ubf_m^\top \hat\Obf_m \Diag[\hat\pi])^{-1}
\hat\Ubf_m^\top \hat\Nbf_{m,M} \hat\Vbf_M
(\hat\Obf^\top \hat\Vbf_M)^{-1}
\right)$ \\
where $\Pi_\text{TM}$ is the projection onto the set of transition matrices. This projection is obtained by projecting each line of the matrix onto the probability simplex.

\end{enumerate}
\end{algorithm}

\newpage

\section{Proofs}
\label{sec_proofs}

\subsection{Proof of Lemma \ref{lemma_permutation_unique}}
\label{sec_preuve_lemma_permutation}

Let $\tau_{n,M}$ be the permutation that minimizes $\displaystyle \tau \longmapsto \max_{k \in \Xcal} \left\| \hat{f}^{(M)}_k - f^{*,(M)}_{\tau(k)} \right\|_2$. \textbf{[H$(\epsilon)$]} means that with probability $1-\epsilon$, one has $\displaystyle \max_{k \in \Xcal} \left\| \hat{f}^{(M)}_k - f^{*,(M)}_{\tau(k)} \right\|_2 \leq \frac{\sigma(M)}{2}$. \\

Let $M \in \Mcal$. Let us show that
$\displaystyle \left\| \hat{f}^{(M)}_{\tau_{n,M}^{-1}(k')} - \hat{f}^{(M_0)}_{\tau_{n,M_0}^{-1}(k)} \right\|_2
	> \left\| \hat{f}^{(M)}_{\tau_{n,M}^{-1}(k)} - \hat{f}^{(M_0)}_{\tau_{n,M_0}^{-1}(k)} \right\|_2$
for all $k, k' \in \Xcal$ such that $k' \neq k$.
If this holds, then the definition of $\hat{\tau}^{(M)}$ implies that $\hat{\tau}^{(M)} = \tau_{n,M}^{-1} \circ \tau_{n,M_0}$. Thus, one has $\displaystyle \max_{k \in \Xcal} \left\| \hat{f}^{(M)}_{k, \text{new}} - f^{*,(M)}_{\tau_{n,M_0}(k)} \right\|_2 \leq \frac{\sigma(M)}{2}$, which is exactly Equation~(\ref{eq_Halt}) with $\tau_n = \tau_{n,M_0}$. \\

Applying the triangular inequality leads to
\begin{align*}
\left\| \hat{f}^{(M)}_{\tau_{n,M}^{-1}(k)} - \hat{f}^{(M_0)}_{\tau_{n,M_0}^{-1}(k)} \right\|_2
	\leq& \left\| \hat{f}^{(M)}_{\tau_{n,M}^{-1}(k)} - f^{*,(M)}_{k} \right\|_2
		+ \left\| f^{*,(M)}_{k} - f^{*,(M_0)}_{k} \right\|_2
		+ \left\| f^{*,(M_0)}_{k} - \hat{f}^{(M_0)}_{\tau_{n,M_0}^{-1}(k)} \right\|_2 \\
	\leq& \frac{\sigma(M)}{2}
		+ B_{M,M_0}
		+ \frac{\sigma(M_0)}{2}
\end{align*}
and
\begin{align*}
\left\| \hat{f}^{(M)}_{\tau_{n,M}^{-1}(k')} - \hat{f}^{(M_0)}_{\tau_{n,M_0}^{-1}(k)} \right\|_2
	\geq& \left\| f^{*,(M_0)}_{k'} - f^{*,(M_0)}_{k} \right\|_2
		- \left\| \hat{f}^{(M)}_{\tau_{n,M}^{-1}(k')} - f^{*,(M)}_{k'} \right\|_2 \\
		&- \left\| f^{*,(M)}_{k'} - f^{*,(M_0)}_{k'} \right\|_2
		- \left\| f^{*,(M_0)}_{k} - \hat{f}^{(M_0)}_{\tau_{n,M_0}^{-1}(k)} \right\|_2 \\
	\geq& m(\fbf^*, M_0) - \frac{\sigma(M)}{2} - B_{M,M_0} - \frac{\sigma(M_0)}{2}.
\end{align*}
Thus, the result holds as soon as $m(\fbf^*, M_0) - \frac{\sigma(M)}{2} - B_{M,M_0} - \frac{\sigma(M_0)}{2} > \frac{\sigma(M)}{2} + B_{M,M_0} + \frac{\sigma(M_0)}{2}$, which is the condition of Lemma \ref{lemma_permutation_unique}.

\subsection{Proof of Theorem \ref{th_garantieSpectral}}
\label{sec_preuve_spectral}

The structure of the proof is the same as the one of Theorem 3.1 of \cite{dCGLLC15}.

The first difference lies in the fact that we consider different models for each component of the tensors $\hat\Nbf_{m,M}$ and $\hat\Mbf_{m,M,m}$ in Step 1. As a consequence, we use the left \emph{and} right singular vectors of $\hat\Nbf_{m,M}$ instead of just the right singular vectors of $\hat\Pbf_{m,m}$. A careful reading shows that their proof can be adapted straightforwardly to this situation.

The second difference consists in generating several independant random unitary matrices in Step 4 and keeping the one that separates the eigenvalues of all $\hat\Cbf_i(k)$ best. This allows to replace Lemma F.6 of \cite{dCGLLC15} by the following one, based on the independence of the unitary matrices:
\begin{lemma}
For all $x > 0$ and $r \in \Nbb^*$,
\begin{equation*}
\Pbb\left[
	\forall k, k_1 \neq k_2, |\hat\Lambda_{i_0}(k,k_1) - \hat\Lambda_{i_0}(k,k_2)|
		\geq \frac{2 e^{-x/r} (1 - \epsilon_{\Nbf_{m,M}}^2)^{1/2}}{\sqrt{e} K^{5/2}(K-1)} \gamma(\Obf_M)
\right] \geq 1 - e^{-x}
\end{equation*}
and
\begin{equation*}
\Pbb\left[
	\| \hat\Lambda_{i_0} \|_\infty
		\geq \frac{1 + \sqrt{2} \sqrt{x + \log(K^2 r)}}{\sqrt{K}} \| \Obf_M \|_{2, \infty}
\right] \leq e^{-x},
\end{equation*}
The notations $\epsilon_{\Nbf_{m,M}}$ (or $\epsilon_{\Pbf_M}$ in the original proof), $\gamma(\Obf_M)$ et $\| \Obf_M \|_{2, \infty}$ are introduced in \cite{dCGLLC15}.
\end{lemma}
Using this lemma, their proof leads to our result by taking $r = x = t$.

\subsection{Definition of the polynomial \texorpdfstring{$H$}{H}}
\label{sec_expression_H}

\subsubsection{Definition}

We parametrize the application
\begin{equation}
\label{eq_pre_forme_quadratique}
(\pi, \Qbf, \fbf) \in \Delta \times \Qcal \times \Span(\fbf^*)^K \longmapsto \|g^{\pi, \Qbf, \fbf} - g^{\pi^*, \Qbf^*, \fbf^*}\|_2^2
\end{equation}
in the following way. For $p \in \Rbb^{K-1}$, $q \in \Rbb^{K \times (K-1)}$ and $A \in \Rbb^{K \times (K-1)}$, define the extensions
\begin{itemize}[itemsep=1pt]
\item $\bar{p} \in \Rbb^K$ defined by $\bar{p}(k) = p(k)$ for all $k \in [K-1]$ and $\bar{p}(K) = - \sum_{k \in [K-1]} p(k)$;
\item $\bar{q} \in \Rbb^{K \times K}$ by $\bar{q}(k,K) = - \sum_{k' \in [K-1]} q(k,k')$;
\item $\bar{A} \in \Rbb^{K \times K}$ by $\bar{A}(k,K) = - \sum_{k' \in [K-1]} A(k,k')$.
\end{itemize}
$\bar{p}$ corresponds to $\pi - \pi^*$, $\bar{q}$ to $\Qbf - \Qbf^*$ and $A$ to the components of $\fbf - \fbf^*$ on $\fbf^*$ (which is a basis as soon as \textbf{[Hid]} holds). The condition on the last component of $\bar{p}$ and of each line of $\bar{q}$ and $\bar{A}$ follows from the fact that $\bar{p}$ corresponds to the difference of two probability vectors, $\bar{q}$ corresponds to the difference of two transition matrices and $\bar{A}$ correspond to the difference of two vectors of probability densities on a basis of probability densities.

Then, consider the quadratic form derived from the Taylor expansion of
\begin{align*}
(p,q,A) \in \Rbb^{K-1} \times \Rbb^{K \times (K-1)} \times \Rbb^{(K-1) \times K}
		\longmapsto \|g^{\pi^* + \bar{p}, \Qbf^* + \bar{q}, \fbf + \bar{A} \fbf^*} - g^{\pi^*, \Qbf^*, \fbf^*}\|_2^2.
\end{align*}
Let $M$ be the matrix associated to this quadratic form. We define $H$ as the determinant of $M$. Direct computations show that $H$ is a polynomial in the coefficients of $\pi^*$, $\Qbf^*$ and $G(\fbf^*)$.

\subsubsection{Link between \texorpdfstring{$H$}{H} and the quadratic form from Equation (\ref{eq_pre_forme_quadratique})}
\label{sec_lien_H_forme_quadratique}

The goal of this section is to show how $H$ can be used to lower bound the quadratic form from Equation (\ref{eq_pre_forme_quadratique}) by a positive constant times the distance between $(\pi, \Qbf, \fbf)$ and $(\pi^*, \Qbf^*, \fbf^*)$. We will not need the assumptions \textbf{[Hid]}, \textbf{[HF]} or \textbf{[Hdet]} unless specified otherwise.

Let us start by the relation between the norms of $(p,q,A)$ and $(\bar{p}, \bar{q}, \bar{A})$.
\begin{lemma}
For all $(p,q,A) \in \Rbb^{K-1} \times \Rbb^{K \times (K-1)} \times \Rbb^{(K-1) \times K}$,
\begin{align*}
\| p \|_2^2 \leq \| \bar{p} \|_2^2 \leq K \| p \|_2^2, \\
\| q \|_F^2 \leq \| \bar{q} \|_F^2 \leq K \| q \|_F^2 ,\\
\| A \|_F^2 \leq \| \bar{A} \|_F^2 \leq K \| A \|_F^2.
\end{align*}
\end{lemma}
\begin{proof}
$\| p \|_2^2 \leq \| \bar{p} \|_2^2$ is immediate. Then,
\begin{align*}
\| \bar{p} \|_2^2 =& \| p \|_2^2 + \left(\sum_{k \in [K-1]} p(k) \right)^2 \\
	\leq& \| p \|_2^2 + (K-1) \sum_{k \in [K-1]} p(k)^2 \\
	=& K \| p \|_2^2.
\end{align*}
The proof is the same for $q$ and $A$.
\end{proof}

\noindent
The next lemma will be used to link the norms of $A$ and $A\fbf$.
\begin{lemma}
For all $\bar{A} \in \Rbb^{K \times K}$ and $\fbf^* \in (\Lbf^2(\Ycal, \mu))^K$,
\begin{equation*}
\sigma_K(G(\fbf^*)) \| \bar{A} \|_F^2
	\leq \sum_{k \in \Xcal} \| (\bar{A}\fbf^*)_k \|_2^2 \leq K \| G(\fbf^*) \|_\infty \| \bar{A} \|_F^2
\end{equation*}
\end{lemma}
\begin{proof}
For the first inequality, we use that for all $k \in \Xcal$,
\begin{align*}
\| (\bar{A}\fbf^*)_k \|_2^2
	=& \bar{A}(k,\cdot) G(\fbf^*) \bar{A}(k,\cdot)^\top \\
	\geq& \sigma_K(G(\fbf^*)) \| \bar{A}(k,\cdot) \|_2^2
\end{align*}
and the inequality follows by summing over $k$.

For the second inequality,
\begin{align*}
\sum_{k \in \Xcal} \| (\bar{A}\fbf^*)_k \|_2^2
	=& \sum_{k \in [K]} \int (\bar{A}\fbf^*)_k(x)^2 \mu(dx) \\
	=& \sum_{k \in [K]} \int \left(\sum_{j \in [K]} \bar{A}(k,j) f^*_j(x) \right)^2 \mu(dx) \\
	\leq& \sum_{k \in [K]} \int K \sum_{j \in [K]} \bar{A}(k,j)^2 (f^*_j)^2(x) \mu(dx) \\
	\leq& K \left(\sum_{k,j \in [K]} \bar{A}(k,j)^2\right) \sup_{j \in \Xcal} \int (f^*_j)^2(x) \mu(dx) \\
	=& K \| \bar{A} \|_F^2 \| G(\fbf^*) \|_\infty.
\end{align*}
\end{proof}

\noindent
Finally, we will use the following result of \cite{lehericy2016order} (Section B.2) in order to upper bound the spectrum of the matrix $M$.
\begin{lemma}
\label{lemma_upperbound_diff_g}
For all $\pi_1, \pi_2 \in \Delta$, for all $\Qbf_1, \Qbf_2 \in \Qcal$ and for all $\fbf_1, \fbf_2 \in (\Lbf^2(\Ycal, \mu))^K$,
\begin{equation*}
\| g^{\pi_1, \Qbf_1, \fbf_1} - g^{\pi_2, \Qbf_2,\fbf_2} \|_2
	\leq
	\sqrt{3 K (\| G(\fbf_1) \|_\infty^3 \vee \| G(\fbf_2) \|_\infty^3)} \dperm ((\pi_1, \Qbf_1,\fbf_1), (\pi_2, \Qbf_2,\fbf_2))
\end{equation*}
\end{lemma}

\noindent
Together, these results imply that for all $(p,q,A)$,
\begin{align*}
\|g^{\pi^* + \bar{p}, \Qbf^* + \bar{q}, \fbf^* + \bar{A} \fbf^*} -& g^{\pi^*, \Qbf^*, \fbf^*}\|_2^2 \\
	\leq& 3 K (\| G(\fbf^* + \bar{A} \fbf^*) \|_\infty^3 \vee \| G(\fbf^*) \|_\infty^3) (\| \bar{p} \|_2^2 + \| \bar{q} \|_F^2 + \sum_{k \in \Xcal} \| (\bar{A}\fbf)_k \|_F^2)\\
	\leq& 3 K \| G(\fbf^*) \|_\infty^3 (1 + K^2 \| A \|_F^2)^3 (K \| p \|_2^2 + K \| q \|_F^2 + K^2 \| G(\fbf^*) \|_\infty \| A \|_F^2)
\end{align*}
so that $\sigma_1(M) \leq \sqrt{3 K^3} (1 \vee \| G(\fbf) \|_\infty^2)$. Since $H = \prod_{i=1}^{(K-1)(2K+1)} \sigma_i(M)$, one has
\begin{equation*}
\sigma_{(K-1)(2K+1)}(M) \geq \frac{H}{(3 K^3 (1 \vee \| G(\fbf) \|_\infty^4))^{K^2-K/2}}.
\end{equation*}
Now, assume that \textbf{[Hid]} holds, so that $\sigma_K(G(\fbf^*)) > 0$, then
\begin{align*}
\|g^{\pi^* + \bar{p}, \Qbf^* + \bar{q}, \fbf^* + \bar{A} \fbf^*} - g^{\pi^*, \Qbf^*, \fbf^*}\|_2^2
	&\geq \sigma_{(K-1)(2K+1)}(M) (\| p \|_2^2 + \| q \|_F^2 + \| A \|_F^2) \\
		&\quad + o(\| p \|_2^2 + \| q \|_F^2 + \| A \|_F^2) \\
	&\geq \frac{\sigma_{(K-1)(2K+1)}(M)}{1 \wedge K \| G(\fbf^*) \|_\infty} \left(\| \bar{p} \|_2^2 + \| \bar{q} \|_F^2 +  \sum_{k \in \Xcal} \| (\bar{A}\fbf^*)_k \|_F^2 \right) \\
		&\quad + o \left( \frac{1}{1 \wedge \sigma_K(G(\fbf^*))} \left( \| \bar{p} \|_2^2 + \| \bar{q} \|_F^2 + \sum_{k \in \Xcal} \| (\bar{A}\fbf^*)_k \|_F^2 \right) \right)
\end{align*}
and finally
\begin{multline}
\label{eq_minoration_quadratique_avec_H}
\|g^{\pi^* + \bar{p}, \Qbf^* + \bar{q}, \fbf^* + \bar{A} \fbf^*} - g^{\pi^*, \Qbf^*, \fbf^*}\|_2^2 \\
	\geq c_2(\pi^*, \Qbf^*, \fbf^*) \left(\| \bar{p} \|_2^2 + \| \bar{q} \|_F^2 +  \sum_{k \in \Xcal} \| (\bar{A}\fbf^*)_k \|_F^2 \right) \\
		+ o \left( \| \bar{p} \|_2^2 + \| \bar{q} \|_F^2 + \sum_{k \in \Xcal} \| (\bar{A}\fbf^*)_k \|_F^2 \right)
\end{multline}
where
\begin{equation*}
c_2(\pi^*, \Qbf^*, \fbf^*) = \frac{H}{(1 \wedge K \| G(\fbf^*) \|_\infty) (3 K^3 (1 \vee \| G(\fbf^*) \|_\infty^4))^{K^2-K/2}}
\end{equation*}
is positive as soon as \textbf{[Hid]} and \textbf{[Hdet]} hold.

\subsection{Proof of Theorem \ref{th_minoration_quad}}
\label{sec_proof_minoration_quad}

Let
\begin{equation*}
N_\fbf(p, q, \hbf) = \|g^{\pi^* + p, \Qbf^* + q, \fbf + \hbf} - g^{\pi^*, \Qbf^*, \fbf}\|_2^2
\end{equation*}
and
\begin{equation*}
\| (p,q,\hbf) \|_\fbf^2 = \dperm((\pi^* + p, \Qbf^* + q, \fbf + \hbf),(\pi^*, \Qbf^*, \fbf))^2.
\end{equation*}
We want to show that there exists a constant $c^* > 0$ such that there exists a neighborhood $\Vcal$ of $\fbf^*$ such that if one writes
\begin{equation*}
c_\fbf := \inf_{p \in (\Delta - \Delta), \; q \in (\Qcal - \Qcal), \; \hbf \in (\Fcal - \Fcal)^K} \frac{N_\fbf(p, q, \hbf)}{\| (p,q,\hbf) \|_\fbf^2 }
\end{equation*}
then $\inf_{\fbf \in \Vcal} c_\fbf \geq c^*$.

The proof follows the structure of the proof of Theorem 6 of \cite{dCGL15}. It consists of three steps: the first one controls the component of $\hbf$ that is orthogonal to $\fbf$. This makes it possible to restrict $\hbf$ to the finite-dimensional space spanned by $\fbf$ in the two other parts. The second step controls the case when $\hbf$ is small, so that the behaviour of $N_\fbf$ is given by its quadratic form, and the last step controls the case where $\hbf$ is far from zero.

\subsubsection{The orthogonal part.}
Let $\ubf$ be the orthogonal projection of $\hbf$ on $\Span(\fbf)$. Then
\begin{equation*}
N_\fbf(p, q, \hbf) = N_\fbf(p, q, \ubf) + M_\fbf(p, q, \ubf, \hbf - \ubf)
\end{equation*}
where
\begin{multline*}
M_\fbf(p, q, \ubf, \abf)
	= \sum_{i_1, j_1, k_1} \sum_{i_2, j_2, k_2}
		(\pi^*+p)(i_1) (\Qbf^*+q)(i_1, j_1) (\Qbf^*+q)(j_1, k_1) \\
		(\pi^*+p)(i_2) (\Qbf^*+q)(i_2, j_2) (\Qbf^*+q)(j_2, k_2) \\
		\Big(
			\langle a_{i_1}, a_{i_2} \rangle \langle (f+u)_{j_1}, (f+u)_{j_2} \rangle \langle (f+u)_{k_1}, (f+u)_{k_2} \rangle \\
			+ \langle (f+u)_{i_1}, (f+u)_{i_2} \rangle \langle a_{j_1}, a_{j_2} \rangle \langle (f+u)_{k_1}, (f+u)_{k_2} \rangle \\
			+ \langle (f+u)_{i_1}, (f+u)_{i_2} \rangle \langle (f+u)_{j_1}, (f+u)_{j_2} \rangle \langle a_{k_1}, a_{k_2} \rangle \\
			+ \langle a_{i_1}, a_{i_2} \rangle \langle a_{j_1}, a_{j_2} \rangle \langle (f+u)_{k_1}, (f+u)_{k_2} \rangle \\
			+ \langle a_{i_1}, a_{i_2} \rangle \langle (f+u)_{j_1}, (f+u)_{j_2} \rangle \langle a_{k_1}, a_{k_2} \rangle \\
			+ \langle (f+u)_{i_1}, (f+u)_{i_2} \rangle \langle a_{j_1}, a_{j_2} \rangle \langle a_{k_1}, a_{k_2} \rangle \\
			+ \langle a_{i_1}, a_{i_2} \rangle \langle a_{j_1}, a_{j_2} \rangle \langle a_{k_1}, a_{k_2} \rangle
		\Big).
\end{multline*}
Let us write $\Pi'$ the matrix whose diagonal terms are the elements of $\pi^* + p$ and $\Qbf'$ the matrix $\Qbf^*+q$, then $M_\fbf$ can be written as
\begin{multline*}
M_\fbf(p, q, \ubf, \abf)
	= \sum_{i,j}
		\Big(
			((\Pi' \Qbf')^\top G(\abf) \Pi' \Qbf')_{i,j} G(\fbf+\ubf)_{i,j} (\Qbf'^\top G(\fbf+\ubf) \Qbf')_{i,j} \\
			+ ((\Pi' \Qbf')^\top G(\fbf+\ubf) \Pi' \Qbf')_{i,j} G(\abf)_{i,j}  (\Qbf'^\top G(\fbf+\ubf) \Qbf')_{i,j} \\
			+ ((\Pi' \Qbf')^\top G(\fbf+\ubf) \Pi' \Qbf')_{i,j} G(\fbf+\ubf)_{i,j} (\Qbf'^\top G(\abf) \Qbf')_{i,j} \\
			+ ((\Pi' \Qbf')^\top G(\abf) \Pi' \Qbf')_{i,j} G(\abf)_{i,j} (\Qbf'^\top G(\fbf+\ubf) \Qbf')_{i,j} \\
			+ ((\Pi' \Qbf')^\top G(\abf) \Pi' \Qbf')_{i,j} G(\fbf+\ubf)_{i,j} (\Qbf'^\top G(\abf) \Qbf')_{i,j} \\
			+ ((\Pi' \Qbf')^\top G(\fbf+\ubf) \Pi' \Qbf')_{i,j} G(\abf)_{i,j} (\Qbf'^\top G(\abf) \Qbf')_{i,j} \\
			+ ((\Pi' \Qbf')^\top G(\abf) \Pi' \Qbf')_{i,j} G(\abf)_{i,j} (\Qbf'^\top G(\abf) \Qbf')_{i,j}
		\Big).
\end{multline*}
By the Schur product theorem, these terms are nonnegative since they correspond to Hadamard products of three Gram matrices which are nonnegative. Thus, one can lower bound $M_\fbf(p, q, \ubf, \abf)$ by the second term of the sum, which leads to
\begin{align*}
M_\fbf(p, q, \ubf, \abf)
	\geq \sum_{i,j=1}^K
		((\Pi' \Qbf')^\top G(\fbf+\ubf) \Pi' \Qbf')_{i,j} (\Qbf'^\top G(\fbf+\ubf) \Qbf')_{i,j} \langle a_i, a_j \rangle 
\end{align*}
Assume \textbf{[Hid]} holds for the parameters $(\pi^*+p, \Qbf^*+q, \fbf+\ubf)$, then the matrices $(\Pi' \Qbf')^\top G(\fbf+\ubf) \Pi' \Qbf'$ and $\Qbf'^\top G(\fbf+\ubf) \Qbf'$ are positive symmetric with respective lowest eigenvalue lower bounded by $(\inf_k(\pi^*_k+p_k) \sigma_K(\Qbf^* + q))^2 \sigma_K(G(\fbf+\ubf))$ and $\sigma_K(\Qbf^* + q)^2 \sigma_K(G(\fbf+\ubf))$. Therefore, their Hadamard product is positive, and one has 
\begin{equation*}
(((\Pi' \Qbf')^\top G(\fbf+\ubf) \Pi' \Qbf')_{i,j} (\Qbf'^\top G(\fbf+\ubf) \Qbf')_{i,j})_{i,j}
	= (D \Ubf)^\top (D \Ubf)
\end{equation*}
with $\Ubf$ an orthogonal matrix and $D$ a diagonal matrix with positive diagonal coefficients. Moreover, the Schur product theorem implies that $\sigma_K(D)^2 \geq (\inf_k(\pi^*_k+p_k))^2 \sigma_K(\Qbf^* + q)^4 \sigma_K(G(\fbf+\ubf))^2$. Then
\begin{align*}
M_\fbf(p, q, \ubf, \abf)
	\geq& \sum_{i,j=1}^K ((D \Ubf)^\top (D \Ubf))_{i,j} \langle a_i, a_j \rangle \\
	=& \sum_{j=1}^K \| D \Ubf \abf \|_2^2 \\
	\geq& \sigma_K(D)^2 \| \Ubf \abf \|_2^2 \\
	\geq& (\inf_k(\pi^*_k+p_k))^2 \sigma_K(\Qbf^* + q)^4 \sigma_K(G(\fbf+\ubf))^2 \| \abf \|_2^2.
\end{align*}
Finally, let $c_1(\pi^* + p, \Qbf^*+q, \fbf+\ubf) = (\inf_k(\pi^*_k+p_k))^2 \sigma_K(\Qbf^* + q)^4 \sigma_K(G(\fbf+\ubf))^2$. The application $(p, \pi^*, q, \Qbf^*, \ubf, \fbf) \mapsto c_1(\pi^* + p, \Qbf^*+q, \fbf+\ubf)$ is continuous and nonnegative, it is positive when \textbf{[Hid]} holds for the parameters $(\pi^* + p, \Qbf^*+q, \fbf+\ubf)$, and one has
\begin{equation*}
M_\fbf(p, q, \ubf, \abf)
	\geq c_1(\pi^* + p, \Qbf^*+q, \fbf+\ubf) \| \abf \|_2^2.
\end{equation*}

We will now control the term $N_\fbf(p,q,\ubf)$. Two cases appear: when $(\pi^* + p, \Qbf^* + q ,\fbf + \ubf)$ is close to $(\pi^*, \Qbf^*, \fbf^*)$ in some sense and when it is not. The first case will be solved using the nondegeneracy of the quadratic form ensured by \textbf{[Hdet]}. The second case will be solved using the identifiability of the HMM.

\subsubsection{In the neighborhood of \texorpdfstring{$\fbf^*$}{f*}.}

The Taylor expansion of
\begin{align*}
(p,q,\ubf) \in (\Delta - \Delta) \times (\Qcal - \Qcal) \times ((\Fcal - \Fcal) \cap \Span(\fbf))^K \mapsto N_\fbf(p, q, \ubf)
\end{align*}
around $(0,0,0)$ leads to a nonnegative quadratic form and no linear part. \textbf{[Hdet]}, \textbf{[Hid]} and equation (\ref{eq_minoration_quadratique_avec_H}) ensure that this form is positive for $\fbf = \fbf^*$. Let $c_2(\Qbf^*, \pi^*, \fbf)$ be as defined in Section \ref{sec_lien_H_forme_quadratique}, then $\fbf \mapsto c_2(\Qbf^*, \pi^*, \fbf)$ is continuous and it is positive in the neighborhood of $\fbf^*$. Moreover, there exists a positive constant $\eta$ depending on $\|  G(\fbf) \|_\infty$ such that for all $(p,q,\ubf)$ such that $\| (p, q, \ubf) \|_\fbf \leq 1$, one has
\begin{align*}
N_\fbf(p, q, \ubf)
	\geq c_2(\Qbf^*, \pi^*, \fbf) \| (p, q, \ubf) \|_\fbf^2 - \eta \| (p, q, \ubf) \|_\fbf^3.
\end{align*}
For instance, $\eta = 4000 K^6 \| G(\fbf) \|_\infty^3$ works: the terms of order 2 or more in the Taylor expansion of $N_\fbf$ are the scalar product of sums of terms of the form $\sum_{i,j,k \in \Xcal} \pi^*(i) \Qbf^*(i,j) \Qbf^*(j,k) f_i \otimes f_j \otimes f_k$ where zero to three of the $f$ may be replaced by $u$, zero to two of the $\Qbf^*$ by $q$ and $\pi^*$ may be replaced by $p$ and at least one of them is replaced. There are 63 possibilities, which leads to a sum of $(63 K^3)^2$ terms, each of which can be bounded by $\| G(\fbf) \|_\infty^3 (\max \{ p(i), q(i,j), \|u_i\|_2 \; | \; i,j \in \Xcal \})^r$ where $r$ is the number of replaced terms. By taking the right permutation of states, the max can be bounded by $\| (p, q, \ubf) \|_\fbf$, hence the result.

Then, using $\| (p,q,\hbf) \|_\fbf^2 = \| (p,q,\ubf) \|_\fbf^2 + \| \abf \|_2^2$ leads to
\begin{align*}
\frac{N_\fbf(p, q, \hbf)}{\| (p, q, \hbf) \|_\fbf^2}
	\geq& c_1(\Qbf^*+q, \pi^*+p, \fbf+\ubf) \frac{\| \abf \|_2^2}{\| (p,q,\ubf) \|_\fbf^2 + \| \abf \|_2^2} \\
	&+ c_2(\Qbf^*, \pi^*, \fbf) \frac{\| (p, q, \ubf) \|_\fbf^2}{\| (p,q,\ubf) \|_\fbf^2 + \| \abf \|_2^2}
	- \eta \frac{\| (p, q, \ubf) \|_\fbf^2}{(\| (p,q,\ubf) \|_\fbf^2 + \| \abf \|_2^2)^{1/2}} \\
	\geq& c_1(\Qbf^*+q, \pi^*+p, \fbf+\ubf) \frac{\| \abf \|_2^2}{\| (p,q,\ubf) \|_\fbf^2 + \| \abf \|_2^2} \\
	&+ c_2(\Qbf^*, \pi^*, \fbf) \frac{\| (p, q, \ubf) \|_\fbf^2}{\| (p,q,\ubf) \|_\fbf^2 + \| \abf \|_2^2}
	- \eta \sqrt{\| (p, q, \ubf) \|_\fbf^2}
\end{align*}
Let $c_0 = \min(c_1/2, c_2)/2$, then $c_0$ is continuous and there exists a continuous function $(\pi^*, \Qbf^*, \fbf) \mapsto \epsilon(\pi^*, \Qbf^*, \fbf)$ which is positive as soon as \textbf{[Hid]} and \textbf{[Hdet]} hold for $(\pi^*, \Qbf^*, \fbf)$ and such that
\begin{equation*}
\| (p, q, \ubf) \|_\fbf \leq \epsilon(\pi^*, \Qbf^*, \fbf)
	\; \Rightarrow \; \frac{N_\fbf(p, q, \hbf)}{\| (p, q, \hbf) \|_\fbf^2} \geq c_0(\Qbf^*, \pi^*, \fbf).
\end{equation*}
Thus, there exists positive constants $\epsilon_0$ and $c_\text{near}$ depending on $\Qbf^*$, $\pi^*$ and $\fbf^*$ such that 
\begin{multline*}
\forall (p,q,\hbf,\fbf) \in (\Delta-\Delta) \times (\Qcal-\Qcal) \times (\Fcal - \Fcal)^K \times \Fcal^K \\
	\text{ s.t. } \| (p, q, \ubf) \|_\fbf \leq \epsilon_0 \text{ and } \sum_{k \in \Xcal} \| f_k - f^*_k \|_2^2 \leq \epsilon_0^2, \qquad
	\frac{N_\fbf(p, q, \hbf)}{\| (p, q, \hbf) \|_\fbf^2} \geq c_\text{near}.
\end{multline*}

\subsubsection{Far from \texorpdfstring{$\fbf^*$}{f*}.}

\begin{lemma}
\label{lemma_N_is_unifC0}
The application
\begin{align*}
(p,q,\ubf,\fbf) \in (\Delta-\Delta) \times (\Qcal-\Qcal) \times (\Fcal - \Fcal)^K \times \Fcal^K
	\longmapsto N_\fbf(p,q,\ubf)
\end{align*}
restricted to the set of $(p,q,\ubf,\fbf)$ such that $\ubf \in \Span(\fbf)^K$ is uniformly continuous for the norm $\| \cdot \|_\text{tot}$ defined by
\begin{align*}
\| (p,q,\ubf,\fbf) \|_\text{tot}^2
	:= \| p \|_2^2 + \| q \|_F^2 + \sum_{k \in \Xcal} \left( \| u_k \|_2^2 + \| f_k \|_2^2 \right).
\end{align*}
\end{lemma}
Thus, by compactness of $(\Delta-\Delta) \times (\Qcal-\Qcal) \times ((\Fcal - \Fcal) \cap \Span(\fbf))^K$, the application
\begin{align*}
c_\text{far}: \fbf \longmapsto \inf_{(p,q,\ubf) \in (\Delta-\Delta) \times (\Qcal-\Qcal) \times ((\Fcal - \Fcal) \cap \Span(\fbf))^K \text{ s.t. } \| (p, q, \ubf) \|_\fbf > \epsilon_0} N_\fbf(p, q, \ubf)
\end{align*}
is continuous. Let us now prove that $c_\text{far}(\fbf^*) > 0$.

Let $(p_n, q_n, \ubf_n)_n \in ((\Delta-\Delta) \times (\Qcal-\Qcal) \times ((\Fcal - \Fcal) \cap \Span(\fbf^*))^K)^\Nbb$ be a sequence such that $\| (p_n, q_n, \ubf_n) \|_{\fbf^*} > \epsilon_0$ for all $n$ and
\begin{align*}
c_\text{far}(\fbf^*)
	= \lim_n N_{\fbf^*}(p_n, q_n, \ubf_n).
\end{align*}

By compactness, this sequences converges towards a limit $(p,q,\ubf)$ up to taking a subsequence. Necessarily $\| (p, q, \ubf) \|_{\fbf^*} \geq \epsilon_0$. Since \textbf{[Hid]} holds, Theorem \ref{th_identifiabilite} shows that $N_{\fbf^*}(p, q, \ubf) > 0$, which implies $c_\text{far}(\fbf^*) > 0$ by continuity of $N_{\fbf^*}$. Note that $c_\text{far}(\fbf^*)$ may depend on $\Fcal$ in addition to the parameters $\pi^*$, $\Qbf^*$ and $\fbf^*$.

Thus, by continuity, there exists $\epsilon_1 > 0$ such that for all $\fbf \in \Fcal^K$ such that $\sum_{k \in \Xcal} \| f_k - f^*_k \|_2^2 \leq \epsilon_1^2$, $c_\text{far}(\fbf) \geq c_\text{far}(\fbf^*) / 2$.

Finally, \textbf{[HF]} implies that there exists a constant $\Ccal$ depending only on $\cquad$ such that $\| (p,q,\ubf) \|_\fbf^2 \leq \| (p,q,\hbf) \|_\fbf^2 \leq \Ccal$ for all $(p,q,\hbf) \in (\Delta-\Delta) \times (\Qcal-\Qcal) \times (\Fcal - \Fcal)^K$. Therefore,
\begin{align*}
\forall (p,q,\hbf,\fbf) \in (\Delta-\Delta) \times (\Qcal-\Qcal) \times (\Fcal - \Fcal)^K \times \Fcal^K \\
	\text{ s.t. } \| (p, q, \ubf) \|_\fbf \geq \epsilon_0 \text{ and } \sum_{k \in \Xcal} \| f_k - f^*_k \|_2^2 \leq \epsilon_1^2, \qquad
\frac{N_{\fbf}(p, q, \hbf)}{\| (p,q,\hbf) \|_\fbf^2}
	&\geq \frac{N_{\fbf}(p, q, \ubf)}{\Ccal} \\
	&\geq \frac{c_\text{far}(\fbf^*)}{2 \Ccal}.
\end{align*}
The theorem follows by taking $c^*(\pi^*, \Qbf^*, \fbf^*, \Fcal) = \displaystyle \min\left(\frac{c_\text{far}(\fbf^*)}{2 \Ccal}, c_\text{near} \right)$ and the neighborhood containing all $\fbf \in \Fcal^K$ such that $\sum_{k \in \Xcal} \| f_k - f^*_k \|_2^2 \leq \min(\epsilon_0, \epsilon_1)^2$. Moreover, $(\pi, \Qbf, \fbf) \longmapsto c^*(\pi, \Qbf, \fbf, \Fcal)$ is lower bounded by this value in a neighborhood of $(\pi^*, \Qbf^*, \fbf^*)$, so that it can be assumed to be lower semicontinuous.

Note that the dependency of $c^*$ on $\Fcal$ appears during this last step and is made non explicit because of the compactness assumption.

\subsubsection{Proof of Lemma \ref{lemma_N_is_unifC0}}

\begin{align*}
\Big| N_\fbf(p,q,\ubf) -& N_{\fbf'}(p',q',\ubf') \Big|\\
	=& \left| \| g^{\pi^*+p, \Qbf^*+q, \fbf+\ubf} - g^{\pi^*, \Qbf^*, \fbf} \|_2^2
		- \| g^{\pi^*+p', \Qbf^*+q', \fbf'+\ubf'} - g^{\pi^*, \Qbf^*, \fbf'} \|_2^2 \right| \\
	\leq& 2 \| g^{\pi^*+p, \Qbf^*+q, \fbf+\ubf} - g^{\pi^*+p', \Qbf^*+q', \fbf'+\ubf'} \|_2^2
		+ 2 \| g^{\pi^*, \Qbf^*, \fbf} - g^{\pi^*, \Qbf^*, \fbf'} \|_2^2 \\
		&+ 2 \left| \left\langle g^{\pi^*+p', \Qbf^*+q', \fbf'+\ubf'} - g^{\pi^*, \Qbf^*, \fbf'} , g^{\pi^*+p, \Qbf^*+q, \fbf+\ubf} - g^{\pi^*+p', \Qbf^*+q', \fbf'+\ubf'} \right\rangle \right| \\
		&+ 2 \left| \left\langle g^{\pi^*+p', \Qbf^*+q', \fbf'+\ubf'} - g^{\pi^*, \Qbf^*, \fbf'} , g^{\pi^*, \Qbf^*, \fbf} - g^{\pi^*, \Qbf^*, \fbf'} \right\rangle \right|
\end{align*}
Then, using the fact that $\| g^{\pi, \Qbf, \fbf} - g^{\pi', \Qbf', \fbf'} \|_2 \leq \sqrt{3K} \cquad^3 \| (\pi - \pi', \Qbf - \Qbf', \fbf - \fbf', 0) \|_\text{tot}$ (see Lemma \ref{lemma_upperbound_diff_g}), that $\| g^{\pi, \Qbf, \fbf} \|_2 \leq \cquad^3$ (see for instance Lemma 29 of \cite{lehericy2016order}) and the Cauchy-Schwarz inequality,
\begin{align*}
\left| N_\fbf(p,q,\ubf) - N_{\fbf'}(p',q',\ubf') \right|
	\leq& 6K \cquad^6 \| (p-p', q-q', \fbf+\ubf-\fbf'-\ubf', 0) \|_\text{tot}^2 \\
		&+ 6K \cquad^6\| (0, 0, 0, \fbf-\fbf') \|_\text{tot}^2 \\
		&+ 4\sqrt{3K} \cquad^6 \| (p-p', q-q', \fbf+\ubf-\fbf'-\ubf', 0) \|_\text{tot} \\		
		&+ 4\sqrt{3K} \cquad^6 \| (0, 0, 0, \fbf-\fbf') \|_\text{tot}^2 \\
	\leq& 24 K \cquad^6 \Big( \| (p-p', q-q', \ubf-\ubf', \fbf-\fbf') \|_\text{tot}^2 \\
		&+ \| (p-p', q-q', \ubf-\ubf', \fbf-\fbf') \|_\text{tot} \Big),
\end{align*}
which proves the uniform continuity of the application.

%\vskip 0.2in
\bibliography{these}

\end{document}